\RequirePackage{fix-cm}
\documentclass[smallextended]{svjour3}       % onecolumn (second format)
\smartqed  % flush right qed marks, e.g. at end of proof
\usepackage{graphicx}
\usepackage{amsmath}
\usepackage{amssymb}
\usepackage{arydshln}
\usepackage{enumerate}
\usepackage{enumitem}
\usepackage{mathtools}
\usepackage{hyperref}
\usepackage{bm}
\usepackage{float}
\usepackage{algorithm}
\usepackage[noend]{algpseudocode}
\usepackage{soul,color}
\usepackage{subcaption}
\usepackage{multirow}
\usepackage[title]{appendix}
\usepackage{tikz}
\usetikzlibrary{tikzmark}
\DeclareMathOperator*{\argmin}{\arg\!\min}

\makeatletter
\renewcommand{\subsection}{%
  \@startsection{subsection}
    {2}
    {\z@}
    {-21dd plus-8pt minus-4pt}
    {10.5dd}
    {\normalsize\bfseries\boldmath}%
}
\makeatother
\newcommand{\vertiii}[1]{{\left\vert\kern-0.25ex\left\vert\kern-0.25ex\left\vert #1 \right\vert\kern-0.25ex\right\vert\kern-0.25ex\right\vert}}
\allowdisplaybreaks
\begin{document}
\title{A Distributed Algorithm for High-Dimension Convex Quadratically Constrained Quadratic Programs}
\titlerunning{A Distributed Algorithm for High-Dimension Convex QCQPs}        % if too long for running head
\author{Run~Chen$^{1}$ \and Andrew~L.~Liu$^{1}$ %etc.
}
\authorrunning{R. Chen \and A.~L.~Liu} % if too long for running head
\institute{Run~Chen \at
              \email{BigRunTheory@gmail.com} \\ 
           \and
           Andrew L.~Liu \at
              \email{andrewliu@purdue.edu} \\
           \and
           $^{1}$ \at
              School of Industrial Engineering, Purdue University, West Lafayette, IN 47906, USA
}
\date{Received: date / Accepted: date}
\maketitle
\begin{abstract}
%Insert your abstract here. Include keywords, PACS and mathematical subject classification numbers as needed. (Please provide an abstract of 150 to 250 words. The abstract should not contain any undefined abbreviations or unspecified references.)
We propose a Jacobi-style distributed algorithm to solve convex, quadratically constrained quadratic programs (QCQPs), which arise from a broad range of applications. While small to medium-sized convex QCQPs can be solved efficiently by interior-point algorithms, high-dimension problems pose significant challenges to traditional algorithms that are mainly designed to be implemented on a single computing unit. 
The exploding volume of data (and hence, the problem size), however, may overwhelm any such units. In this paper, we propose a distributed algorithm for general, non-separable, high-dimension convex QCQPs, using a novel idea of predictor-corrector primal-dual update with an adaptive step size. The algorithm enables distributed storage of data as well as parallel, distributed computing. We establish the conditions for the proposed algorithm to converge to a global optimum, and implement our algorithm on a computer cluster with multiple nodes using Message Passing Interface (MPI). The numerical experiments are conducted on data sets of various scales from different applications, and the results show that our algorithm exhibits favorable scalability for solving high-dimension problems. %(163 words)
%(Please provide 4 to 6 keywords which can be used for indexing purposes.)
\end{abstract}
%
%
%
%
%%%%%%%%%%%%%%%%%%%%%%%%%%%%%%%%%%%%%%%%%%%%%%%%%%%%%%%%%%%%%%%%%%%
\section{Introduction}
%Problem Formulation
In this paper, we consider the following constrained optimization problem:
\begin{equation}\label{eq: Standard QCQP Problem Form}
\begin{aligned}
\underset{\mathbf{x} \in \mathbb{R}^{n_1}}{\text{minimize}} \quad &\frac{1}{2} \mathbf{x}^{T} P_0 \mathbf{x} + \mathbf{q}_0^T \mathbf{x} + r_0 \\
\text{subject to} \quad &\frac{1}{2} \mathbf{x}^{T} P_i \mathbf{x} + \mathbf{q}_i^T \mathbf{x} + r_i \leq 0, \quad i = 1, \dots, m_1,
\end{aligned}
\end{equation}
where $P_i \in \mathbb{R}^{n_1 \times n_1}$, $\mathbf{q}_i \in \mathbb{R}^{n_1}$, and $r_i \in \mathbb{R}$ for $i = 0, 1, \dots, m_1$ are all given. Such a problem is referred to as a quadratically constrained quadratic program (QCQP). (Note that  linear constraints are included with $P_i = \mathbf{0}$, a matrix of all 0's, for some $i$.) If additionally, $P_0, P_1, \dots, P_{m_1}$ are all  positive semidefinite (PSD) matrices, then the problem is convex.
%Applications
Convex QCQPs arise from a wide range of application areas, including multiple kernel learning \cite{lanckriet2004learning}, ranking recommendations \cite{chatterjee2018constrained}, signal processing \cite{SignalProcessing_QCQP}, radar applications \cite{rabaste2015mismatched}, computer vision \cite{ComputerVision_QCQP}, and electric power system operation \cite{Low_QCQP}, to name a few. Small to medium-sized convex QCQPs can be solved efficiently by the well-established interior-point method (IPM) \cite{nesterov1994interior}, which has polynomial running time for solving convex optimization problems. However, in order to write out the barrier function in the IPM for the feasible domain of a QCQP, decomposition of matrices $P_i = F_i^T F_i$ for $i = 1, \dots, m_1$ is usually required \cite{nemirovski2004interior}, which may not be readily available through the input data. For example, in kernel-based learning applications, each quadratic constraint comprises a kernel matrix, whose components are directly defined by a kernel function: $K_{jj^{\prime}} = k(\mathbf{x}_j, \mathbf{x}_{j^{\prime}})$. The operations to obtain a matrix decomposition, such as through Cholesky decomposition,  typically have  computational complexity of $O(n^3)$, which could become very costly as the size of the matrices grows. When the dimension of the QCQPs increases dramatically due to huge amount of data, or when the data just cannot be all stored in a central location, a centralized algorithm, such as the IPM, may no longer be applicable. This directly motivates the proposed algorithm in this paper, which not only does not require any matrix decomposition , but also facilitates distributed storage of data to achieve memory efficiency and enables parallel computing even for QCQPs of non-separable constraints.
%
%SOCP & SDP
\par In addition to being a typical optimization problem, a convex QCQP is also a special instance of a second-order cone program (SOCP), which is in turn a special form of semi-definite program (SDP) \cite{lobo1998applications}. When using commercial solvers, such as CPLEX, to solve a convex QCQP, it is usually transformed into an SOCP through preprocessing \cite{manual1987ibm}, and then a barrier-method-based optimizer is applied. 
%Small to medium-sized SOCPs can be solved efficiently using the IPM \cite{nesterov1994interior}. 
To solve large-scale conic programs, \cite{o2016conic} applies an operator splitting method (such as the well-known alternating direction method of multipliers, or ADMM) to the homogeneous self-dual embedding, which is an equivalent convex feasibility problem involving finding a nonzero point in the intersection of a subspace and a cone. There are also ADMM-based distributed algorithms for solving large-scale SDPs proposed in \cite{kalbat2015fast,pakazad2018distributed}; but they can only be applied to a class of decomposable SDPs with special graph representations (chordal graphs, for example). To either translate a convex QCQP to a standard SOCP or use the Schur Complement to rewrite each quadratic inequality as a linear matrix inequality (LMI) and hence translate a convex QCQP to an SDP , however, calls for matrix decomposition: $P_i = F_i^T F_i$ for $i = 1, \dots, m_1$. As mentioned before, such operations can be very expensive for large-scale matrices. 
%Other Distributed Algorithms
There is another ADMM-based distributed algorithm that decomposes a general QCQP with $m$ constraints  into $m$ single-constrained QCQPs using a reformulated consensus optimization form \cite{huang2016consensus}. However, even the size of the single-constrained QCQP can be very large in many applications, which may still need further decomposition, making the overall algorithm's efficiency in doubt. There is also a recent approach to transform quadratic constraints into linear constraints by sampling techniques and then to apply ADMM-based algorithms to solve the resulting high-dimension quadratic programs (QPs) \cite{basu2017large}. This approach is studied only for QCQPs with all matrices being positive definite (PD), and all the test problems shown in \cite{basu2017large} are of a single constraint. How would the sampling approach perform with PSD matrices in the constraints or with multiple quadratic constraints is unknown.
%
%
%Our Contribution
\par To overcome the above-mentioned limitations of the existing algorithms, we propose a novel first-order distributed algorithm, which decomposes a convex QCQP by a method inspired by the idea of the predictor corrector proximal multiplier method (PCPM) \cite{chen1994proximal}. The advantages of our algorithm include the following: (i) non-separable, quadratic functions can become naturally separable after introducing the so-called predictor and corrector variables for both primal and dual variables, which greatly facilitates distributed computing (with Jacobi-style parallel updating, as opposed to Guass-Seidel style sequential updating); while ADMM-type algorithms cannot be directly applied to QCQPs without separable constraints; (ii)  both the primal/dual predictor variables and corrector variables can be updated component-wise, making the method well-suited for massively parallel computing, and each $n$-by-$n$ Hessian matrix can be stored column-wise in distributed computing units; (iii) no matrix decomposition or inversion is needed. 
\par Convergence of our algorithm to an optimal solution will be shown, along with various numerical results. We first test the algorithm on solving standard QCQPs with randomly generated data sets of different scales, and then apply it to solve high-dimension multiple kernel learning problems. Numerical experiments are conducted on a multi-node computer cluster through message passing interface (MPI), and multiple nodes are used to highlight the benefits of distributed implementation of our algorithm. Numerical results are compared with those obtained from the commercial solver CPLEX (version 12.8.0, using the barrier optimizer). The comparison will show that our algorithm can scale to very large problems at the cost of consuming more cheap iterations to reach a higher accuracy. With a modest accuracy, our algorithm exhibits favorable scalability for solving high-dimension QCQPs when CPLEX fails to provide a solution due to memory limit or other issues. 
%
%Outline
\par The remainder of the paper is organized as follows. In Section 2, we briefly summarize the original PCPM algorithm and highlight the novel idea in our proposed algorithm. Section 3 provides convergence analyses of the algorithm, followed by discussions on how to implement the algorithm in a distributed framework in Section 4. Numerical performance of various testing problems is reported in Section 5. Finally, we conclude with some discussions in Section 6.
%
%
%
%
%%%%%%%%%%%%%%%%%%%%%%%%%%%%%%%%%%%%%%%%%%%%%%%%%%%%%%%%%%%%%%%%%%%%%%%%%%%%%%%%
\section{Algorithm Description}
\subsection{PCPM Algorithm}
To present our distributed algorithm, we first briefly describe the original PCPM algorithm \cite{chen1994proximal} to make this paper self-contained. For this purpose, it suffices to consider a 2-block linearly constrained convex optimization problem:  
\begin{equation}\label{eq: 2_block CP Problem Form}
\begin{aligned}
\underset{\mathbf{x}_1 \in \mathbb{R}^{n_1},\ \mathbf{x}_2 \in \mathbb{R}^{n_2}}{\text{minimize}} \quad &f_1(\mathbf{x}_1) + f_2(\mathbf{x}_2) \\
\text{subject to} \quad &A_1 \mathbf{x}_1 + A_2 \mathbf{x}_2 = \mathbf{b}, \quad (\bm{\lambda})
\end{aligned}
\end{equation}
where $f_1: \mathbb{R}^{n_1} \to (-\infty, +\infty]$ and $f_2: \mathbb{R}^{n_2} \to (-\infty, +\infty]$ are closed proper convex functions, $A_1 \in \mathbb{R}^{m \times n_1}$ and $A_2 \in \mathbb{R}^{m \times n_2}$ are full row-rank matrices, $\mathbf{b} \in \mathbb{R}^m$ is a given vector, and $\bm{\lambda} \in \mathbb{R}^m$ is the corresponding  Lagrangian multiplier associated with the linear equality constraint. The classic Lagrangian function $\mathcal{L}: \mathbb{R}^{n_1} \times \mathbb{R}^{n_2} \times \mathbb{R}^m \to \mathbb{R}$ is defined as:
\begin{equation}\label{eq:L}
\mathcal{L}(\mathbf{x}_1, \mathbf{x}_2, \bm{\lambda}) = f_1(\mathbf{x}_1) + f_2(\mathbf{x}_2) + \bm{\lambda}^T (A_1 \mathbf{x}_1 + A_2 \mathbf{x}_2 - \mathbf{b}).  
\end{equation}
It is well-known that for a convex problem of the specific form in \eqref{eq: 2_block CP Problem Form} (where the linear constraint qualification automatically holds), finding an optimal solution is equivalent to finding a saddle point $(\mathbf{x}_1^*, \mathbf{x}_2^*, \bm{\lambda}^*)$ such that $\mathcal{L}(\mathbf{x}_1^*, \mathbf{x}_2^*, \bm{\lambda}) \leq \mathcal{L}(\mathbf{x}_1^*, \mathbf{x}_2^*, \bm{\lambda}^*) \leq \mathcal{L}(\mathbf{x}_1, \mathbf{x}_2, \bm{\lambda}^*)$. To find such a saddle point, a simple dual decomposition algorithm can be applied to $\mathcal{L}(\mathbf{x}_1, \mathbf{x}_2, \bm{\lambda})$. More specifically, at each iteration $k$, given a fixed Lagrangian multiplier $\bm{\lambda}^k$, the primal decision variables $(\mathbf{x}_1^{k+1}, \mathbf{x}_2^{k+1})$ can be  obtained, in parallel, by minimizing $\mathcal{L}(\mathbf{x}_1, \mathbf{x}_2, \bm{\lambda}^k)$. Then a dual update $\bm{\lambda}^{k+1} = \bm{\lambda}^k + \rho (A_1 \mathbf{x}_1^{k+1} + A_2 \mathbf{x}_2^{k+1} - \mathbf{}b)$ is performed. 
\par While the above algorithmic idea is simple, it is well-known that convergence cannot be established without more restrictive assumptions, such as strict convexity of $f_1$ and $f_2$ (e.g., Theorem~26.3 in \cite{rockafellar2015convex}). One approach to overcome such difficulties is the proximal point algorithm, which obtains $(\mathbf{x}_1^{k+1},\mathbf{x}_2^{k+1})$  by minimizing the proximal augmented Lagrangian function defined as $\mathcal{L}_{\rho}(\mathbf{x}_1, \mathbf{x}_2, \bm{\lambda}^k) := \mathcal{L}(\mathbf{x}_1, \mathbf{x}_2, \bm{\lambda}^k) + \frac{\rho}{2} \lVert A_1 \mathbf{x}_1 + A_2 \mathbf{x}_2 - \mathbf{b} \rVert_2^2 + \frac{1}{2 \rho} \lVert \mathbf{x}_1 - \mathbf{x}_1^k \rVert_2^2 + \frac{1}{2 \rho} \lVert \mathbf{x}_2 - \mathbf{x}_2^k \rVert_2^2$. The parameter $\rho$ is given, which determines the step-size for updating both primal and dual variables in each iteration, and plays a key role in the convergence of the overall algorithm.  The primal minimization step now becomes (with the Lagrangian function $\mathcal{L}(\mathbf{x}_1, \mathbf{x}_2, \bm{\lambda}^k)$ explicitly written out in the form of Equation \eqref{eq:L}):
\begingroup
\begin{align}
(\mathbf{x}_1^{k+1}, \mathbf{x}_2^{k+1}) = \underset{\mathbf{x}_1 \in \mathbb{R}^{n_1}, \mathbf{x}_2 \in \mathbb{R}^{n_2}}{\argmin} &f_1(\mathbf{x}_1) + f_2(\mathbf{x}_2) + (\bm{\lambda}^k)^T (A_1 \mathbf{x}_1 + A_2 \mathbf{x}_2 - \mathbf{b}) \nonumber \\
+ &\frac{\rho}{2} \lVert A_1 \mathbf{x}_1 + A_2 \mathbf{x}_2 - \mathbf{b} \rVert_2^2 \nonumber \\
+ &\frac{1}{2 \rho} \lVert \mathbf{x}_1 - \mathbf{x}_1^k \rVert_2^2 + \frac{1}{2 \rho} \lVert \mathbf{x}_2 - \mathbf{x}_2^k \rVert_2^2. \label{eq: proximal point primal minimization}
\end{align}
\endgroup
With \eqref{eq: proximal point primal minimization}, however, $\mathbf{x}_1^{k+1}$ and $\mathbf{x}_2^{k+1}$ can no longer be obtained in parallel due to the augmented term $\lVert A_1 \mathbf{x}_1 + A_2 \mathbf{x}_2 - \mathbf{b} \rVert_2^2$. To overcome this difficulty, the PCPM algorithm introduces a predictor variable $\bm{\mu}^{k+1}$:
\begin{equation}\label{eq:PCPM_Predictor}
\bm{\mu}^{k+1} \coloneqq \bm{\lambda}^k + \rho (A_1 \mathbf{x}_1^k + A_2 \mathbf{x}_2^k - \mathbf{b}).
\end{equation}
Using the predictor variable, the optimization in \eqref{eq: proximal point primal minimization} can be approximated as:
\begingroup
\begin{align}
(\mathbf{x}_1^{k+1}, \mathbf{x}_2^{k+1}) = \underset{\mathbf{x}_1 \in \mathbb{R}^{n_1}, \mathbf{x}_2 \in \mathbb{R}^{n_2}}{\argmin} &f_1(\mathbf{x}_1) + f_2(\mathbf{x}_2) + (\bm{\mu}^{k+1})^T (A_1 \mathbf{x}_1 + A_2 \mathbf{x}_2 - \mathbf{b}) \nonumber \\
+ &\frac{1}{2 \rho} \lVert \mathbf{x}_1 - \mathbf{x}_1^k \rVert_2^2 + \frac{1}{2 \rho} \lVert \mathbf{x}_2 - \mathbf{x}_2^k \rVert_2^2, \label{eq: PCPM primal minimization}
\end{align}
\endgroup
which allows $\mathbf{x}_1^{k+1}$ and $\mathbf{x}_2^{k+1}$ to be obtained in parallel again. After solving \eqref{eq: PCPM primal minimization}, the PCPM algorithm updates the dual variable as follows: 
\begin{equation}\label{eq:PCPM_Corrector}
\bm{\lambda}^{k+1} = \bm{\lambda}^k + \rho (A_1 \mathbf{x}_1^{k+1} + A_2 \mathbf{x}_2^{k+1} - \mathbf{b}),
\end{equation} 
which is referred to as a corrector update. 
\subsection{A Distributed Algorithm for High-Dimension Convex QCQPs}
Now consider a convex QCQP problem in the following form:
\begin{equation}\label{eq: QCQP Problem Form}
\begin{aligned}
\underset{\mathbf{x} \in \mathbb{X},\ \mathbf{u} \in \mathbb{R}^{n_2}}{\text{minimize}} \quad &\frac{1}{2} \mathbf{x}^{T} P_0 \mathbf{x} + \mathbf{q}_0^T \mathbf{x} + \mathbf{c}_0^T \mathbf{u} + r_0 \\
\text{subject to} \quad &\frac{1}{2} \mathbf{x}^{T} P_i \mathbf{x} + \mathbf{q}_i^T \mathbf{x} + \mathbf{c}_i^T \mathbf{u} + r_i \leq 0, \quad i = 1, \dots, m_1, \qquad (\lambda_i) \\
&A \mathbf{x} + B \mathbf{u} = \mathbf{b}, \qquad (\bm{\gamma}) 
\end{aligned}
\end{equation}
where  $P_i \in \mathbb{R}^{n_1 \times n_1}$, $\mathbf{q}_i \in \mathbb{R}^{n_1}$, $r_i \in \mathbb{R}$ for $i = 0, 1, \dots, m_1$, $A \in \mathbb{R}^{m_2 \times n_1}$, $B \in \mathbb{R}^{m_2 \times n_2}$ and $\mathbf{b} \in \mathbb{R}^{m_2}$ are all given. Note that we introduce a new variable $\mathbf{u} \in \mathbb{R}^{n_2}$ to explicitly write out the linear-only terms $\mathbf{c}_i^T \mathbf{u}$ with coefficients $\mathbf{c}_i \in \mathbb{R}^{n_2}$ for $i = 0, 1, \dots, m_1$, and also write out a linear equality constraint $A \mathbf{x} + B \mathbf{u} = \mathbf{b}$ separately. 
While $\mathbb{X} = \prod_{j=1}^{n^{\prime}} \mathbb{X}_j \subset \mathbb{R}^{n_1}$ and each $\mathbb{X}_j \subset \mathbb{R}^{n_j}$ can be any closed and convex set in $\mathbb{R}^{n_j}$ with $\sum_{j=1}^{n^{\prime}} n_j = n_1$, we consider specifically the one-dimension box constraint here; that is $\mathbb{X}_j = \{x_j \in \mathbb{R}^1 | 0 \leq x_j \leq \bar{X}_j\}$ for $j = 1, \dots, n_1$.
\par The specific QCQP formulation in \eqref{eq: QCQP Problem Form} is not more general than the standard form \eqref{eq: Standard QCQP Problem Form}. The reason that we write out a QCQP in this specific form is to emphasize the fact that when dealing with QCQPs with linear constraints (including box constraints), our algorithm does not require the problem to be reformulated into the standard form in \eqref{eq: Standard QCQP Problem Form}. This can be convenient from implementation perspective, as several applications, including multiple kernel learning, naturally lead to a QCQP in the form of \eqref{eq: QCQP Problem Form}. 
\par To avoid technical difficulties, we make the blanket assumption throughout this paper that the Slater's constraint qualification (CQ) holds. Consequently, if an optimal solution exists of \eqref{eq: QCQP Problem Form}, then there always exists a corresponding Lagrangian multiplier $(\bm{\lambda}, \bm{\gamma}) = (\lambda_1 \cdots \lambda_{m_1}, \gamma_1 \cdots \gamma_{m_2})^T$. To apply the PCPM algorithm to the QCQP in \eqref{eq: QCQP Problem Form}, at each iteration $k$, with a given primal-dual pair $(\mathbf{x}^k, \mathbf{u}^k, \bm{\lambda}^k, \bm{\gamma}^k)$, we start with a dual predictor update:
\begingroup
\begin{align}
&\bullet \enspace \text{dual predictor}: \nonumber \\
&\mu_i^{k+1} = \Pi_{\mathbb{R}_+}\Big(\lambda_i^k + \rho \big[\frac{1}{2} (\mathbf{x}^k)^T P_i \mathbf{x}^k + \mathbf{q}_i^T \mathbf{x}^k + \mathbf{c}_i^T \mathbf{u}^k + r_i\big]\Big), \quad i = 1, \dots, m_1, \nonumber \\
&\nu_i^{k+1} = \gamma_i^k + \rho \big[A \mathbf{x}^k + B \mathbf{u}^k - \mathbf{b}\big]_i, \quad i = 1, \dots, m_2, \label{eq: dual predictor update}
\end{align}
\endgroup
where $\Pi_{\mathbb{Z}}(\mathbf{z})$ denotes the projection of a vector $\mathbf{z} \in \mathbb{R}^n$ onto a closed and convex set $\mathbb{Z} \subset \mathbb{R}^n$, and $\mathbb{R}_+$ refers to the set of all non-negative real numbers. 
\par After the dual predictor update step \eqref{eq: dual predictor update}, we update the primal variables $(\hspace*{-1pt}\mathbf{x}^{k+1}\hspace*{-1pt}, \hspace*{-2pt}\mathbf{u}^{k+1}\hspace*{-1pt})$ by minimizing the Lagrangian function $\mathcal{L}(\mathbf{x}, \mathbf{u}, \bm{\mu}^{k+1}, \bm{\nu}^{k+1})$ evaluated at the dual predictor variable $(\bm{\mu}^{k+1}, \bm{\nu}^{k+1})$, plus the proximal terms. The primal minimization step can be written as
\begin{subequations}
\begin{align}
\begin{split}
\mathbf{x}^{k+1} = \underset{\mathbf{x} \in \mathbb{X}}{\argmin} \enspace &\frac{1}{2} \mathbf{x}^T P_0 \mathbf{x} + \mathbf{q}_0^T \mathbf{x} + \sum_{i=1}^{m_1} \mu_i^{k+1} \big(\frac{1}{2} \mathbf{x}^T P_i \mathbf{x} + \mathbf{q}_i^T \mathbf{x}\big) \\
+ &(\bm{\nu}^{k+1})^T A \mathbf{x} + \frac{1}{2 \rho} \lVert \mathbf{x} - \mathbf{x}^k \rVert_2^2,
\end{split} \label{eq: x primal minimization} \\
\mathbf{u}^{k+1} = \underset{\mathbf{u} \in \mathbb{R}^{n_2}}{\argmin} \enspace &\mathbf{c}_0^T \mathbf{u} + \sum_{i=1}^{m_1} \mu_i^{k+1} \mathbf{c}_i^T \mathbf{u} + (\bm{\nu}^{k+1})^T B \mathbf{u} + \frac{1}{2 \rho} \lVert \mathbf{u} - \mathbf{u}^k \rVert_2^2. \label{eq: u primal minimization}
\end{align}
\end{subequations}
Introducing the dual predictors $\bm{\mu}$ and $\bm{\nu}$ allows parallel updating of the primal variables $\mathbf{x}$ and $\mathbf{u}$, exactly as in the general PCPM algorithm. However, the primal variable $\mathbf{x} = (x_1 \dots x_j \dots x_{n_1})^T$ cannot be further decomposed into parallel updating of each component $x_j$,  due to the coupling terms $\mathbf{x}^TP_i\mathbf{x}$, $i = 0, 1, \dots, m_1$,  unless all $P_i$'s are diagonal matrices. To realize parallel updating of $x_j$'s, we propose a simple idea to use $P_i\mathbf{x}^k$ as a ``predictor" for $P_i\mathbf{x}$ in the optimization \eqref{eq: x primal minimization}. 
\par To illustrate the idea, it may be easier to consider the first-order optimality condition of \eqref{eq: x primal minimization}: 
\begin{equation}\label{eq: x coupled optimality condition}
\frac{1}{\rho} (\mathbf{x}^k - \mathbf{x}^{k+1}) \in \underbrace{P_0 \mathbf{x}^{k+1}}_{(\Delta_0)} + \mathbf{q}_0 + \sum_{i=1}^{m_1} \mu_i^{k+1} \big(\underbrace{P_i \mathbf{x}^{k+1}}_{(\Delta_i)} + \mathbf{q}_i\big) + A^T \bm{\nu}^{k+1} + \mathcal{N}_{\mathbb{X}}(\mathbf{x}^{k+1}),
\end{equation}
where $\mathcal{N}_{\mathbb{X}}(\mathbf{x}^{k+1})$ is the normal cone to the convex set $\mathbb{X} = \prod_{j=1}^{n_1} \mathbb{X}_j$ at the solution point $\mathbf{x}^{k+1}$. By approximating each $(\Delta_i)$ using the predictor $P_i \mathbf{x}^k$, $i = 0, 1, \ldots, m_1$, the first-order optimality condition now becomes
\begin{equation}\label{eq: x first decoupled optimality condition}
\frac{1}{\rho} (\mathbf{x}^k - \mathbf{x}^{k+1}) \in P_0 \mathbf{x}^k + \mathbf{q}_0 + \sum_{i=1}^{m_1} \mu_i^{k+1} \big(P_i \mathbf{x}^k + \mathbf{q}_i\big) + A^T \bm{\nu}^{k+1} + \mathcal{N}_{\mathbb{X}}(\mathbf{x}^{k+1}).
\end{equation}
With \eqref{eq: x first decoupled optimality condition}, it is easy to see that $\mathbf{x}^{k+1}$ can be obtained through component-wise calculations. (Note that the normal cone of box constraints has explicit algebraic expressions and can also be decomposed component-wise with respect to $\mathbf{x}^{k+1}$.) Unfortunately, this simple idea would not work theoretically in the sense that convergence to an optimal solution cannot be established. This is mainly due to the difficulty to bound the error of 
$\lVert P_i \mathbf{x}^{k+1} - P_i \mathbf{x}^{k} \rVert$ along the iterations. 
\par To overcome this hurdle, we propose a novel approach to split \eqref{eq: x first decoupled optimality condition} into two steps by first introducing  ``primal predictor" variable $\mathbf{y}^{k+1}$ %(defined later in \eqref{eq: x primal predictor update}) 
for the primal decision variable $\mathbf{x}^k$, followed by a corrector update:
\begin{subequations} 
\begin{align}
\begin{split}
&\textbf{step 1 (predictor)}: \\
&\frac{1}{\rho} (\mathbf{x}^k - \mathbf{y}^{k+1}) \in P_0 \mathbf{x}^k + \mathbf{q}_0 + \sum_{i=1}^{m_1} \lambda_i^k \big(P_i \mathbf{x}^k + \mathbf{q}_i\big) + A^T \bm{\gamma}^k + \mathcal{N}_{\mathbb{X}}(\mathbf{y}^{k+1});
\end{split} \label{eq: x decoupled optimality condition 1} \\
\begin{split}
&\textbf{step 2 (corrector)}: \\
&\frac{1}{\rho} (\mathbf{x}^k - \mathbf{x}^{k+1}) \in P_0 \mathbf{y}^{k+1} \hspace*{-3pt}+ \mathbf{q}_0 + \sum_{i=1}^{m_1} \mu_i^{k+1} \big(P_i \mathbf{y}^{k+1} + \mathbf{q}_i\big) + A^T \bm{\nu}^{k+1} \hspace*{-3pt}+ \mathcal{N}_{\mathbb{X}}(\mathbf{x}^{k+1}).
\end{split} \label{eq: x decoupled optimality condition 2}
\end{align}
\end{subequations}
By focusing on box constraints for the generic set $\mathbb{X}_j$, and using the notation $[\mathbf{z}]_j$ to denote the $j$-th component of a vector $\mathbf{z}$, we can rewrite \eqref{eq: x decoupled optimality condition 1} and \eqref{eq: x decoupled optimality condition 2} component-wise as follows, for each $j = 1, \dots, n_1$:
\begin{subequations}
\begin{align}
\begin{split}
&\bullet \enspace \text{primal predictor of} \enspace x_j^k: \\
&y_j^{k+1} \coloneqq \Pi_{\mathbb{X}_j}\Big(x_j^k - \rho \big[P_0 \mathbf{x}^k + \mathbf{q}_0 + \sum_{i=1}^{m_1} \lambda_i^k \big(P_i \mathbf{x}^k + \mathbf{q}_i\big) + A^T \bm{\gamma}^k\big]_j\Big),
\end{split} \label{eq: x primal predictor update} \\
\begin{split}
&\bullet \enspace \text{primal corrector of} \enspace x_j^k: \\
&x_j^{k+1} = \Pi_{\mathbb{X}_j}\Big(x_j^k - \rho \big[P_0 \mathbf{y}^{k+1} + \mathbf{q}_0 + \sum_{i=1}^{m_1} \mu_i^{k+1} \big(P_i \mathbf{y}^{k+1} + \mathbf{q}_i\big) + A^T \bm{\nu}^{k+1}\big]_j \Big),
\end{split} \label{eq: x primal corrector update}
\end{align}
\end{subequations}
where the projection onto the box constraint set $\mathbb{X}_j$ can be expressed as:
\begin{equation}
\Pi_{\mathbb{X}_j}(x_j) \colon=
\left\{
\begin{array}{ll}
0, &\text{if } x_j < 0; \\[2pt]
x_j, &\text{if } 0 \leq x_j \leq \bar{X}_j; \\[2pt]
\bar{X}_j, &\text{if } x_j > \bar{X}_j.
\end{array}
\right.
\end{equation}
With \eqref{eq: x primal predictor update} and \eqref{eq: x primal corrector update}, in addition to the apparent benefits of updating the variables component-wise, the multiplications of $P_i \mathbf{x}^k$ and $P_i \mathbf{y}^{k+1}$, $i = 0, 1, \dots, m_1$ in \eqref{eq: x primal predictor update} and \eqref{eq: x primal corrector update} do not need to be carried out completely in each computing unit responsible for updating $y_j^{k+1}$ and $x_j^{k+1}$. The multiplications can be divided into multiple sub-tasks, and each of them only needs the $j$-th column of matrices $P_i$'s and can be accomplished locally by each computing unit. 
%Such an observation will allow distributed storage of the potentially huge-sized matrices. 
More detailed discussions of this point are provided in Section \ref{subsec: DS}.

%\hl{In each computing unit responsible for primal updates, for example, calculating $y_j^{k+1}$ for $j = 1, \dots, n_1$ using} \eqref{eq: x primal predictor update}, \hl{despite the fact that the value of $j$-th component of each $P_i \mathbf{x}^k$ for $i = 0, 1, \dots, m_1$ is needed, only the $j$-th column of each matrix $P_i$ is needed to be stored in this single unit, denoted as $[P_i]_j$. Since $P_i \mathbf{x}^k = \sum_{j=1}^{n_1} [P_i]_j x_j^k$, the calculation of $P_i \mathbf{x}^k$ is first accomplished using MPI to collect subtask results of multiplying $[P_i]_j$ with $x_j^k$ by all computing units, and then each component $[P_i \mathbf{x}^k]_j$ is scattered back to the corresponding unit $j$. This feature greatly facilitates our algorithm with distributed storage of data to achieve memory efficiency. More calculation details of using MPI are revealed in Section} \ref{subsec: DS}. 
%
\par The update of the other primal variable, $\mathbf{u}^{k+1}$, can be performed in a similar fashion, which is to split into two steps by first introducing a predictor variable $\mathbf{v}^{k+1}$ for $\mathbf{u}^k$, followed by a corrector update:
\begin{subequations}
\begin{align}
\begin{split}
&\bullet \enspace \text{primal predictor of} \enspace u_j^k: \\
&v_j^{k+1} \colon= u_j^k - \rho\big[\mathbf{c}_0 + \sum_{i=1}^{m_1} \lambda_i^k \mathbf{c}_i + B^T \bm{\gamma}^k\big]_j, \quad j = 1, \dots, n_2,
\end{split} \label{eq: u primal predictor update} \\
\begin{split}
&\bullet \enspace \text{primal corrector of} \enspace u_j^k: \\
&u_j^{k+1} = u_j^k - \rho\big[\mathbf{c}_0 + \sum_{i=1}^{m_1} \mu_i^{k+1} \mathbf{c}_i + B^T \bm{\nu}^{k+1}\big]_j, \quad j = 1, \dots, n_2.
\end{split} \label{eq: u primal corrector update}
\end{align}
\end{subequations} 
A dual corrector update is then performed for each Lagrangian multiplier $(\bm{\lambda}^{k+1}, \bm{\gamma}^{k+1})$: 
\begingroup
\begin{align}
&\bullet \enspace \text{dual corrector}: \nonumber \\
&\lambda_i^{k+1} = \Pi_{\mathbb{R}_+}\Big(\lambda_i^k + \rho \big[\frac{1}{2} (\mathbf{y}^{k+1})^T P_i \mathbf{y}^{k+1} + \mathbf{q}_i^T \mathbf{y}^{k+1} + \mathbf{c}_i^T \mathbf{v}^{k+1} + r_i\big]\Big), \nonumber \\[2pt]
&\qquad \qquad i = 1, \dots, m_1, \nonumber \\
&\gamma_i^{k+1} = \gamma_i^k + \rho \big[A \mathbf{y}^{k+1} + B \mathbf{v}^{k+1} - \mathbf{b}\big]_i, \quad i = 1, \dots, m_2. \label{eq: dual corrector update}
\end{align}
\endgroup
\par The overall structure of the proposed algorithm, which we name it 
PC$^2$PM, to reflect the fact that two sets of predictors and correctors are utilized, is presented in Algorithm \ref{alg: PC^2PM} below.
\begin{algorithm}
\caption{PC$^2$PM}\label{alg: PC^2PM}
\begin{algorithmic}[1]
\State \textbf{Initialization} choose an arbitrary starting point $(\mathbf{x}^0, \mathbf{u}^0, \bm{\lambda}^0, \bm{\gamma}^0)$.
\State $k \gets 0$.
\While{termination conditions are not met}
\State $\bullet$ Adaptive Step-size \par \textbf{update} the step-size $\rho^{k+1}$;
\State $\bullet$ Predictor Update \par \textbf{update} $(\bm{\mu}^{k+1}, \bm{\nu}^{k+1})$, $\mathbf{y}^{k+1}$, and $\mathbf{v}^{k+1}$ according to \eqref{eq: dual predictor update},  \eqref{eq: x primal predictor update} and \eqref{eq: u primal predictor update};   
\State $\bullet$ Corrector Update \par \textbf{update} $\mathbf{x}^{k+1}$, 
$\mathbf{u}^{k+1}$, and $(\bm{\lambda}^{k+1}, \bm{\gamma}^{k+1})$  according to \eqref{eq: x primal corrector update}, \eqref{eq: u primal corrector update} and \eqref{eq: dual corrector update};
\State $k \leftarrow k+1$
\EndWhile
\State \textbf{return} $(\mathbf{x}^k, \mathbf{u}^k, \bm{\lambda}^k, \bm{\gamma}^k)$.
\end{algorithmic}
\end{algorithm} 
\par Note that the starting point of the PC$^2$PM algorithm can be arbitrary, and is not required to be feasible. To establish convergence of the algorithm, the specific rules to update the step-size $\rho$ are crucial, which is the main focus of the next section. The implementation details, including distributed data storage, parallel computing through Message Passing Interface (MPI), and termination conditions, are provided in Section \ref{sec: Imp}.
%
%
%
%
%%%%%%%%%%%%%%%%%%%%%%%%%%%%%%%%%%%%%%%%%%%%%%%%%%%%%%%%%%%%%%%%%%%%%%%%%%%%%%%%
\section{Convergence Analysis}
In this section, we establish sufficient conditions for the PC$^2$PM algorithm to converge to an optimal solution from any starting point. First, we make a standard assumption on \eqref{eq: QCQP Problem Form} about the existence of an optimal solution.
\begin{assumption}\label{as:Exist}[Existence of an Optimal Solution]
The convex QCQP \eqref{eq: QCQP Problem Form} is assumed to have an optimal solution, denoted by $(\mathbf{x}^*, \mathbf{u}^*)$. 
\end{assumption}

With Assumption \ref{as:Exist} and the assumption on Slater's CQ, we know that a saddle point exists for the convex QCQP \eqref{eq: QCQP Problem Form}; more specifically, $(\mathbf{x}^*, \mathbf{u}^*, \bm{\lambda}^*, \bm{\gamma}^*)$ is a saddle point of \eqref{eq: QCQP Problem Form} if for any $\mathbf{x} \in \mathbb{X}$, $\mathbf{u} \in \mathbb{R}^{n_2}$, $\bm{\lambda} \in \mathbb{R}_+^{m_1}$ and $\bm{\gamma} \in \mathbb{R}^{m_2}$, we have that 
\begin{equation}\label{eq: saddle point inequality}
	\mathcal{L}(\mathbf{x}^*, \mathbf{u}^*, \bm{\lambda}, \bm{\gamma}) \leq \mathcal{L}(\mathbf{x}^*, \mathbf{u}^*, \bm{\lambda}^*, \bm{\gamma}^*) \leq \mathcal{L}(\mathbf{x}, \mathbf{u}, \bm{\lambda}^*, \bm{\gamma}^*), 
\end{equation}
where $\mathcal{L}(\mathbf{x}, \mathbf{u}, \bm{\lambda}, \bm{\gamma})$ is the Lagrangian function of \eqref{eq: QCQP Problem Form}:
\begin{equation}\label{eq: Lagrangian}
\begin{aligned}
\mathcal{L}(\mathbf{x}, \mathbf{u}, \bm{\lambda}, \bm{\gamma}) \coloneqq &\frac{1}{2} \mathbf{x}^T P_0 \mathbf{x} + \mathbf{q}_0^T \mathbf{x} + \mathbf{c}_0^T \mathbf{u} + r_0 \\
+ &\sum_{i=1}^{m_1} \lambda_i \big(\frac{1}{2} \mathbf{x}^T P_i \mathbf{x} + \mathbf{q}_i^T \mathbf{x} + \mathbf{c}_i^T \mathbf{u} + r_i\big) + \bm{\gamma}^T (A \mathbf{x} + B \mathbf{u} - \mathbf{b}).
\end{aligned}
\end{equation}
The case when the convex QCQP is infeasible will be discussed in \ref{subsec: infea unbound}.
%\hl{According to the Saddle Point Theorem (see Corollary 28.3.1 in} \cite{rockafellar2015convex}), \hl{coupled with the blanket assumption that %Slater's CQ holds for} \eqref{eq: QCQP Problem Form}, \hl{the above assumption is equivalent to say that there exists a saddle point 
%$(\mathbf{x}^*, \mathbf{u}^*, \bm{\lambda}^*, \bm{\gamma}^*)$ of the %Lagrangian function $\mathcal{L}(\mathbf{x}, \mathbf{u}, \bm{\lambda}, \bm{\gamma})$, where $(\mathbf{x}^*, \mathbf{u}^*)$ is the optimal solution of} \eqref{eq: QCQP Problem Form}, \hl{$\bm{\lambda}^* \in \mathbb{R}_+^{m_1}$ and $\bm{\gamma}^* \in \mathbb{R}^{m_2}$.} For any $\mathbf{x} \in \mathbb{X}$, $\mathbf{u} \in \mathbb{R}^{n_2}$, $\bm{\lambda} \in \mathbb{R}_+^{m_1}$ and $\bm{\gamma} \in \mathbb{R}^{m_2}$, we have:

%
\par Next, we derive some essential lemmas for constructing the main convergence proof.
\begin{lemma}[Inequality of Proximal Minimization Point]\label{lem: Lemma1}
Given a closed, convex set $\mathbb{Z} \subset \mathbb{R}^n$, and a closed, convex differentiable function $F: \mathbb{Z} \to \mathbb{R}$. With a given  point $\bar{\mathbf{z}} \in \mathbb{Z}$ and a positive number $\rho > 0$, if $\widehat{\mathbf{z}}$ is a proximal minimization point; i.e. $\widehat{\mathbf{z}} \coloneqq \arg \underset{\mathbf{z} \in \mathbb{Z}}{\min}\ F(\mathbf{z}) + \frac{1}{2 \rho} \lVert \mathbf{z} - \bar{\mathbf{z}} \rVert_2^2$, then we have that
\begin{equation}
2 \rho [F(\widehat{\mathbf{z}}) - F(\mathbf{z})] \leq \lVert \bar{\mathbf{z}} - \mathbf{z} \rVert_2^2 - \lVert \widehat{\mathbf{z}} - \mathbf{z} \rVert_2^2 - \lVert \widehat{\mathbf{z}} - \bar{\mathbf{z}} \rVert_2^2, \quad \forall \mathbf{z} \in \mathbb{Z}.
\end{equation}
\end{lemma}
\begin{proof}
Denote $\Phi(\mathbf{z}) = F(\mathbf{z}) + \frac{1}{2 \rho} \lVert \mathbf{z} - \bar{\mathbf{z}} \rVert_2^2$. By the definition of $\widehat{\mathbf{z}}$, we have $\nabla_{\mathbf{z}} \Phi(\widehat{\mathbf{z}}) = \mathbf{0}$. Since $\Phi(\mathbf{z})$ is strongly convex with modulus $\frac{1}{\rho}$, it follows that $2 \rho \big[\Phi(\mathbf{z}) - \Phi(\widehat{\mathbf{z}})\big] \geq \lVert \widehat{\mathbf{z}} - \mathbf{z} \rVert_2^2$ for any $\mathbf{z} \in \mathbb{Z}$. \qed
\end{proof}
For the ease of presenting the next two lemmas, we introduce a notation for the linear approximation of the Lagrangian function \eqref{eq: Lagrangian}.
\begin{definition}
With a given tuple $\left(\mathbf{x}^{\prime},\  \bm{\lambda}^{\prime},\ \bm{\gamma}^{\prime}  \right) \in \mathbb{X} \times  \mathbb{R}_+^{m_1} \times \mathbb{R}^{m_2} $, we define the following function $\mathcal{R}: \mathbb{X} \times \mathbb{R}^{n_2} \to \mathbb{R}$ as a linear approximation of the Lagrangian function $\mathcal{L}(\mathbf{x}, \mathbf{u}, \bm{\lambda}, \bm{\gamma})$ evaluated at $(\mathbf{x}^{\prime}, \bm{\lambda}^{\prime}, \bm{\gamma}^{\prime})$.
\begin{equation}\label{eq: R_func}
\begin{aligned}
\mathcal{R}(\mathbf{x}, \mathbf{u}; \mathbf{x}^{\prime}, \bm{\lambda}^{\prime}, \bm{\gamma}^{\prime}) \coloneqq &\ (P_0 \mathbf{x}^{\prime} + \mathbf{q}_0)^T \mathbf{x} + \mathbf{c}_0^T \mathbf{u} + r_0 \\
+ &\sum_{i=1}^{m_1} \lambda_i^{\prime} \big[(P_i \mathbf{x}^{\prime} + \mathbf{q}_i)^T \mathbf{x} + \mathbf{c}_i^T \mathbf{u} + r_i\big] + (\bm{\gamma}^{\prime})^T (A \mathbf{x} + B \mathbf{u} - \mathbf{b}),
\end{aligned}
\end{equation}
for any $\mathbf{x} \in \mathbb{X}$ and $\mathbf{u} \in \mathbb{R}^{n_2}$.
\end{definition}
\begin{lemma}\label{lem: Lemma2}
The update steps \eqref{eq: dual predictor update}, \eqref{eq: x primal predictor update}, \eqref{eq: x primal predictor update}, \eqref{eq: u primal predictor update}, \eqref{eq: u primal corrector update} and \eqref{eq: dual corrector update} are equivalent to obtaining proximal minimization points as follows:
\begin{subequations}
\begin{flalign}
 (\bm{\mu}^{k+1},\ \bm{\nu}^{k+1})  & =   \underset{\bm{\lambda} \in \mathbb{R}_+^{m_1},\ \bm{\gamma} \in \mathbb{R}^{m_2}}{\argmin} -\mathcal{L}(\mathbf{x}^k, \mathbf{u}^k, \bm{\lambda}, \bm{\gamma})  
\nonumber  \\
& \hspace*{10pt} + \frac{1}{2 \rho^{k+1}} \lVert \bm{\lambda} - \bm{\lambda}^k \rVert_2^2 + \frac{1}{2 \rho^{k+1}} \lVert \bm{\gamma} - \bm{\gamma}^k \rVert_2^2; \label{eq: mu_k+1}  \\[10pt]
(\mathbf{y}^{k+1},\ \mathbf{v}^{k+1})  & = \underset{\mathbf{x} \in \mathbb{X},\ \mathbf{u} \in \mathbb{R}^{n_2}}{\argmin} \mathcal{R}(\mathbf{x}, \mathbf{u}; \mathbf{x}^k, \bm{\lambda}^k, \bm{\gamma}^k) \nonumber \\
& \hspace*{10pt} + \frac{1}{2 \rho^{k+1}} \lVert \mathbf{x} - \mathbf{x}^k \rVert_2^2 + \frac{1}{2 \rho^{k+1}} \lVert \mathbf{u} - \mathbf{u}^k \rVert_2^2; \label{eq: y_k+1} \\[10pt]
(\mathbf{x}^{k+1},\ \mathbf{u}^{k+1}) & = \underset{\mathbf{x} \in \mathbb{X},\ \mathbf{u} \in \mathbb{R}^{n_2}}{\argmin}
\mathcal{R}(\mathbf{x}, \mathbf{u}; \mathbf{y}^{k+1}, \bm{\mu}^{k+1}, \bm{\nu}^{k+1}) \nonumber \\
& \hspace*{10pt} + \frac{1}{2 \rho^{k+1}} \lVert \mathbf{x} - \mathbf{x}^k \rVert_2^2 + \frac{1}{2 \rho^{k+1}} \lVert \mathbf{u} - \mathbf{u}^k \rVert_2^2; \label{eq: x_k+1} \\[10pt]
(\bm{\lambda}^{k+1},\ \bm{\gamma}^{k+1}) & = \underset{\bm{\lambda} \in \mathbb{R}_+^{m_1},\ \bm{\gamma} \in \mathbb{R}^{m_2}}{\argmin} -\mathcal{L}(\mathbf{y}^{k+1}, \mathbf{v}^{k+1}, \bm{\lambda}, \bm{\gamma}) \nonumber \\
& \hspace*{10pt} + \frac{1}{2 \rho^{k+1}} \lVert \bm{\lambda} - \bm{\lambda}^k \rVert_2^2 + \frac{1}{2 \rho^{k+1}} \lVert \bm{\gamma} - \bm{\gamma}^k \rVert_2^2. \label{eq: lambda_k+1}
\end{flalign}
\end{subequations}
\end{lemma}
\qed 
\par Since all the four optimization in \eqref{eq: mu_k+1} -- \eqref{eq: lambda_k+1} are convex optimization problems with linear constraints, the proof follows directly from the first-order optimality conditions of each of the optimization problems, and hence is omitted. 
\begin{lemma}\label{lem: Lemma3}
At a saddle point $(\mathbf{x}^*, \mathbf{u}^*, \bm{\lambda}^*, \bm{\gamma}^*)$ of the QCQP \eqref{eq: QCQP Problem Form}, the following inequality holds for any $\mathbf{x} \in \mathbb{X}$, $\mathbf{u} \in \mathbb{R}^{n_2}$, $\bm{\lambda} \in \mathbb{R}_+^{m_1}$ and $\bm{\gamma} \in \mathbb{R}^{m_2}$:
\begin{equation}
\begin{aligned}
&\mathcal{R}(\mathbf{x}^*, \mathbf{u}^*; \mathbf{x}, \bm{\lambda}, \bm{\gamma}) - \mathcal{R}(\mathbf{x}, \mathbf{u}; \mathbf{x}, \bm{\lambda}, \bm{\gamma}) \\
\leq &\sum_{i=1}^{m_1} (\lambda_i^* - \lambda_i) \big(\frac{1}{2} \mathbf{x}^T P_i \mathbf{x} + \mathbf{q}_i^T \mathbf{x} + \mathbf{c}_i^T \mathbf{u} + r_i\big) + (\bm{\gamma}^* - \bm{\gamma})^T (A \mathbf{x} + B \mathbf{u} - \mathbf{b}).
\end{aligned}
\end{equation}
\end{lemma}
\begin{proof}
For any $\mathbf{x} \in \mathbb{X}$, $\mathbf{u} \in \mathbb{R}^{n_2}$, $\bm{\lambda} \in \mathbb{R}_+^{m_1}$ and $\bm{\gamma} \in \mathbb{R}^{m_2}$,  we have that $\mathcal{L}(\mathbf{x}, \mathbf{u}, \bm{\lambda}^*, \bm{\gamma}^*) \geq \mathcal{L}(\mathbf{x}^*, \mathbf{u}^*, \bm{\lambda}, \bm{\gamma})$ by the saddle point inequality \eqref{eq: saddle point inequality}. We also have the inequality $\frac{1}{2} (\mathbf{x} - \mathbf{x}^*)^T P_0 (\mathbf{x} - \mathbf{x}^*) + \sum_{i=1}^{m_1} \lambda_i \big[\frac{1}{2} (\mathbf{x} - \mathbf{x}^*)^T P_i (\mathbf{x} - \mathbf{x}^*)\big] \geq 0$ due to the positive semi-definiteness of each matrix $P_0, P_1, \dots, P_{m_1}$. Adding the two inequalities together completes the proof. \qed
\end{proof}
\par We next establish fundamental estimates of the distance between the solution point $(\mathbf{x}^{k+1}, \mathbf{u}^{k+1}, \bm{\lambda}^{k+1}, \bm{\gamma}^{k+1})$ at each iteration $k$ and the saddle point $(\mathbf{x}^*, \mathbf{u}^*, \bm{\lambda}^*, \bm{\gamma}^*)$. 
\begin{proposition}\label{prp: Distance}
Let $(\mathbf{x}^*, \mathbf{u}^*, \bm{\lambda}^*, \bm{\gamma}^*)$ be a saddle point of the QCQP \eqref{eq: QCQP Problem Form}. 
For all $k \geq 0$, we have that
\begingroup
\begin{align}
&\ \lVert \mathbf{x}^{k+1} - \mathbf{x}^* \rVert_2^2 + \lVert \mathbf{u}^{k+1} - \mathbf{u}^* \rVert_2^2 \nonumber \\
\leq &\ \lVert \mathbf{x}^k - \mathbf{x}^* \rVert_2^2 + \lVert \mathbf{u}^k - \mathbf{u}^* \rVert_2^2 \nonumber \\[5pt]
- &\ \left(\lVert \mathbf{y}^{k+1} - \mathbf{x}^{k+1} \rVert_2^2 + \lVert \mathbf{v}^{k+1} - \mathbf{u}^{k+1} \rVert_2^2 + \lVert \mathbf{y}^{k+1} - \mathbf{x}^k \rVert_2^2 + \lVert \mathbf{v}^{k+1} - \mathbf{u}^k \rVert_2^2\right) \nonumber \\[3pt]
+ & \ 2 \rho^{k+1} \Bigg\{  (\mathbf{y}^{k+1} - \mathbf{x}^{k+1})^T P_0 (\mathbf{y}^{k+1} - \mathbf{x}^k) \nonumber \\[3pt]
&\hspace*{35pt} + \sum_{i=1}^{m_1} \mu_i^{k+1} (\mathbf{y}^{k+1} - \mathbf{x}^{k+1})^T P_i (\mathbf{y}^{k+1} - \mathbf{x}^k) \nonumber \\[3pt]
&\hspace*{35pt} + \sum_{i=1}^{m_1} (\lambda_i^* - \mu_i^{k+1})\left[\frac{1}{2} (\mathbf{y}^{k+1})^T P_i \mathbf{y}^{k+1} + \mathbf{q}_i^T \mathbf{y}^{k+1} + \mathbf{c}_i^T \mathbf{v}^{k+1} + r_i\right] \nonumber \\[3pt]
&\hspace*{35pt} + (\bm{\gamma}^* - \bm{\nu}^{k+1})^T (A \mathbf{y}^{k+1} + B \mathbf{v}^{k+1} - \mathbf{b}) \nonumber \\[3pt]
&\hspace*{35pt} + \sum_{i=1}^{m_1} (\mu_i^{k+1} - \lambda_i^{k}) \left[(P_i \mathbf{x}^k + \mathbf{q}_i)^T (\mathbf{y}^{k+1} - \mathbf{x}^{k+1}) + \mathbf{c}_i^T (\mathbf{v}^{k+1} - \mathbf{u}^{k+1})\right] \nonumber \\[3pt]
&\hspace*{35pt} + (\bm{\nu}^{k+1} - \bm{\gamma}^k)^T \left[A (\mathbf{y}^{k+1} - \mathbf{x}^{k+1}) + B (\mathbf{v}^{k+1} - \mathbf{u}^{k+1})\right]\Bigg\}, \label{eq: primal distance}
\end{align}
\endgroup
and
\begingroup
\begin{align}
&\ \lVert \bm{\lambda}^{k+1} - \bm{\lambda}^* \rVert_2^2 + \lVert \bm{\gamma}^{k+1} - \bm{\gamma}^* \rVert_2^2 \nonumber \\
\leq &\ \lVert \bm{\lambda}^k - \bm{\lambda}^* \rVert_2^2 + \lVert \bm{\gamma}^k - \bm{\gamma}^* \rVert_2^2 \nonumber \\[5pt]
- &\ \big(\lVert \bm{\mu}^{k+1} - \bm{\lambda}^{k+1} \rVert_2^2 + \lVert \bm{\nu}^{k+1} - \bm{\gamma}^{k+1} \rVert_2^2 + \lVert \bm{\mu}^{k+1} - \bm{\lambda}^k \rVert_2^2 + \lVert \bm{\nu}^{k+1} - \bm{\gamma}^k \rVert_2^2\big) \nonumber \\[3pt]
+ &\ 2 \rho^{k+1} \Bigg\{\sum_{i=1}^{m_1} (\lambda_i^{k+1} - \lambda_i^*) \big[\frac{1}{2} (\mathbf{y}^{k+1})^T P_i \mathbf{y}^{k+1} + \mathbf{q}_i^T \mathbf{y}^{k+1} + \mathbf{c}_i^T \mathbf{v}^{k+1} + r_i\big] \nonumber \\[3pt]
&\hspace*{33pt} + (\bm{\gamma}^{k+1} - \bm{\gamma}^*)^T (A \mathbf{y}^{k+1} + B \mathbf{v}^{k+1} - \mathbf{b}) \nonumber \\[3pt]
&\hspace*{33pt} + \sum_{i=1}^{m_1} (\mu_i^{k+1} - \lambda_i^{k+1}) \big[\frac{1}{2} (\mathbf{x}^k)^T P_i \mathbf{x}^k + \mathbf{q}_i^T \mathbf{x}^k + \mathbf{c}_i^T \mathbf{u}^k + r_i\big] \nonumber \\[3pt]
&\hspace*{33pt} + (\bm{\nu}^{k+1} - \bm{\gamma}^{k+1})^T (A \mathbf{x}^k + B \mathbf{u}^k - \mathbf{b})\Bigg\}. \label{eq: dual distance}
\end{align}
\endgroup
\end{proposition}
\begin{proof}
The details of the proof are provided in Appendix~\ref{app: Proof_Sec3}. \qed
\end{proof}
\par Now we are ready to present the main convergence result. A key to the proof depends on the rules to adaptively update the step-size. The rules, however, are lengthy and purely technical, and hence their details are deferred to Appendix \ref{app:StepRule}.
\begin{theorem}[Global Convergence]\label{thm: Converge}
Assume that the Slater's CQ and Assumption 1 hold. 
At each iteration $k$ of Algorithm~\ref{alg: PC^2PM}, let the step-size $\rho^{k+1}$ be updated according to the update rules in Appendix~\ref{app:StepRule}. Then with an arbitrary starting point $(\mathbf{x}^0, \mathbf{u}^0, \bm{\lambda}^0, \bm{\gamma}^0) \in \mathbb{R}^{n_1} \times \mathbb{R}^{n_2} \times \mathbb{R}^{m_1} \times \mathbb{R}^{m_2}$, the sequence $\{(\mathbf{x}^k, \mathbf{u}^k, \bm{\lambda}^k, \bm{\gamma}^k)\}$  generated by Algorithm~\ref{alg: PC^2PM} converges to a saddle point $(\mathbf{x}^*, \mathbf{u}^*, \bm{\lambda}^*, \bm{\gamma}^*)$ of the QCQP \eqref{eq: QCQP Problem Form}.
\end{theorem}
\begin{proof}
Please see Appendix \ref{app: Proof_Sec3} for details. \qed
\end{proof}

A point we want to emphasize here is that the convergence result is quite strong in the sense that the entire iterative sequence, not just a subsequence, can be shown to converge to an optimization solution, with an arbitrary starting point. Such a result can help alleviate a strong assumption we made, which is to assume that a given convex QCQP has an optimal solution. While the algorithm or its convergence proof does not handle infeasible or unbounded cases, we will show in Section~\ref{subsec: infea unbound}  that from a practical perspective, our algorithm can just be blindly applied to a convex QCQP, and either infeasibility or unboundedness can be inferred from observing the behavior of the residuals we use for the algorithm's stopping criteria, which are to be defined in Section~\ref{subsec: Stopping}.

%
%
%
%
%%%%%%%%%%%%%%%%%%%%%%%%%%%%%%%%%%%%%%%%%%%%%%%%%%%%%%%%%%%%%%%%%%%%%%%%%%%%%%%%
\section{Implementation}\label{sec: Imp}
In this section, we discuss how to efficiently implement the PC$^2$PM algorithm, especially within a distributed framework. 
\subsection{Distributed Storage of Data and Parallel Computing}\label{subsec: DS}
As mentioned in the introduction section, one key feature of the PC$^2$PM algorithm for solving convex QCQPs is that when implemented across multiple computing units, each computing unit does not need to store entire matrices. Instead, only each primal computing unit needs to store certain columns of the matrices (that is, the Hessian matrices in the objective function and the constraints). To illustrate this point, we use the primal predictor update \eqref{eq: x primal predictor update} as an example. 
Assume that ideally we have $n_1$ primal computing units dedicated to updating $y_j$, $j = 1, \dots, n_1$. To ease the argument, we write out the updating rule again here: 
\begin{equation}\label{eq:Update_y}
y_j^{k+1} = \Pi_{\mathbb{X}_j} \left(x_j^k - \rho \left[P_0 \mathbf{x}^k + \mathbf{q}_0 + \sum_{i=1}^{m_1} \lambda_i^k \big(P_i \mathbf{x}^k + \mathbf{q}_i\big) + A^T \bm{\gamma}^k\right]_j\right). 
\end{equation}
In each unit $j$, only the values of $x_j^k$, $[P_i]_j$, $[\mathbf{q}_i]_j$ for $i = 0, 1, \dots, m_1$ and $[A]_j$ are needed to be stored locally. 
%To update $y_j^{k+1}$ according to \eqref{eq:Update_y}, unit $j$ needs the values of $x_j^k$ and $[\mathbf{q}_i]_j$ for $i = 0, 1, \dots, m_1$, which can be simply retrieved from local storage. 
To calculate $[P_i \mathbf{x}^k]_j$ for $i = 0, 1, \dots, m_1$, there is no need to store the entire $P_i$ matrices on each computing unit. 
Instead, the value of $[P_i \mathbf{x}^k]_j$ can be obtained using MPI to communicate among all primal computing units, where only one column of the $P_i$ matrices (and $x_j^k$) is stored locally. Here we use a simple example to illustrate the mechanism. Let $n_1 = 3$, Fig~\ref{fig: matrix_vector_multiplication}
\begin{figure}[htb]
\begin{subfigure}{.556\textwidth}
\centering
\includegraphics[width=0.99\linewidth]{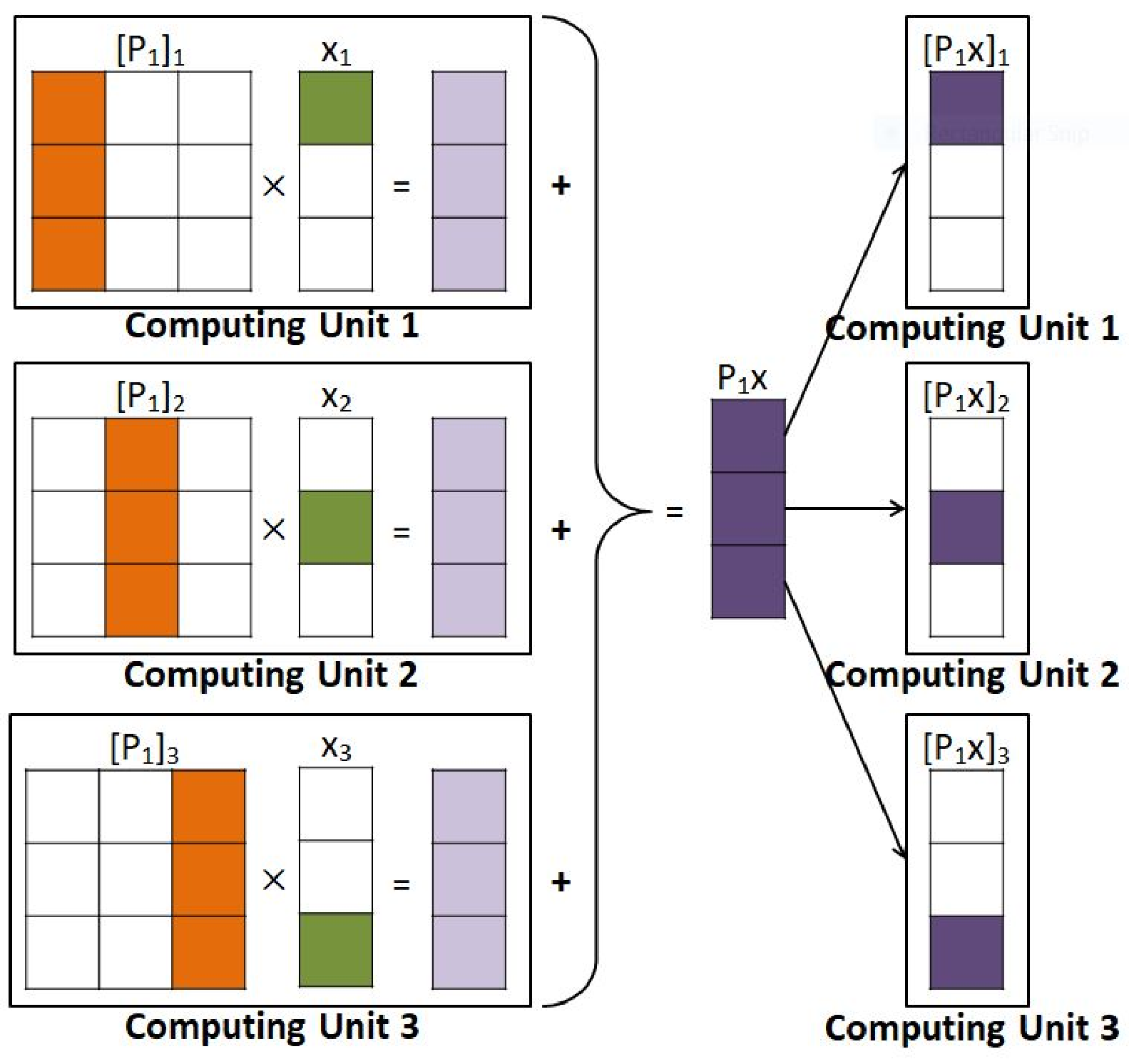}  
\caption{Calculating $[P_1 \mathbf{x}^k]_j$ for each computing unit $j$.}
\label{fig: matrix_vector_multiplication}
\end{subfigure}
\begin{subfigure}{.436\textwidth}
\centering
\includegraphics[width=0.99\linewidth]{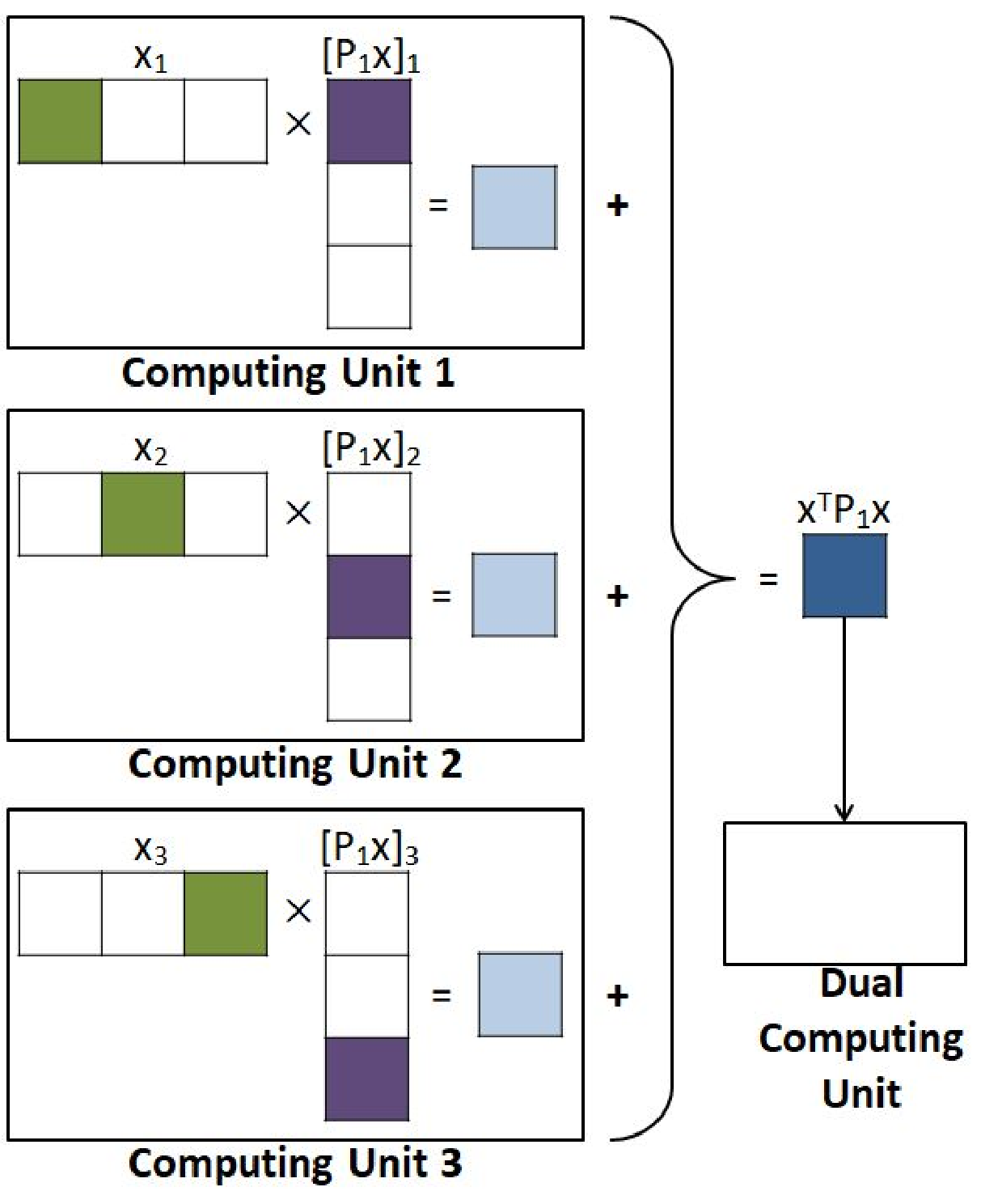}  
\caption{Calculating $(\mathbf{x}^k)^T P_1 \mathbf{x}^k$.}
\label{fig: vector_matrix_vector_multiplication}
\end{subfigure}
\caption{Illustrations of matrix-vector multiplications using MPI functions.}
\end{figure}
shows how $[P_1 \mathbf{x}^k]_j$, $j = 1, 2, 3$ are calculated in a distributed fashion through MPI. First, each computing unit $j$ completes a subtask of multiplying $[P_1]_j$ and $x_j^k$ using their locally stored information; then the intermediate results are summed up using the MPI\_Reduce function in a root process  to get the value of $P_1 \mathbf{x}^k$. Each component of the vector $P_1 \mathbf{x}^k$ is then sent back to the corresponding computing unit $j$ using the MPI\_Scatter function. After obtaining the values of $[P_i \mathbf{x}^k]_j$ for $i = 0, 1, \dots, m_1$ in this way, the update step \eqref{eq:Update_y} can be carried out upon receiving the values of $(\lambda_1^k, \dots, \lambda_{m_1}^k)$ and $\bm{\gamma}^k$ from other dual computing units dedicated for updating the dual variables using MPI\_Send and MPI\_Recv functions, with the fact that $[A^T \bm{\gamma}^k]_j = [A]_j^T \bm{\gamma}^k$. Such a feature will be particularly beneficial for solving high-dimension QCQPs from real world applications, as in many such cases the number of variables ($n_1$ for $\mathbf{x}$ and $n_2$ for $\mathbf{u}$) can be enormous. 
\par In the $3$-dimension example shown in Fig~\ref{fig: matrix_vector_multiplication}, once each $[P_1 \mathbf{x}^k]_j$ is received by computing unit $j$ for $j = 1, 2, 3$, a subtask of multiplying $x_j^k$ and $[P_1 \mathbf{x}^k]_j$ is needed to calculate the value of $(\mathbf{x}^k)^T P_1 \mathbf{x}^k$ for dual update, such as in \eqref{eq: dual predictor update}:
\begin{equation}
\mu_i^{k+1} = \Pi_{\mathbb{R}_+}\left(\lambda_i^k + \rho \left[\frac{1}{2} (\mathbf{x}^k)^T P_i \mathbf{x}^k + \mathbf{q}_i^T \mathbf{x}^k + \mathbf{c}_i^T \mathbf{u}^k + r_i\right]\right).
\end{equation}
Such a process is illustrated in Fig~\ref{fig: vector_matrix_vector_multiplication}, which shows that the locally calculated intermediate results are summed up using the MPI\_Reduce function and sent to the corresponding dual computing unit. Other matrix-vector (and vector-vector) multiplications in the update steps of Algorithm~\ref{alg: PC^2PM} can all be calculated in a similar fashion.\footnote{For more information, we refer the readers to our implementation codes programmed in C available online at \href{https://github.com/BigRunTheory/A-Distributed-Algorithm-for-Large-scale-Convex-QCQPs}{https://github.com/BigRunTheory/A-Distributed-Algorithm-for-Large-scale-Convex-QCQPs}.}
\par Next, we examine the speedup of using multiple compute nodes for parallel distributed computing. We run the PC$^2$PM algorithm on a multi-node computer cluster, where each node has multiple cores ($24$ cores in our case). MPI is used to communicate among all parallel processes mapped to cores belonging to different nodes. For illustration purpose, we focus on a single-constraint convex QCQP:
\begin{equation}\label{eq: Single-Constraint QCQP Problem Form}
\begin{aligned}
\underset{\mathbf{x} \in \mathbb{R}^{n_1}}{\text{minimize}} \quad &\frac{1}{2} \mathbf{x}^{T} P_0 \mathbf{x} + \mathbf{q}_0^T \mathbf{x} + r_0 \\
\text{subject to} \quad &\frac{1}{2} \mathbf{x}^{T} P_1 \mathbf{x} + \mathbf{q}_1^T \mathbf{x} + r_1 \leq 0, \qquad (\lambda_1)
\end{aligned}
\end{equation}
which does not contain the block of decision variable $\mathbf{u}$ or linear constraint $A \mathbf{x} + B \mathbf{u} = \mathbf{b}$. We test the PC$^2$PM algorithm for solving \eqref{eq: Single-Constraint QCQP Problem Form} on a randomly generated data set with $P_i \in \mathbb{R}^{n_1 \times n_1}$, $\mathbf{q}_i \in \mathbb{R}^{n_1}$ and $r_i \in \mathbb{R}$ for $i = 0, 1$. The dimension $n_1$ is set at $2^{14} \approx 1.6 \times 10^4$. Each matrix $P_i$ is randomly generated as a symmetric PSD matrix in the form of $P_i = Q^T D Q$, where $Q \in \mathbb{R}^{n_1 \times n_1}$ is a randomly generated orthogonal matrix, and $D = diag(d_1, \dots, d_{n_1})$ is a randomly generated diagonal matrix with all non-negative entries. Since $d_1, \dots, d_{n_1}$ are also the eigenvalues of each $P_i$, we denote the largest eigenvalue as $d_{\max}$ and the smallest eigenvalue as $d_{\min}$, and hence make the condition number of each matrix as $\kappa(P_i) = \frac{d_{max}}{d_{min}}$. Then, the remaining diagonal entries are randomly generated from the range $[d_{\min}, \ d_{\max}]$. We test different condition numbers for all matrices, increasing from $10^2$ to $10^6$. The values of the smallest eigenvalue $d_{min}$ and the largest eigenvalue $d_{max}$ are listed in Table~\ref{tab: value_d}.
\begin{table}[htb]
\renewcommand{\arraystretch}{1.4}
\centering
\footnotesize
\setlength\tabcolsep{10.0pt}
\begin{tabular}{c|cc} \hline
\textbf{cond. num.}&\textbf{smallest eigenvalue}&\textbf{largest eigenvalue} \\
$\Big(\kappa = \frac{d_{max}}{d_{min}}\Big)$&$\big(d_{min}\big)$&$\big(d_{max}\big)$ \\ \hline
$10^2$&$0.1$&$10.0$ \\ 
$10^4$&$0.003$&$30.0$ \\ 
$10^6$&$0.00002$&$20.0$ \\ \hline
\end{tabular}
\caption{Values of the smallest eigenvalue $d_{min}$ and the largest eigenvalue $d_{max}$ for different condition numbers.}
\label{tab: value_d}
\end{table}
The components of each vector $\mathbf{q}_i$ are randomly generated from the range $[-1.0, \ 1.0]$, and each scalar $r_i$ is randomly generated from the range $[-1.0, \ 0.0]$ to guarantee the feasibility of the constraint sets. 
\par Since the number of Lagrangian multipliers is $1$, the number of dual computing units $n_{\text{dual-comp}}$ is also fixed as $1$. The tasks of updating $n_1$ components of the primal decision variables $\mathbf{x}$ and $\mathbf{y}$ are evenly distributed among all the primal computing units with the number $n_{\text{primal-comp}}$ varying from $1$ to $256$ for comparison purpose. Since each computing unit occupies a single core, the total number of cores used is $n_{\text{core}} = n_{\text{primal-comp}} + n_{\text{dual-comp}}$. The number of nodes needed is calculated as $n_{\text{node}} = \lceil n_{\text{core}}/24 \rceil$ (where $24$ is the number of cores per node). The elapsed wall-clock time $T$ used by the PC$^2$PM algorithm to converge with a tolerance $\tau^{\text{PC}^2\text{PM}} = 10^{-3}$, corresponding to different condition numbers, is listed in Table~\ref{tab: Speedup}, along with the calculated objective function values. (The specific stopping criteria are given in Section~\ref{subsec: Stopping}.)
\begin{table}[htb]
\renewcommand{\arraystretch}{1.2}
\centering
\footnotesize
\setlength\tabcolsep{18.0pt}
\begin{tabular}{c|c|c|c|c} \hline
\multicolumn{5}{c}{\textbf{PC$^2$PM using multiple nodes} (max. 24 cores per node)} \\
\multicolumn{5}{c}{($\tau^{\text{PC}^2\text{PM}} = 10^{-3}$)} \\ \hline
\multirow{2}{*}{\textbf{$\mathbf{n}_{\text{node}}$}}&\multirow{2}{*}{\textbf{$\mathbf{n}_{\text{core}}$}}&\multicolumn{3}{c}{\textbf{elapsed wall-clock time} (hour)} \\
&&$\kappa = 10^2$&$\kappa = 10^4$&$\kappa = 10^6$ \\ \hline
$1$&$1 + 1$&$42.08$&$51.72$&$70.45$ \\
$1$&$2 + 1$&$22.79$&$28.11$&$38.28$ \\
$1$&$4 + 1$&$14.17$&$17.48$&$23.89$ \\
$1$&$8 + 1$&$8.85$&$10.87$&$14.70$ \\
$1$&$16 +1$&$5.60$&$6.91$&$9.62$ \\
$2$&$32 + 1$&$4.82$&$6.13$&$8.26$ \\
$3$&$64 + 1$&$4.33$&$5.42$&$7.51$ \\
$6$&$128 + 1$&$3.94$&$4.96$&$6.95$ \\
$11$&$256 + 1$&$5.51$&$6.84$&$9.32$ \\ \hline
\multicolumn{2}{c}{\textbf{obj. val.}}&$-420.621$&$-214.389$&$-324.428$ \\ \hline
\end{tabular}
\caption{Elapsed wall-clock time used by PC$^2$PM for solving the single-constraint convex QCQP \eqref{eq: Single-Constraint QCQP Problem Form} with different condition numbers ($\kappa$).}
\label{tab: Speedup}
\end{table} 
\par The computational speedup $S$ is defined as the ratio of the elapsed run time taken by a serial code to that taken by a parallel code for solving the same problem. More specifically, $S$ is defined as
\begin{equation}
S \coloneqq \frac{T(1+1)}{T(n_{\text{primal-comp}} + 1)}, \quad  n_{\text{primal-comp}} \geq 2.
\end{equation}
The speedup of solving \eqref{eq: Single-Constraint QCQP Problem Form} is shown in Fig~\ref{fig: Speedup}.
\begin{figure}[htb]
\centering
\begin{subfigure}{.33\textwidth}
\centering
\includegraphics[width=1.0\linewidth]{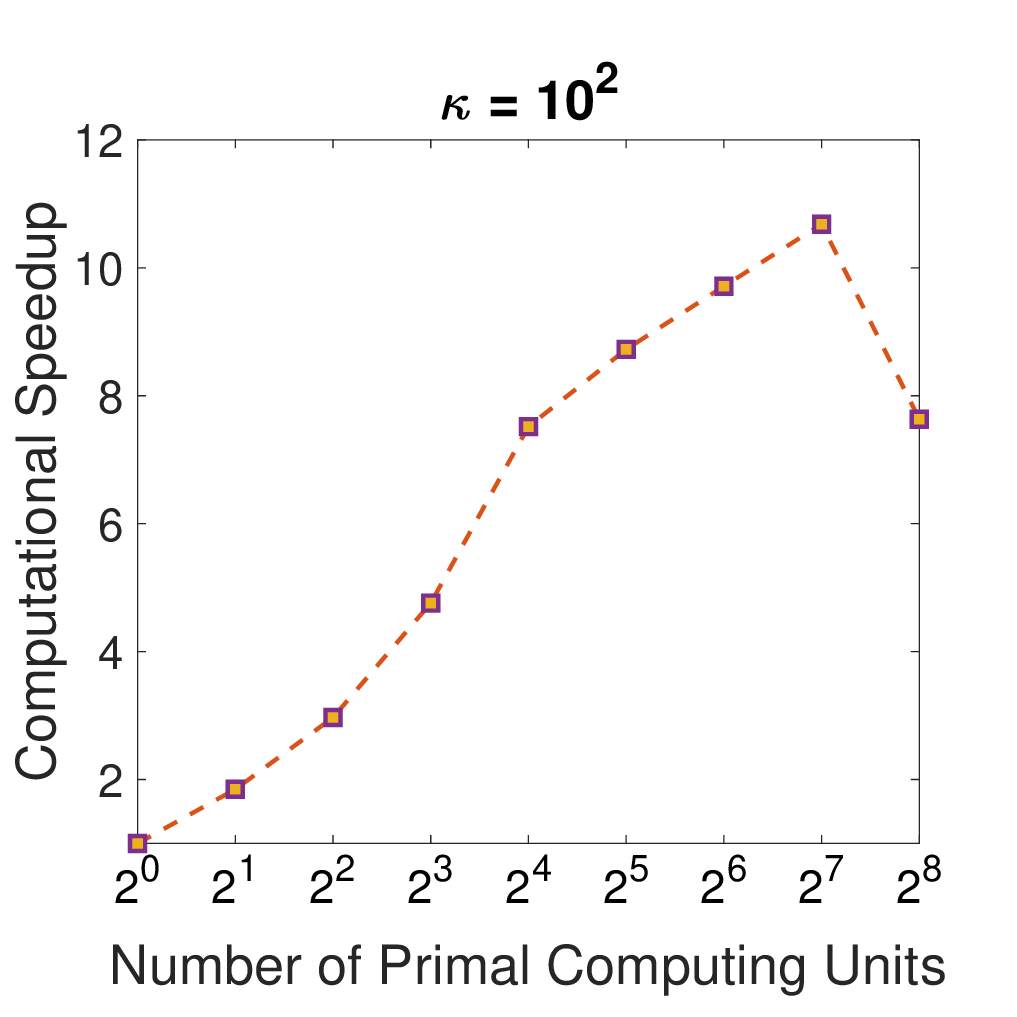}
\end{subfigure}%
\begin{subfigure}{.33\textwidth}
\centering
\includegraphics[width=1.0\linewidth]{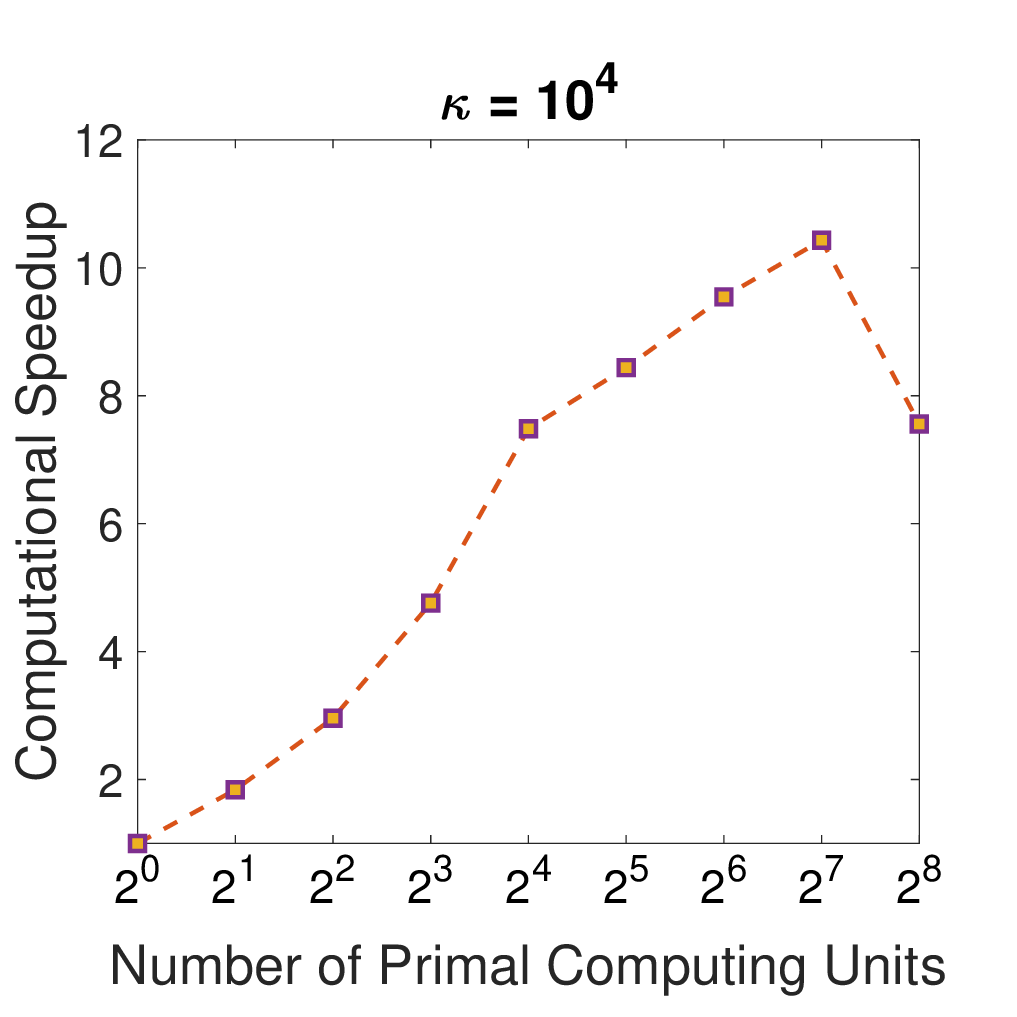}
\end{subfigure}%
\begin{subfigure}{.33\textwidth}
\centering
\includegraphics[width=1.0\linewidth]{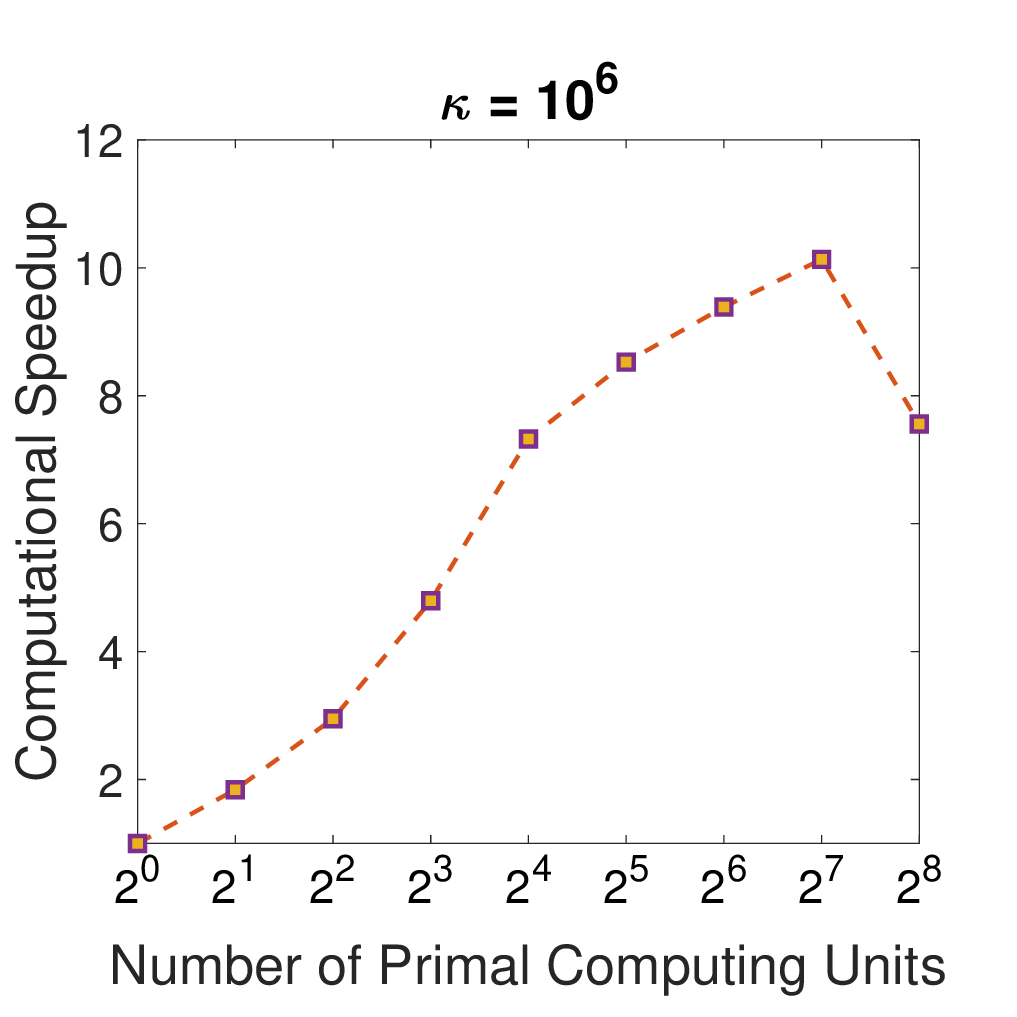}
\end{subfigure}%
\caption{Computational speedup of PC$^2$PM for solving the single-constraint convex QCQP \eqref{eq: Single-Constraint QCQP Problem Form} with different condition numbers ($\kappa$).}
\label{fig: Speedup}%\vspace*{-10pt}
\end{figure}
For this specific case, parallel computing achieved linear speedup initially. However, due to communication overhead, the speedup plateaued (or even decreased) when the number of computing units is too high. As such, we suggest that a proper number of computing units needs to be carefully chosen when implementing the PC$^2$PM algorithm to balance between computational speedup and communication overhead.
\subsection{Adaptive Step Size with Auto-learned Allocation Weights}\label{subsec: adp_step_size}
In establishing the global convergence of the PC$^2$PM algorithm, it is not specified how the values of $\epsilon_s$, $s=1, \ldots, 8$ are chosen in order to calculate the eight components $\rho_1$ -- $\rho_8$ (see the update rules in Appendix~\ref{app:StepRule}). 
Here we develop a practical rule to help determine the values of $\epsilon_s$'s along the iterations. 
The rule may also help accelerate the algorithm's performance, based on our numerical experiments.  
\par Generally speaking, for first-order algorithms, to which the PC$^2$PM algorithm also belongs,  the larger value a step size could take, the fewer number of iterations the algorithms would need to converge. For the step-size formula \eqref{eq: adaptive step size}, it is easy to observe that the value of each $\rho_s$ increases when the corresponding $\epsilon_s$ increases. However, the $\epsilon_s$'s cannot be too large as their summation is bounded by $1 - \epsilon_0 \leq 1$. A naive way to allocate the value of each $\epsilon_s$ is to evenly distribute the upper bound of their summation; that is, $\epsilon_s = \frac{1}{8} (1 - \epsilon_0)$ for $s = 1, \dots, 8$, throughout the convergence. Alternatively, we introduce a weight variable $w_s > 0$ for each $\epsilon_s$. At the beginning of the algorithm, they are all initialized as $1$, indicating an even allocation of the values of the $\epsilon_s$'s. At each iteration $k = 1, 2, \ldots$, we calculate the values of $\epsilon_1^k, \dots, \epsilon_8^k$ based on the following formulation:
\begin{equation}\label{eq: epsilon update}
\epsilon_s^k = \frac{w_s^k}{\sum_{s=1}^8 w_s^k} (1 - \epsilon_0), \quad s = 1, \dots, 8.
\end{equation}
After the step size $\rho^{k+1}$ is determined by $\min\{\rho_1^k, \dots, \rho_8^k\}$ according to \eqref{eq: adaptive step size}, the values of the weights $w_1^{k+1}, \dots, w_8^{k+1}$ are updated based on the ratio of $\rho^{k+1}$ to each $\rho_s^k$: 
\begin{equation}\label{eq: weight update}
w_s^{k+1} = \frac{\rho^{k+1}}{\rho_s^k} w_s^k, \quad s = 1, \dots, 8.
\end{equation}
The idea of the above updating rule is to make sure that all the values of the $\rho_s$'s will have a chance to be increased, avoiding the possibility that a particularly small $\rho_s$ would always be chosen to determine the step size $\rho^{k+1}$, which would slow down the whole algorithm. 
\par For illustration purpose, we use the PC$^2$PM algorithm with the step-size updating rule of \eqref{eq: adaptive step size}, \eqref{eq: epsilon update} and \eqref{eq: weight update} to solve the same single-constraint convex QCQP \eqref{eq: Single-Constraint QCQP Problem Form} in the previous subsection. We test the algorithm on a data set of matrix $P_i$, vector $\mathbf{q}_i$ and scalar $r_i$ randomly generated in the same way as in the previous subsection for $i = 0, 1$, but with $n_1 = 1024$ and $\kappa(P_i) = 100$. The algorithm is implemented on a single core as a serial code. (Note that the parallel computing has nothing to do with the number of iterations for the PC$^2$PM algorithm to converge.) Since it contains neither the decision variable $\mathbf{u}$ nor the linear equality constraint $A \mathbf{x} + B \mathbf{u} = \mathbf{b}$, only $\rho_1^k, \dots, \rho_5^k$ need to be calculated at each iteration. We compare the performance of the algorithm  using equal weights versus using the adaptive weights in \eqref{eq: weight update}. The number of iterations and the elapsed wall-clock time used by the algorithm to converge with a tolerance $\tau^{\text{PC}^2\text{PM}} = 10^{-4}$ are listed in Table~\ref{tab: Comparison of Allocations of Epsilons}. 
\begin{table}[htb]
\renewcommand{\arraystretch}{1.2}
\centering
\footnotesize
\setlength\tabcolsep{10.0pt}
\begin{tabular}{c|cc|cc} \hline
&\multicolumn{2}{c}{\textbf{Equal Weight Allocation}}&\multicolumn{2}{c}{\textbf{Adaptive Weight Allocation}} \\ \hline
\multirow{2}{*}{\textbf{value of} $\epsilon_0$}&\multirow{2}{*}{\textbf{num. iter.}}&\textbf{time}&\multirow{2}{*}{\textbf{num. iter.}}&\textbf{time} \\
&&(second)&&(second) \\ \hline
$0.5$&$59496$&$123$&$28272$&$62$ \\ \hline
$10^{-1}$&$33055$&$74$&$15713$&$36$ \\ \hline
$10^{-2}$&$30050$&$68$&$14287$&$32$ \\ \hline
$10^{-3}$&$29780$&$68$&$14158$&$32$ \\ \hline
$10^{-4}$&$29753$&$68$&$14145$&$32$ \\ \hline
$10^{-5}$&$29750$&$67$&$14143$&$31$ \\ \hline
$10^{-6}$&$29750$&$67$&$14143$&$31$ \\ \hline
$0$&$29750$&$66$&$14143$&$31$ \\ \hline
diminishing&$30094$&$70$&$14378$&$33$ \\ \hline
\end{tabular}
\caption{Number of iterations and elapsed wall-clock time used by PC$^2$PM for solving a single-constraint convex QCQP \eqref{eq: Single-Constraint QCQP Problem Form} with different settings of $(\epsilon_0, \dots, \epsilon_5)$.}
\label{tab: Comparison of Allocations of Epsilons}
\end{table}
We also test using different values of $\epsilon_0$, including a fixed value varying from $0.5$, $10^{-1}, \dots, 10^{-6}$ to $0$ and a diminishing value of $\frac{1}{\sqrt{k+1}}$. The numerical results of this specific instance suggest that by using auto-learned allocation weights, the number of iterations for the algorithm to converge is cut by more than half. Comparing each rows, we also observe that the smaller the value of $\epsilon_0$ is, the faster the algorithm converges. Additionally, for the row of $\epsilon_0 = 0$, we plot out the comparison of values of the resulting step size $\rho^k$ at each iteration $k$ using different weight allocations, shown in Fig~\ref{fig: rhos}.
\begin{figure}[htb]
\centering
\includegraphics[scale = 0.32]{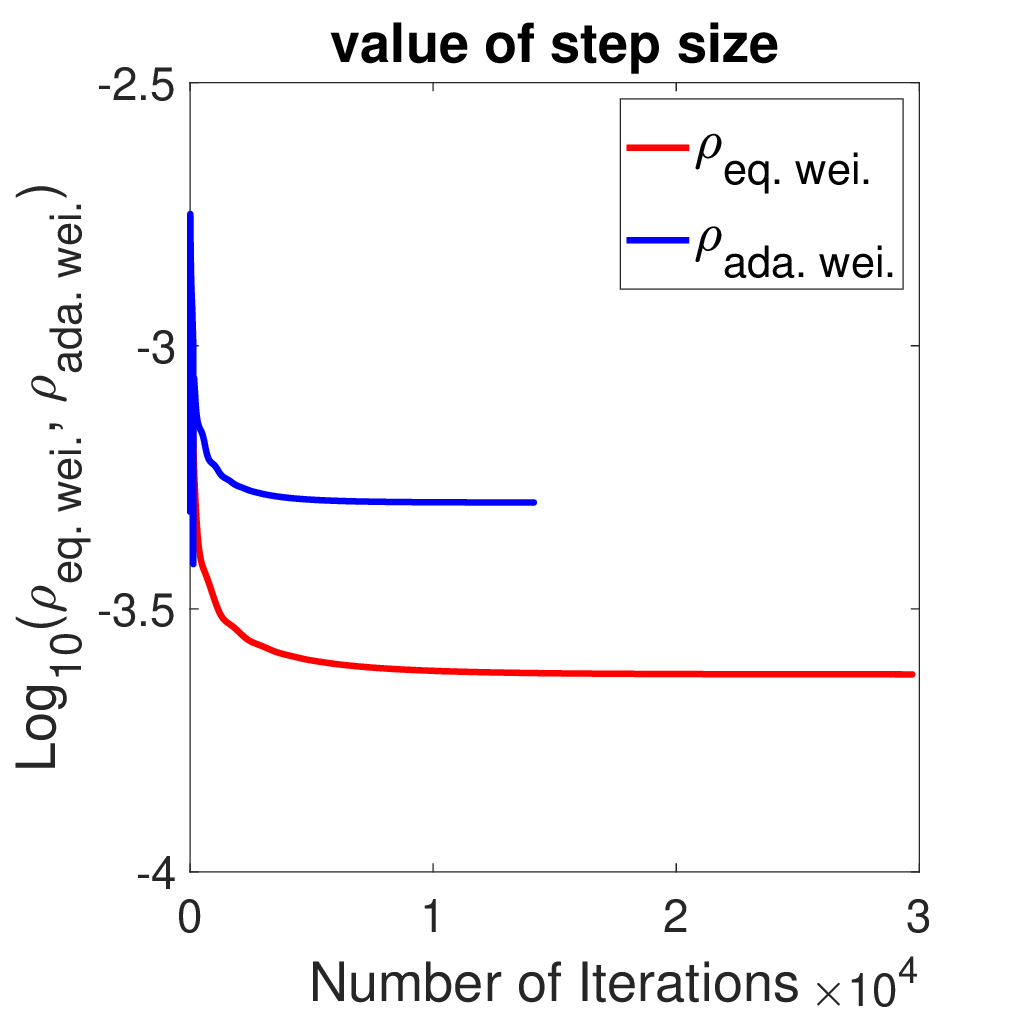}
\caption{Comparison of values of the step size using different weight allocations when $\epsilon_0 = 0$.}
\label{fig: rhos}
\end{figure}
We observe that using the adaptive weights, the step size quickly converges to a much larger value than using the equal weights, which explains the greatly reduced number of iterations.
\subsection{Stopping Criteria}\label{subsec: Stopping}
Since we consider a convex QCQP and assume that the Slater's CQ holds, the first-order optimality conditions (aka the KKT conditions) are both necessary and sufficient. More specifically, for an optimal solution $(\mathbf{x}^*, \mathbf{u}^*)$ of the QCQP \eqref{eq: QCQP Problem Form} and its corresponding dual solution $(\bm{\lambda}^*, \bm{\gamma}^*)$, the following conditions are satisfied: \\
$\bullet$ \textbf{Stationarity}:
\begin{subequations}
\begin{align}
-\big[P_0 \mathbf{x}^* + \mathbf{q}_0 + \sum_{i=1}^{m_1}  		\lambda_i^* (P_i \mathbf{x}^* + \mathbf{q}_i) + A^T   \bm{\gamma}^*\big]_j &\in  \mathcal{N}_{\mathbb{X}_j}(x_j^*), \quad j = 1, \dots, n_1 \label{eq:Opt_x} \\
\mathbf{c}_0 + \sum_{i=1}^{m_1} \lambda_i^* \mathbf{c}_i + B^T \bm{\gamma}^* &=  \mathbf{0} \label{eq:Opt_u}
\end{align}
\end{subequations}
$\bullet$ \textbf{Complementarity}:
\begin{equation}
\lambda_i^* \Big[\frac{1}{2} (\mathbf{x}^*)^T P_i \mathbf{x}^* + \mathbf{q}_i^T \mathbf{x}^* + \mathbf{c}_i^T \mathbf{u}^* + r_i\Big] = 0, \quad i = 1, \dots, m_1 \label{eq:complementary}
\end{equation}
$\bullet$ \textbf{Primal Feasibility}:
\begin{subequations}
\begin{align}
\frac{1}{2} (\mathbf{x}^*)^T P_i \mathbf{x}^* + \mathbf{q}_i^T \mathbf{x}^* + \mathbf{c}_i^T \mathbf{u}^* + r_i &\leq 0, \quad \quad i = 1, \dots, m_1 \label{eq:pri_fea_1} \\
A \mathbf{x}^* + B \mathbf{u}^* - \mathbf{b} &=  \mathbf{0} \label{eq:pri_fea_2}
\end{align}
\end{subequations}
$\bullet$ \textbf{Dual Feasibility}:
\begin{equation}
\lambda_i^* \geq 0, \quad i = 1, \dots, m_1 \label{eq:dual_fea}
\end{equation}

\noindent Conversely, any primal-dual pair $(\mathbf{x}^*, \mathbf{u}^*;  \bm{\lambda}^*, \bm{\gamma}^*)$ satisfying the above conditions is optimal to the primal and dual of the QCQP \eqref{eq: QCQP Problem Form}, respectively.  Based on the optimality conditions \eqref{eq:Opt_x} -- \eqref{eq:dual_fea}, we choose stopping criteria for our algorithm to measure stationarity, as well as complementarity and primal feasibility. (Note that dual feasibility is always maintained by the projection step in each iteration, as shown in \eqref{eq: dual corrector update}.) More specifically, at each iteration $k$, we measure the following two residuals:
\begin{align}
res_1^k = &\sqrt{\frac{1}{n_1 + n_2} \Big[\sum_{j=1}^{n_1} \big(res_1^k\_x_j\big)^2 + \lVert \mathbf{c}_0 + \sum_{i=1}^{m_1} \lambda_i^k \mathbf{c}_i + B^T \bm{\gamma}^k \rVert_2^2\Big]},\ \mathrm{and} \label{eq:res_1}\\
res_2^k = &\sqrt{
\begin{aligned}
&\frac{1}{m_1 + m_2} \Big[\sum_{i=1}^{m_1} \big[\lambda_i^k \big\lvert \frac{1}{2} (\mathbf{x}^k)^T P_i \mathbf{x}^k + \mathbf{q}_i^T \mathbf{x}^k + \mathbf{c}_i^T \mathbf{u}^k + r_i \big\rvert \big]^2 \\
&+ \lVert A \mathbf{x}^k + B \mathbf{u}^k - \mathbf{b} \rVert_2^2\Big]
\end{aligned}
}, \label{eq:res_2}
\end{align}
where $res_1^k\_x_j$ in \eqref{eq:res_1} depends on the actual form of the constraint set $\mathbb{X}$. Again, for box constrains $0 \leq x_j \leq  \bar{X}_j$, $j = 1, \ldots, n_1$, we have that  
\begin{equation}
res_1^k\_x_j \colon=
\left\{
\begin{array}{ll}
\min\{0, \big[\mathbf{grad}_{\mathbf{x}}\big]_j\}, \quad &\text{if } x_j^k = 0 \\[5pt]
\big[\mathbf{grad}_{\mathbf{x}}\big]_j, \quad &\text{if } 0 < x_j^k < \bar{X}_j \\[5pt]
\max\{0, \big[\mathbf{grad}_{\mathbf{x}}\big]_j\}, \quad &\text{if } x_j^k = \bar{X}_j
\end{array}
\right.,
\end{equation}
and $\mathbf{grad}_{\mathbf{x}} = P_0 \mathbf{x}^k + \mathbf{q}_0 + \sum_{i=1}^{m_1} \lambda_i^k (P_i \mathbf{x}^k + \mathbf{q}_i) + A^T \bm{\gamma}^k$. This comes from the rewriting of the optimality condition \eqref{eq:Opt_x} using $\mathbf{grad}_{\mathbf{x}}$ as:
\begin{equation}
-\big[\mathbf{grad}_{\mathbf{x}}^*\big]_j \in \mathcal{N}_{\mathbb{X}_j}(x_j^*),
\end{equation}
where 
\begin{equation}
\mathcal{N}_{\mathbb{X}_j}(x_j^*) \colon=
\left\{
\begin{array}{ll}
(-\infty, 0], \quad &\text{if } x_j^* = 0 \\[5pt]
\{0\}, \quad &\text{if } 0 < x_j^* < \bar{X}_j \\[5pt]
[0, +\infty), \quad &\text{if } x_j^* = \bar{X}_j
\end{array}
\right..
\end{equation}

We terminate our algorithm  when both of the two residuals drop below a pre-specified tolerance $\tau$. Note that the residuals defined in \eqref{eq:res_1} and \eqref{eq:res_2} are based on the average residuals of all the constraints. Other forms of residual metric, such as using the maximum residual of all the constraints, can also be used.
\subsection{Infeasibility and Unboundedness}\label{subsec: infea unbound}
Lastly, we examine how the PC$^2$PM algorithm could computationally detect infeasibility or unboundedness of a convex QCQP.
\par First, we construct an infeasible QCQP as follows:
\begin{equation}\label{eq: 2-Constraint QCQP Problem Form}
\begin{aligned}
\underset{\mathbf{x} \in \mathbb{R}^{n_1}}{\text{minimize}} \quad &\frac{1}{2} \mathbf{x}^{T} P_0 \mathbf{x} + \mathbf{q}_0^T \mathbf{x} + r_0 \\
\text{subject to} \quad &\frac{1}{2} \mathbf{x}^{T} P_1 \mathbf{x} + \mathbf{q}_1^T \mathbf{x} + r_1 \leq 0, \qquad (\lambda_1) \\
&\underbrace{\frac{1}{2} \big(\mathbf{x} + \mathbf{q}_2\big)^T \big(\mathbf{x} + \mathbf{q}_2\big)}_{\geq 0} + \underbrace{\Delta}_{> 0} \leq 0. \qquad (\lambda_2)
\end{aligned}
\end{equation}
All the matrices $P_i$'s, vectors $\mathbf{q}_i$'s and scalars $r_i$'s in \eqref{eq: 2-Constraint QCQP Problem Form} are randomly generated in the same way as in Section~\ref{subsec: DS}, but with $n_1 = 1024$ and $\kappa(P_i) = 100$. Letting $\Delta$ denote a positive scalar in the second quadratic constraint in \eqref{eq: 2-Constraint QCQP Problem Form} apparently makes the problem infeasible. We decrease $\Delta$ from $100$ to $0.01$ and apply the PC$^2$PM algorithm to solve the resulting problems. The corresponding residuals are shown in Fig~\ref{fig: res_infea}. 
\begin{figure}[htb]
\centering
\begin{subfigure}{.33\textwidth}
\centering
\includegraphics[width=1.0\linewidth]{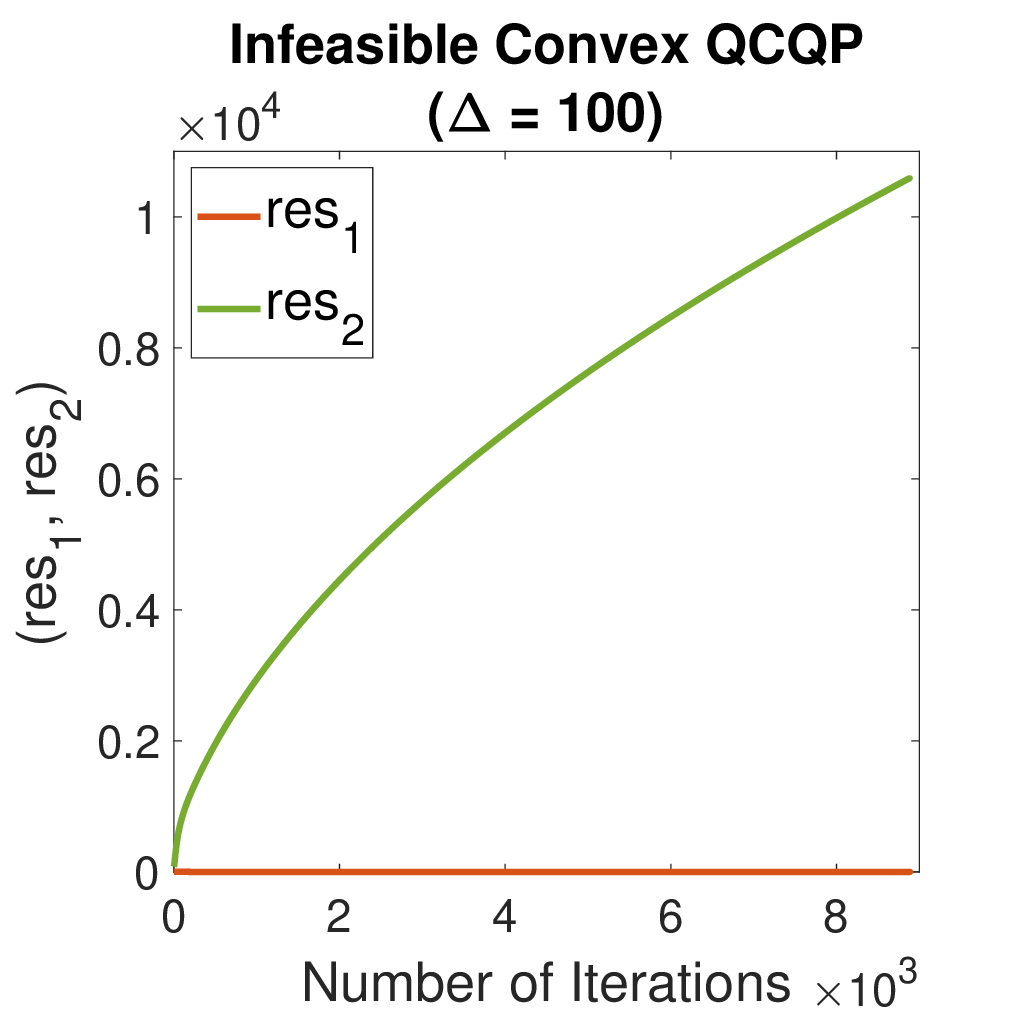}  
\end{subfigure}%
\begin{subfigure}{.33\textwidth}
\centering
\includegraphics[width=1.0\linewidth]{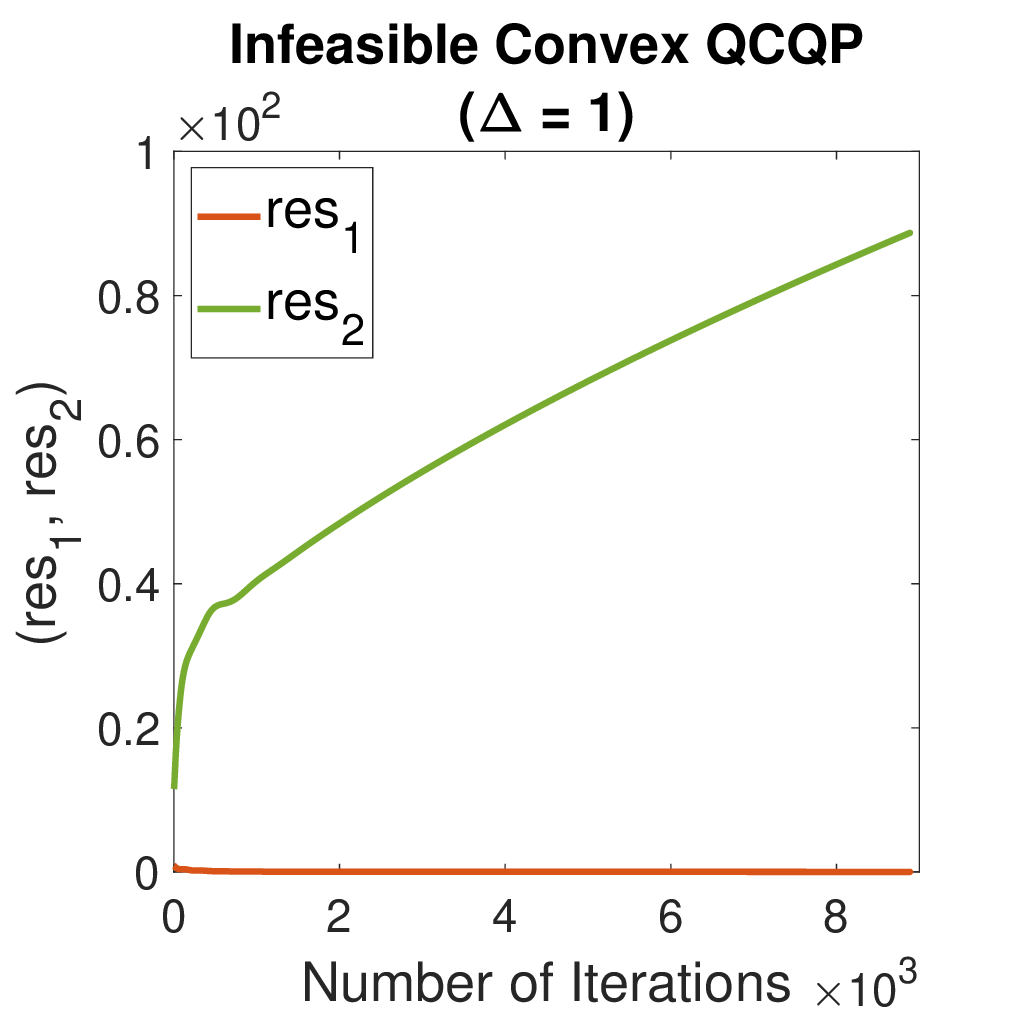}  
\end{subfigure}%
\begin{subfigure}{.33\textwidth}
\centering
\includegraphics[width=1.0\linewidth]{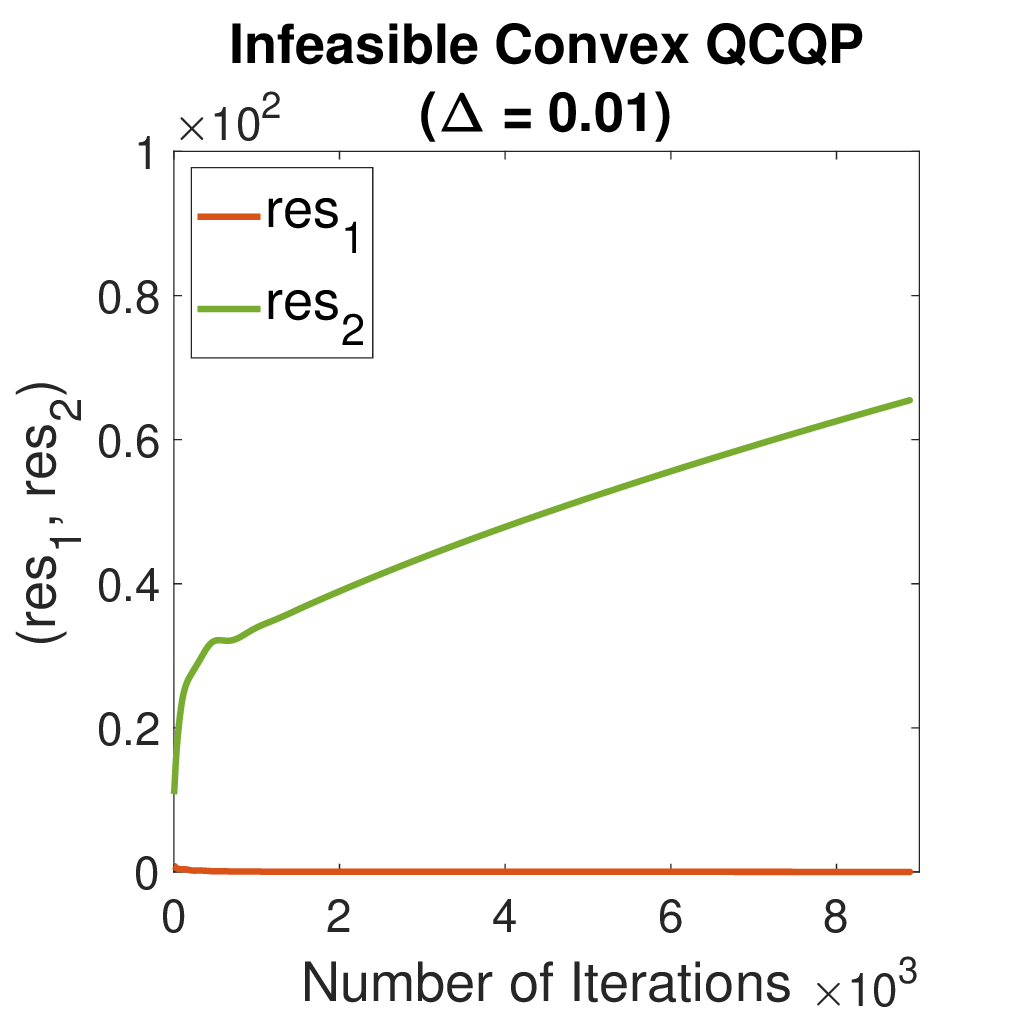}  
\end{subfigure}
\caption{Residuals of applying PC$^2$PM to solve an infeasible convex QCQP. The residual $res_2$ diverges when applying PC$^2$PM to solve \eqref{eq: 2-Constraint QCQP Problem Form}.}
\label{fig: res_infea}
\end{figure}
We observe that the residual $res_2$, measuring the complementarity and the primal feasibility as defined in \eqref{eq:res_2}, diverges for all $\Delta$'s, while the other residual $res_1$ still converges. This is a strong indication that the original problem is infeasible. 
%What really surprises us is that when we call the commercial solver CPLEX 12.8.0, which uses the barrier optimizer, to solve \eqref{eq: 2-Constraint QCQP Problem Form}, it actually returns a finite solution, which is impossible. When we simplify the infeasible constraint as $\frac{1}{2} \mathbf{x}^T \mathbf{x} + (r_2 + 1.0) + 100 \leq 0$, CPLEX then returns with a message stating that no solution exists.
%
\par Next, we construct an unbounded convex QCQP as follows:
\begin{equation}\label{eq: 1-Constraint QCQP Unbounded}
\begin{aligned}
\underset{\mathbf{x} \in \mathbb{R}^{n_1}}{\text{minimize}} \quad &\frac{1}{2} \mathbf{x}^{T} D_0 \mathbf{x} + \mathbf{e}_0^T \mathbf{x} + r_0 \\
\text{subject to} \quad &\frac{1}{2} \mathbf{x}^{T} D_0 \mathbf{x} + \mathbf{e}_1^T \mathbf{x} + r_1 \leq 0. \qquad (\lambda_1)
\end{aligned}
\end{equation}
The matrix $D_0 \in \mathbb{R}^{n_1 \times n_1}$ is a diagonal matrix with all but the last diagonal entry being $1$, and its last diagonal entry is set as 0, hence making it a PSD matrix. The vector $\mathbf{e}_0$ is an $n_1$-dimension vector in the form of $\mathbf{e}_0 = (0 \dots 0 \ 1)^T$. Conversely, the vector $\mathbf{e}_1$ is also $n_1$-dimension but in the form of $\mathbf{e}_1 = (1 \dots 1 \ 0)$. The dimension $n_1$ is also set as $1024$. All scalars $r_i$'s are randomly generated in the same way as in Section~\ref{subsec: DS}. It can be easily seen that the convex QCQP \eqref{eq: 1-Constraint QCQP Unbounded} is unbounded along the direction $(0, \dots, 0, -1)$. As shown in Fig~\ref{fig: res_unbound}, 
\begin{figure}[htb]
\centering
\includegraphics[scale = 0.32]{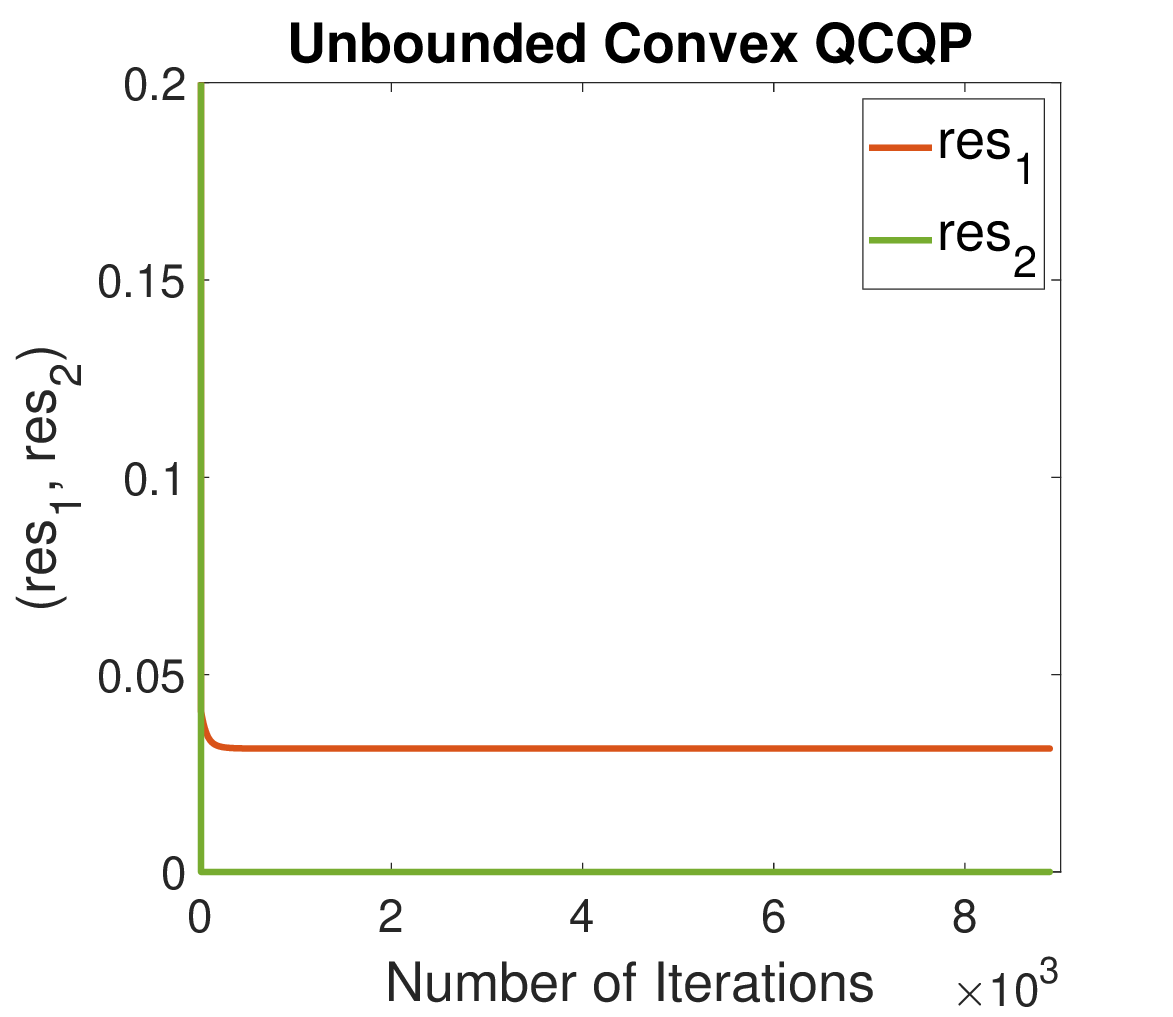}
\caption{Residuals of applying PC$^2$PM to solve an unbounded convex QCQP. The residual $res_1$ converges to a non-zero value when applying PC$^2$PM to solve \eqref{eq: 1-Constraint QCQP Unbounded}.}
\label{fig: res_unbound}
\end{figure}
when applying PC$^2$PM to solve \eqref{eq: 1-Constraint QCQP Unbounded}, we observe that both of the  residuals converge, but the residual $res_1$, measuring stationarity as defined in \eqref{eq:res_1}, converges to a non-zero value. This means that an optimal solution is not found. 
If it is known that a feasible point exists to a convex QCQP (as the example given by \eqref{eq: 1-Constraint QCQP Unbounded}) (and assume that a constraint qualification holds at the feasible point), then by Theorem \ref{thm: Converge}, if the algorithm does not find an optimal solution, it must mean that Assumption 1 is violated, which then implies that the original problem is unbounded (as an optimal solution does not exist). 
%CPLEX also detects the problem as dual infeasible.
%
%
%
%
%%%%%%%%%%%%%%%%%%%%%%%%%%%%%%%%%%%%%%%%%%%%%%%%%%%%%%%%%%%%%%%%%%%%%%%%%%%%%%%%
\section{Numerical Experiments}
In this section, we present more numerical results for solving high-dimension convex QCQPs using our algorithm. We first conduct numerical experiments of applying the PC$^2$PM algorithm to solve convex QCQPs of the standard form \eqref{eq: Standard QCQP Problem Form}, with randomly generated data sets of various sizes. We then solve convex QCQPs with explicit linear constraints as in \eqref{eq: QCQP Problem Form}, which naturally arise from multiple kernel learning applications. For both sets of experiments, we compare the performance of our algorithm with the current state-of-the-art commercial solver CPLEX 12.8.0, which uses the barrier optimizer for solving convex QCQPs. We implement PC$^2$PM with multiple compute nodes on Purdue University's Brown cluster using MPI, called from a C program. Each node on the cluster has two $12$-core Intel Xeon Gold ``Sky Lake" processors (that is, $24$ cores per node) and $96$ GB of memory. CPLEX 12.8.0 is also called using a C program and implemented on a single compute node (with $24$ cores). Note that CPLEX alone, as a centralized algorithm, cannot be run on multiple compute nodes using MPI, but it does allow multiple parallel threads that can be invoked by the barrier optimizer. More specifically, CPLEX has a parameter, CPXPARAM\_Threads, to call for multithread computing \cite{reference1987ibm}. When CPXPARAM\_Threads is set to be $1$, CPLEX is single threaded; when it is set to be $0$, CPLEX can use up to 32 threads, or the number of cores of the machine (with each core being a thread),  whichever is smaller. In our experiments, we always set CPXPARAM\_Threads as $0$, which gives CPLEX 24 threads (since each of our compute node has 24 cores).
\subsection{Solving Standard-Form Convex QCQPs}
We first apply PC$^2$PM to solve convex QCQPs of the standard form \eqref{eq: Standard QCQP Problem Form}, without the decision variables $\mathbf{u}$ or the explicit linear constraints $A \mathbf{x} + B \mathbf{u} - \mathbf{b} = \mathbf{0}$. The input data consist of matrix $P_i$, vector $\mathbf{q}_i$ and scalar $r_i$ for $i = 0, 1, \dots, m_1$, all of which are randomly generated in the same way as in Section~\ref{subsec: DS}. The decision variable's dimension $n_1$ is fixed as $2^{14} \approx 1.6 \times 10^4$, and the number of constraints $m_1$ increases from $1$ to $16$.
\par To balance between the computation speedup and communication overhead, we implement our algorithm with 128 cores allocated for primal variables' updating: \eqref{eq: x primal predictor update} \eqref{eq: u primal predictor update}, \eqref{eq: x primal corrector update} and \eqref{eq: u primal corrector update}, and 
$m_1$ (the number of quadratic constraints) cores for dual updating: \eqref{eq: dual predictor update} and \eqref{eq: dual corrector update}. The total number of compute nodes needed is calculated as $n_{\text{node}}$ = $\lceil \frac{n_{\text{core}}}{24} \rceil$ = $\lceil \frac{128 + m_1}{24} \rceil$. The stopping criteria we used are defined in \eqref{eq:res_1} and \eqref{eq:res_2}, with the tolerance $\tau^{\text{PC}^2\text{PM}}$ set to be $10^{-3}$. Table \ref{tab: Standard QCQP}
\begin{table}[htb]
\renewcommand{\arraystretch}{1.2}
\centering
\footnotesize
\setlength\tabcolsep{3.5pt}
\begin{tabular}{ccc|c|ccccc} \hline
\multirow{2}{*}{$\mathbf{n}_{\mathbf{1}}$}&\multirow{2}{*}{$\mathbf{m}_{\mathbf{1}}$}&\multirow{2}{*}{$\bm{\kappa}$}&\multirow{2}{*}{}&\multirow{2}{*}{\textbf{$\text{n}_{\text{node}}$}}&\multirow{2}{*}{\textbf{$\text{n}_{\text{core}}$}}&\textbf{mem./node}&\textbf{time}&\textbf{obj.} \\
&&&&&&(GB)&(hour)&\textbf{val.} \\ \hline
&\multirow{6}{*}{$\mathbf{1}$}&$10^2$&\multirow{2}{*}{\textbf{PC$^2$PM}}&\multirow{3}{*}{$6$}&\multirow{3}{*}{$128 + 1$}&\multirow{3}{*}{\textcolor{blue}{$1.6$/node}}&$3.94$&$-420.621$ \\
&&$10^4$&&&&&$4.96$&$-214.389$ \\
$1.6$&&$10^6$&$(\tau^{\text{PC}^2\text{PM}} = 10^{-3})$&&&&$6.95$&$-324.428$ \\ \cdashline{3-9}
$\times 10^4$&&$10^2$&\multirow{2}{*}{\textbf{CPLEX 12.8.0}}&\multirow{3}{*}{$1$}&\multirow{3}{*}{$24$}&\multirow{3}{*}{\textcolor{red}{$41.1$}}&$4.97$&$-420.645$ \\
&&$10^4$& &&&&$5.02$&$-214.423$ \\
&&$10^6$&$(\tau^{\text{Barrier}} = 10^{-3})$&&&&$5.25$&$-324.465$ \\ \hline
%%%%%%%%%%%%%%%%%%%%%%%%%%%%%%%%%%%%%%%%%%%%%%%%%%%%%%%%%%%%%%%%%%%%%%%%%%%%%%%%%%%%%%%%%%%%%%%%
&\multirow{6}{*}{$\mathbf{2}$}&$10^2$&\multirow{2}{*}{\textbf{PC$^2$PM}}&\multirow{3}{*}{$6$}&\multirow{3}{*}{$128 + 2$}&\multirow{3}{*}{\textcolor{blue}{$1.9$/node}}&\textcolor{blue}{$2.28$}&$-322.232$ \\
&&$10^4$&&&&&\textcolor{blue}{$2.41$}&$-161.960$ \\
$1.6$&&$10^6$&$(\tau^{\text{PC}^2\text{PM}} = 10^{-3})$&&&&\textcolor{blue}{$1.87$}&$-244.048$ \\ \cdashline{3-9}
$\times 10^4$&&$10^2$&\multirow{2}{*}{\textbf{CPLEX 12.8.0}}&\multirow{3}{*}{$1$}&\multirow{3}{*}{$24$}&\multirow{3}{*}{\textcolor{red}{$74.0$}}&\textcolor{red}{$10.82$}&$-322.213$ \\
&&$10^4$& &&&&\textcolor{red}{$10.83$}&$-161.910$ \\
&&$10^6$&$(\tau^{\text{Barrier}} = 10^{-3})$&&&&\textcolor{red}{$10.35$}&$-244.035$ \\ \hline
%%%%%%%%%%%%%%%%%%%%%%%%%%%%%%%%%%%%%%%%%%%%%%%%%%%%%%%%%%%%%%%%%%%%%%%%%%%%%%%%%%%%%%%%%%%%%%%%
&\multirow{6}{*}{$\mathbf{4}$}&$10^2$&\multirow{2}{*}{\textbf{PC$^2$PM}}&\multirow{3}{*}{$6$}&\multirow{3}{*}{$128 + 4$}&\multirow{3}{*}{\textcolor{blue}{$2.6$/node}}&\textcolor{blue}{$1.46$}&$-243.154$ \\
&&$10^4$&&&&&\textcolor{blue}{$1.23$}&$-126.184$ \\
$1.6$&&$10^6$&$(\tau^{\text{PC}^2\text{PM}} = 10^{-3})$&&&&\textcolor{blue}{$1.46$}&$-189.230$ \\  \cdashline{3-9}
$\times 10^4$&&$10^2$&\multirow{2}{*}{\textbf{CPLEX 12.8.0}}&\multirow{3}{*}{$1$}&\multirow{3}{*}{$24$}&\multirow{3}{*}{\textcolor{red}{O.O.M.}}&\multirow{3}{*}{\textcolor{red}{N.A.}}&\multirow{3}{*}{N.A.} \\
&&$10^4$& &&&&& \\
&&$10^6$&$(\tau^{\text{Barrier}} = 10^{-3})$&&&\textcolor{red}{($> 96$ GB)}&& \\ \hline
%%%%%%%%%%%%%%%%%%%%%%%%%%%%%%%%%%%%%%%%%%%%%%%%%%%%%%%%%%%%%%%%%%%%%%%%%%%%%%%%%%%%%%%%%%%%%%%%
&\multirow{6}{*}{$\mathbf{8}$}&$10^2$&\multirow{2}{*}{\textbf{PC$^2$PM}}&\multirow{3}{*}{$6$}&\multirow{3}{*}{$128 + 8$}&\multirow{3}{*}{\textcolor{blue}{$4.3$/node}}&\textcolor{blue}{$2.13$}&$-189.945$ \\
&&$10^4$&&&&&\textcolor{blue}{$1.33$}&$-97.974$ \\
$1.6$&&$10^6$&$(\tau^{\text{PC}^2\text{PM}} = 10^{-3})$&&&&\textcolor{blue}{$1.54$}&$-144.916$ \\  \cdashline{3-9}
$\times 10^4$&&$10^2$&\multirow{2}{*}{\textbf{CPLEX 12.8.0}}&\multirow{3}{*}{$1$}&\multirow{3}{*}{$24$}&\multirow{3}{*}{\textcolor{red}{O.O.M.}}&\multirow{3}{*}{\textcolor{red}{N.A.}}&\multirow{3}{*}{N.A.} \\
&&$10^4$& &&&&& \\
&&$10^6$&$(\tau^{\text{Barrier}} = 10^{-3})$&&&\textcolor{red}{($> 96$ GB)}&& \\ \hline
%%%%%%%%%%%%%%%%%%%%%%%%%%%%%%%%%%%%%%%%%%%%%%%%%%%%%%%%%%%%%%%%%%%%%%%%%%%%%%%%%%%%%%%%%%%%%%%%
&\multirow{6}{*}{$\mathbf{16}$}&$10^2$&\multirow{2}{*}{\textbf{PC$^2$PM}}&\multirow{3}{*}{$6$}&\multirow{3}{*}{$128 + 16$}&\multirow{3}{*}{\textcolor{blue}{$7.7$/node}}&\textcolor{blue}{$3.72$}&$-147.310$ \\
&&$10^4$&&&&&\textcolor{blue}{$2.34$}&$-74.854$ \\
$1.6$&&$10^6$&$(\tau^{\text{PC}^2\text{PM}} = 10^{-3})$&&&&\textcolor{blue}{$3.04$}&$-111.490$ \\  \cdashline{3-9}
$\times 10^4$&&$10^2$&\multirow{2}{*}{\textbf{CPLEX 12.8.0}}&\multirow{3}{*}{$1$}&\multirow{3}{*}{$24$}&\multirow{3}{*}{\textcolor{red}{O.O.M.}}&\multirow{3}{*}{\textcolor{red}{N.A.}}&\multirow{3}{*}{N.A.} \\
&&$10^4$& &&&&& \\
&&$10^6$&$(\tau^{\text{Barrier}} = 10^{-3})$&&&\textcolor{red}{($> 96$ GB)}&& \\ \hline
\end{tabular}
\caption{Comparison of PC$^2$PM with CPLEX 12.8.0 for solving standard-form, high-dimension convex QCQPs.}
\label{tab: Standard QCQP}
\end{table}
reports the elapsed wall-clock time used by the PC$^2$PM algorithm, along with the amount of memory used by each compute node and the final objective function value, with respect to the increasing condition number $\kappa$. The performance of CPLEX 12.8.0 with the same convergence tolerance is also presented in Table \ref{tab: Standard QCQP} for comparison. In the first two groups of tests with $m_1 = 1 \text{ and } 2$, our algorithm compares favorably to CPLEX and uses much less memory. For the rest groups of test cases, CPLEX fails to provide a solution (actually fails to complete even a single iteration) due to running out of memory; while PC$^2$PM still converges within a reasonable amount of time. As the scale of the problem increases, our algorithm exhibits favorable scalability, due to its distributed storage of data and the capability of massively parallel computing. Another interesting observation from Table \ref{tab: Standard QCQP}, though we do not know the underlying reason, is that when the number of quadratic constraints ($m_1$) is small, PC$^2$PM's run time appears to be sensitive to the condition number of matrices (i.e., the Hessian matrices of the objective function and the constraints); yet when $m_1$ becomes larger, the effect of condition numbers on the run time appears to be subdued.
\par We also plot the two residuals $res_1^k$ and $res_2^k$ in Fig~\ref{fig: Residuals},
\begin{figure}[htb]
\centering
\begin{subfigure}{.33\textwidth}
\centering
\includegraphics[width=1.0\linewidth]{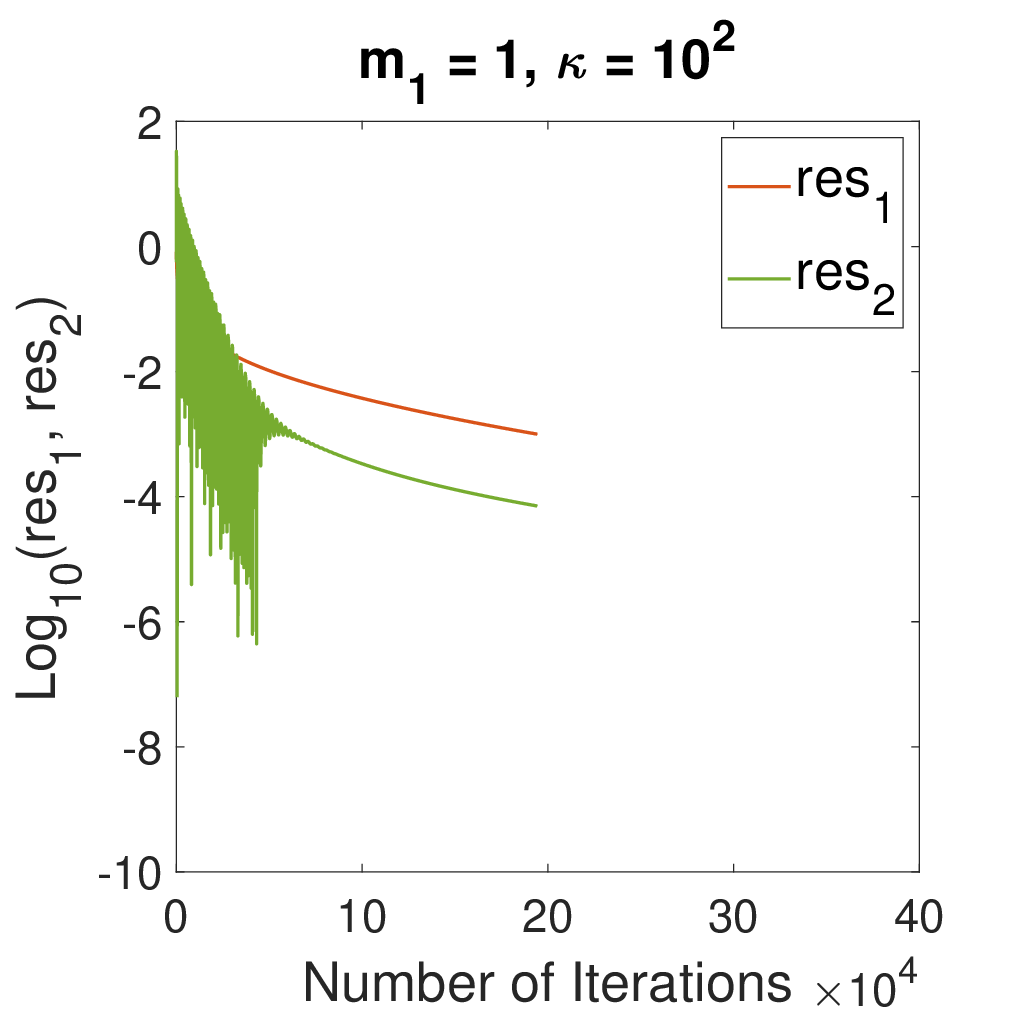}
\end{subfigure}%
\begin{subfigure}{.33\textwidth}
\centering
\includegraphics[width=1.0\linewidth]{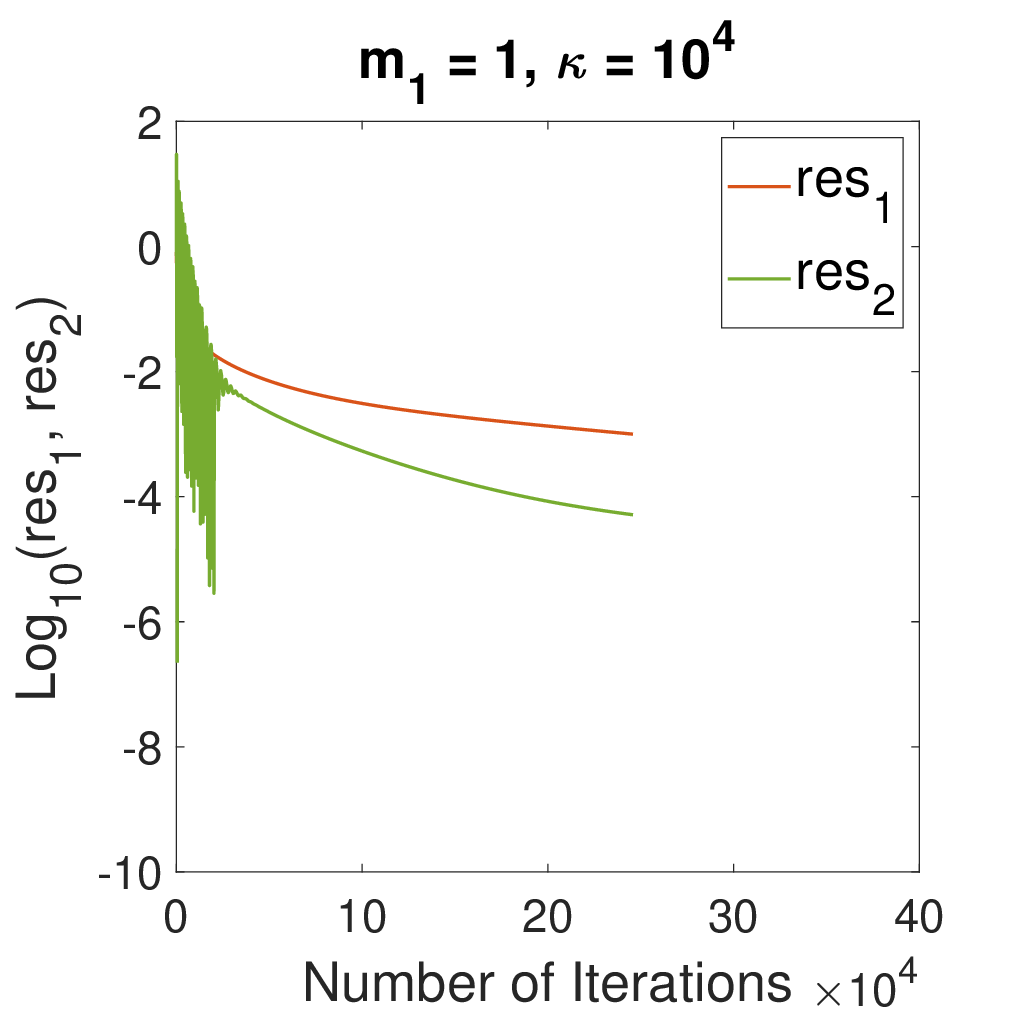}
\end{subfigure}%
\begin{subfigure}{.33\textwidth}
\centering
\includegraphics[width=1.0\linewidth]{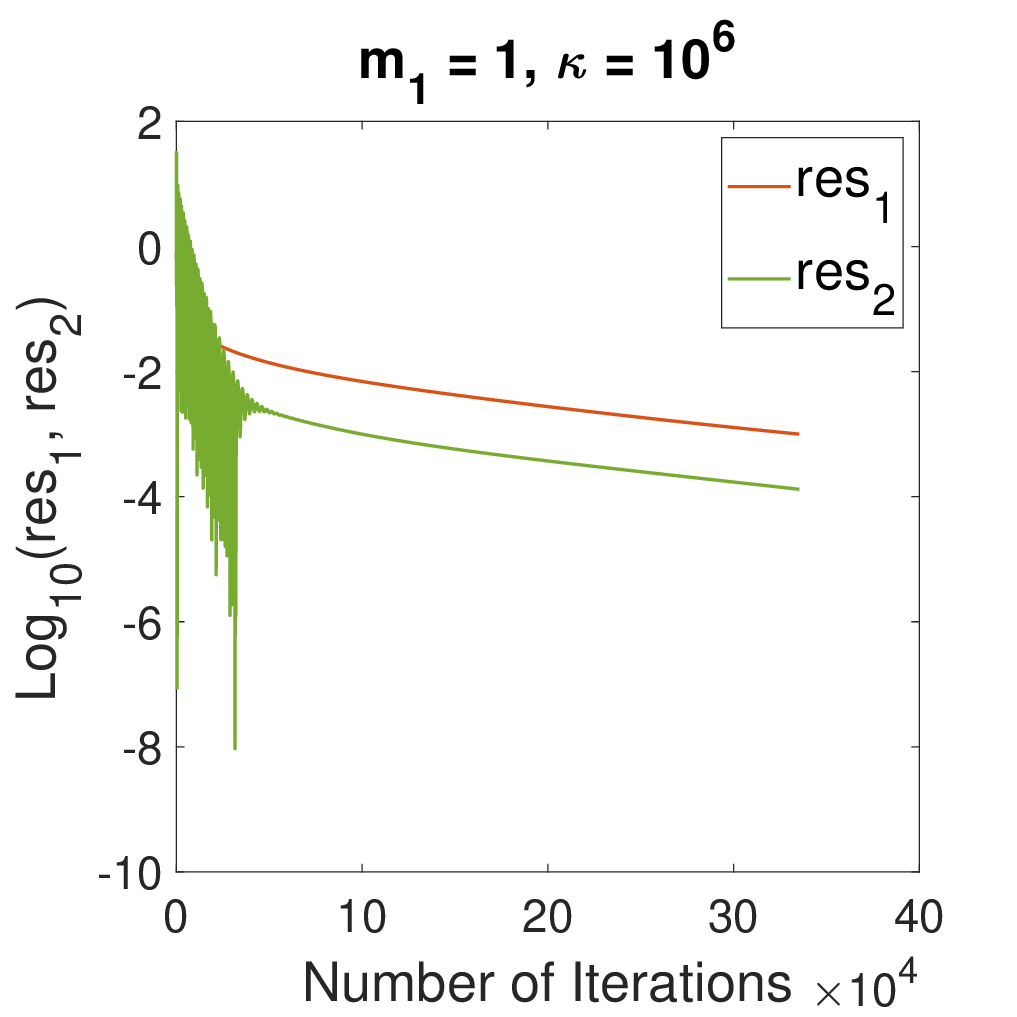}
\end{subfigure}
\begin{subfigure}{.33\textwidth}
\centering
\includegraphics[width=1.0\linewidth]{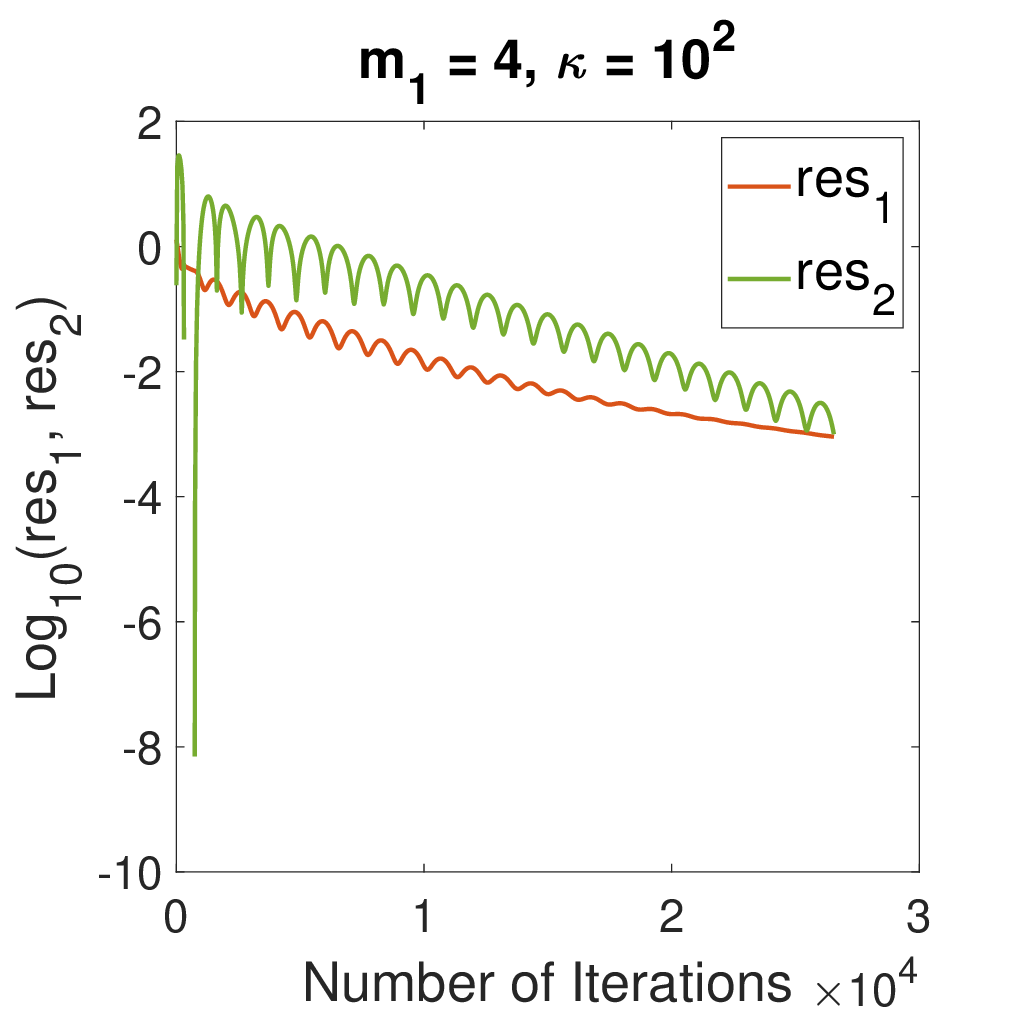}
\end{subfigure}%
\begin{subfigure}{.33\textwidth}
\centering
\includegraphics[width=1.0\linewidth]{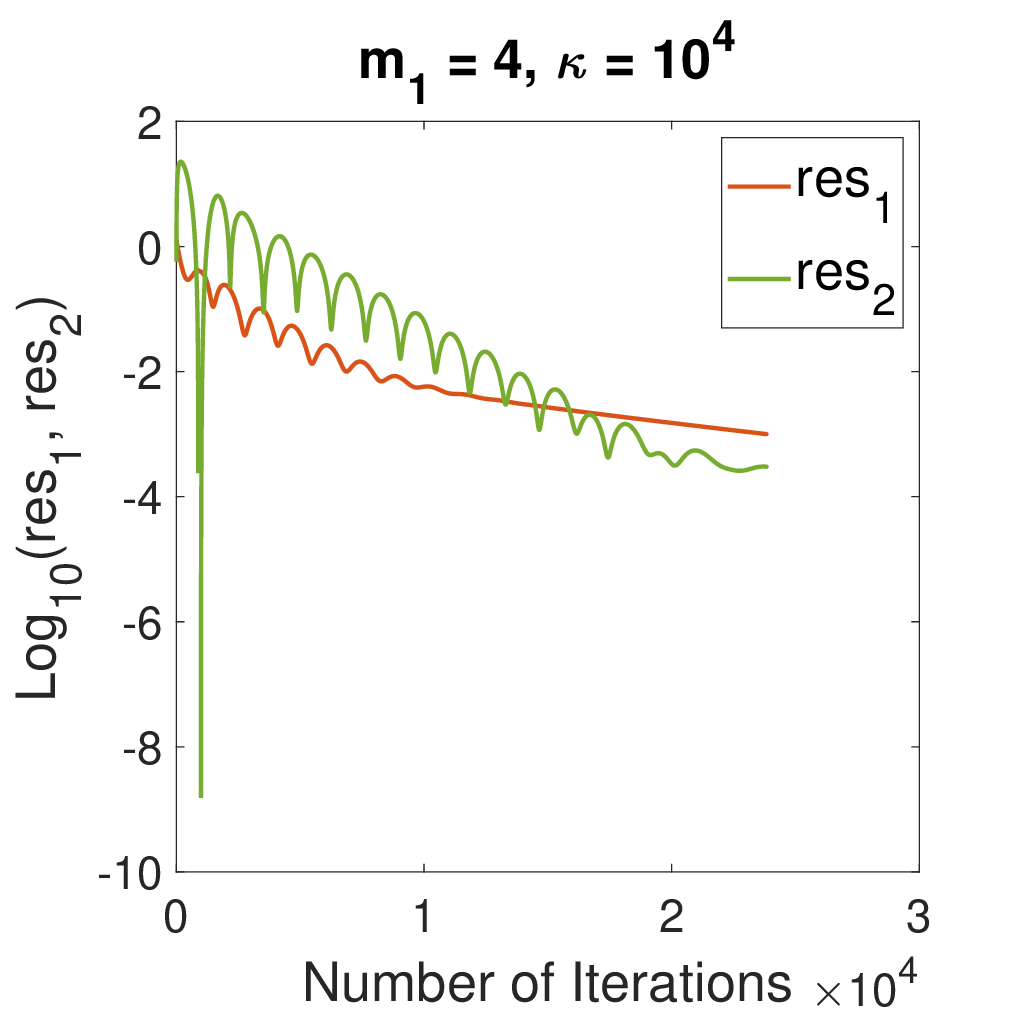}
\end{subfigure}%
\begin{subfigure}{.33\textwidth}
\centering
\includegraphics[width=1.0\linewidth]{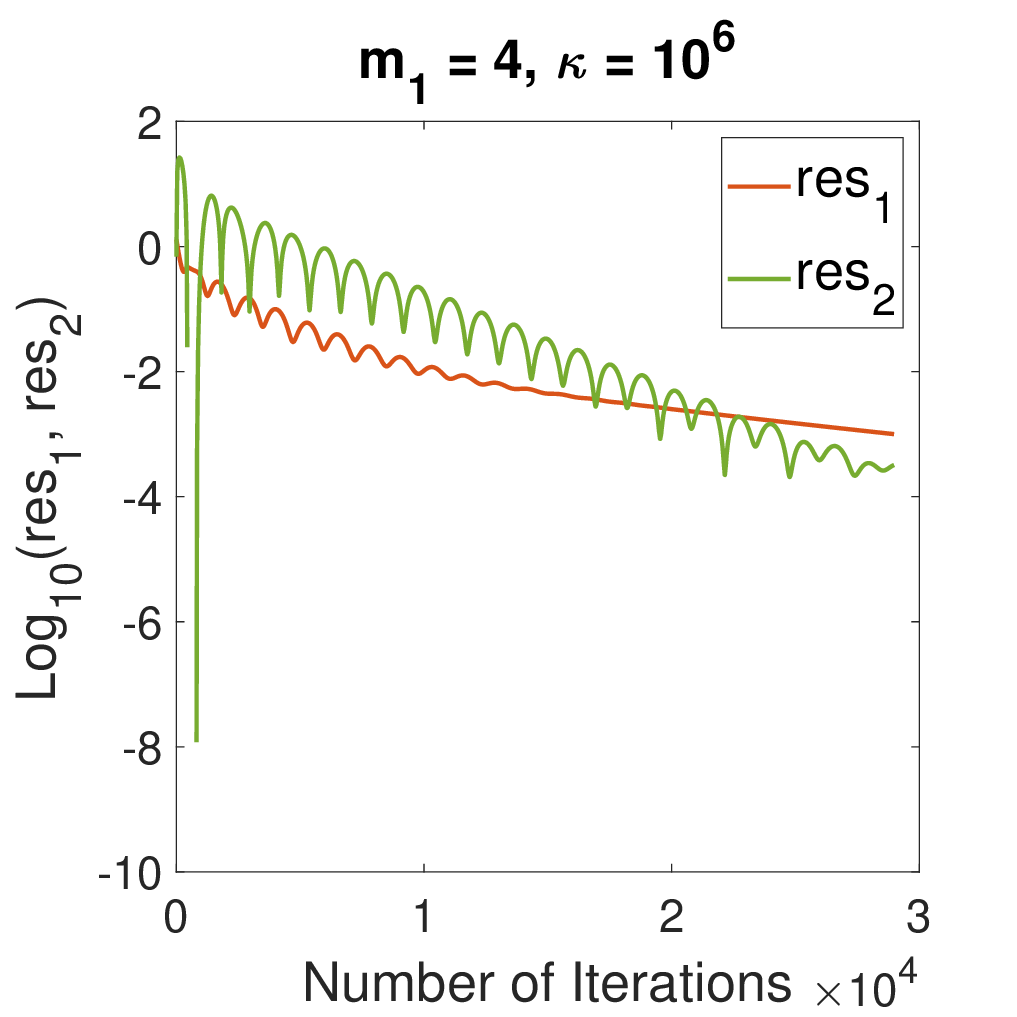}
\end{subfigure}
\begin{subfigure}{.33\textwidth}
\centering
\includegraphics[width=1.0\linewidth]{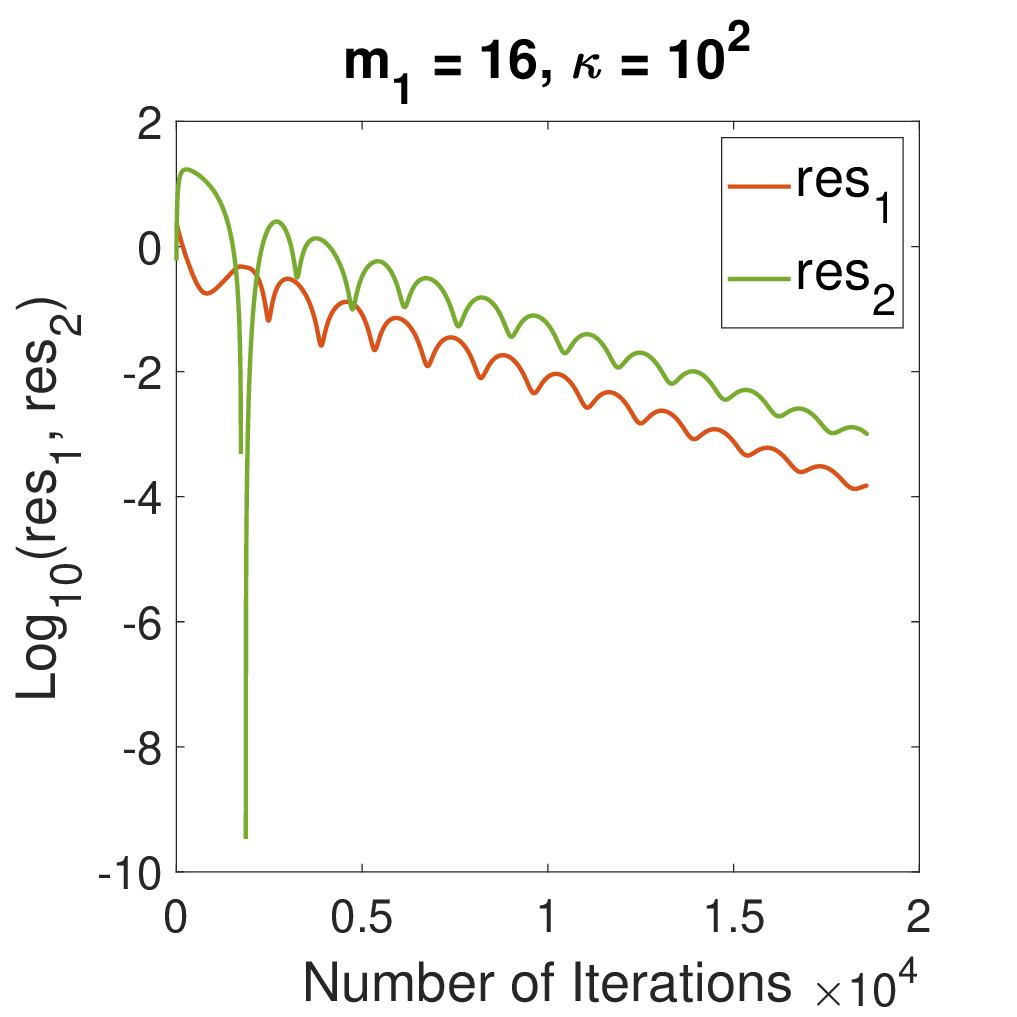}
\end{subfigure}%
\begin{subfigure}{.33\textwidth}
\centering
\includegraphics[width=1.0\linewidth]{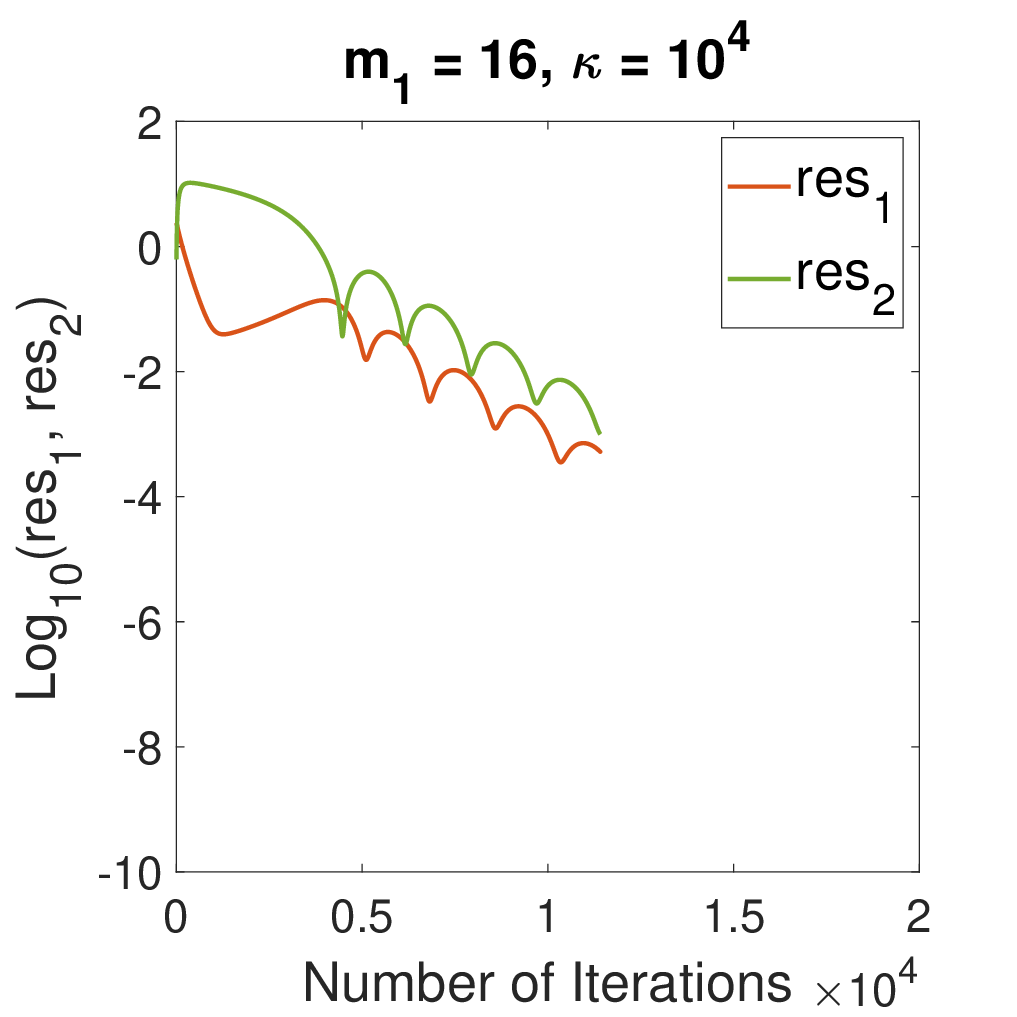}
\end{subfigure}%
\begin{subfigure}{.33\textwidth}
\centering
\includegraphics[width=1.0\linewidth]{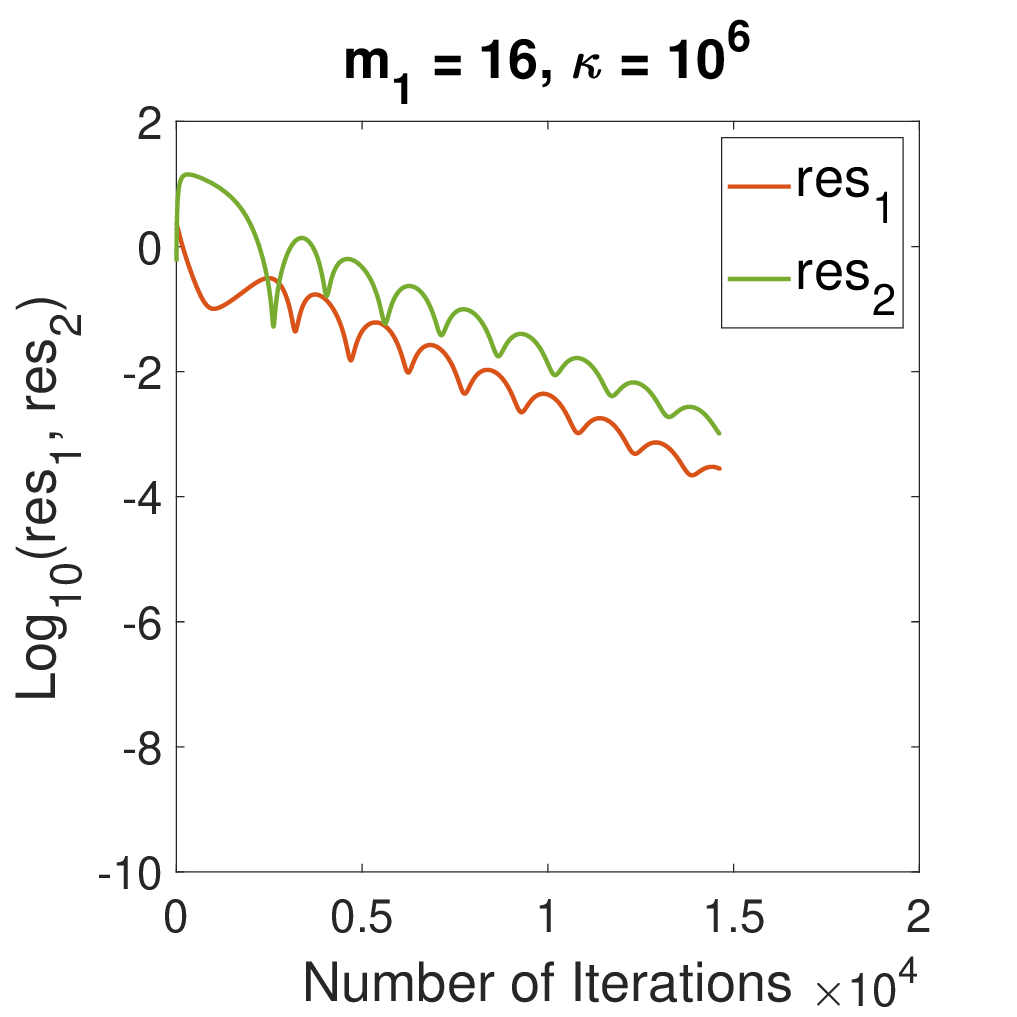}
\end{subfigure}
\caption{Convergence of residuals.}
\label{fig: Residuals}%\vspace*{-10pt}
\end{figure}
with $res_1^k$ corresponding to the gradient of the Lagrangian function, and $res_2^k$ corresponding to the feasibility and complementarity conditions.
The three plots in a same row are with the same number of constraints $m_1$, but with different condition numbers of the Hessian matrices. 
As seen in Fig~\ref{fig: Residuals}, from left to right, when $m_1$ is small, as the condition number $\kappa$ increases, more iterations are required for the PC$^2$PM algorithm to converge; yet when $m_1$ becomes larger, the number of iterations depends more on the absolute value of the objective function than the condition number. Another observation is that when the number of constraints increases (i.e., from top to bottom), the convergence of the residuals becomes more smooth.
\subsection{Multiple Kernel Learning in Support Vector Machine}
In this subsection, we briefly introduce how the Support Vector Machine (SVM) with multiple kernel learning can be formulated as a convex QCQP, and present numerical results of applying our algorithm to solve high-dimension instances. As discussed in \cite{hastie2009elements}, SVM is a discriminative classifier proposed for binary classification problems. Given a set of $n_{tr}$ pairs of independently and identically distributed training data points $\{(\mathbf{d}_j, l_j)\}_{j=1}^{n_{tr}}$, where $\mathbf{d}_j \in \mathbb{R}^{n_d}$ is the $n_d$-dimension input vector and $l_j \in \{-1, 1\}$ is its class label, SVM searches for a hyperplane that can best separate the points from two classes. The hyperplane is defined as $\{\mathbf{d} \in \mathbb{R}^{n_d} | f(\mathbf{d}) = \bm{\beta}^T \mathbf{d} + \beta_0 = 0\}$,  where $\bm{\beta} \in \mathbb{R}^{n_d}$ is a unit vector with $\lVert \bm{\beta} \rVert_2 = 1$, and $\beta_0 \in \mathbb{R}$ is a scalar. The points belonging to either class should be separated as far away from the hyperplane as possible, while still remain on the correct side. When the data points cannot be clearly separated in the original space $\mathbb{R}^{n_d}$, we instead search in a feature space $\mathbb{R}^{n_f}$, by mapping the input data space $\mathbb{R}^{n_d}$ to the feature space through a function $\Phi: \mathbb{R}^{n_d} \to \mathbb{R}^{n_f}$. For example, a 2-dimension data space can be lifted to a 3-dimension feature space. Using the function $\Phi$, we can define  a \textit{kernel function} $k: \mathbb{R}^{n_d} \times \mathbb{R}^{n_d} \to \mathbb{R}$ as $k(\mathbf{d}, \mathbf{d}^{\prime}) \coloneqq \ \langle\Phi(\mathbf{d}),\ \Phi(\mathbf{d}^{\prime})\rangle$ for any $\mathbf{d}, \mathbf{d}^{\prime} \in \mathbb{R}^{n_d}$, where $\langle,\rangle$ denotes an inner product. The resulting discriminant function $\mathcal{G}: \mathbb{R}^{n_d} \to \{-1, 1\}$, which the SVM searches for, can be expressed as:
\begin{equation}\label{eq: Induction_SKL}
\mathcal{G}(\mathbf{d}) = \text{sign}\Big(\sum_{j=1}^{n_{tr}} \alpha_j l_j k(\mathbf{d}_j, \mathbf{d}) + b\Big), \quad \forall \mathbf{d} \in \mathbb{R}^{n_d},
\end{equation}
where $\bm{\alpha} \equiv (\alpha_1, \ldots, \alpha_{n_{tr}})^T$ is the weight vector and $b$ is the bias. The popular choices of kernel functions in the SVM literature include the linear kernel function $k_{LIN}$, the polynomial kernel function $k_{POL}$ and the Gaussian kernel function $k_{GAU}$:
\begin{subequations}
\begin{align}
&k_{LIN}(\mathbf{d}, \mathbf{d}^{\prime}) \coloneqq \mathbf{d}^T \mathbf{d}^{\prime}, \quad \forall \mathbf{d}, \mathbf{d}^{\prime} \in \mathbb{R}^{n_d} \label{eq: Linear_kernel} \\ 
&k_{POL}(\mathbf{d}, \mathbf{d}^{\prime}) \coloneqq (1 + \mathbf{d}^T \mathbf{d}^{\prime})^2, \quad \forall \mathbf{d}, \mathbf{d}^{\prime} \in \mathbb{R}^{n_d} \label{eq: Polynomial_kernel} \\
&k_{GAU}(\mathbf{d}, \mathbf{d}^{\prime})) \coloneqq e^{-\frac{\lVert \mathbf{d} - \mathbf{d}^{\prime} \rVert_2^2}{2 \sigma^2}}, \quad \sigma > 0, \forall \mathbf{d}, \mathbf{d}^{\prime} \in \mathbb{R}^{n_d}. \label{eq: Gaussian_kernel}
\end{align}
\end{subequations}
\par Instead of using a single kernel function, \cite{lanckriet2004learning} explores SVM using a kernel function that can be expressed as a non-negative combination of a pre-specified set of kernel functions $\{k_1, \dots, k_m\}$, with the non-negative coefficients $\lambda_1, \dots, \lambda_m$ to be allocated; that is, $k(\mathbf{d}, \mathbf{d}^{\prime}) = \sum_{i=1}^m \lambda_i k_i(\mathbf{d}, \mathbf{d}^{\prime})$ for any $\mathbf{d}, \mathbf{d}^{\prime} \in \mathbb{R}^{n_d}$ with $\lambda_1, \ldots, \lambda_m\geq 0$. The allocation process can be expressed as solving a convex QCQP, where each $\lambda_i$ is the Lagrangian multiplier corresponding to each quadratic constraint. The formulation of the convex QCQP, as provided in \cite{lanckriet2004learning}, is as follows:
\begin{enumerate}[label=(\roman*)]
\item \textbf{1-norm Soft Margin SVM} learns the coefficients through solving the following convex QCQP:
\begin{equation}\label{eq: SM1 MKL SVM Dual}
\begin{aligned}
\underset{\bm{\alpha} \in \mathbb{R}^{n_{tr}}, \alpha_0 \in \mathbb{R}}{\text{minimize}} \quad &-\mathbf{e}^T \bm{\alpha} + R \alpha_0 \\
\text{subject to} \quad &\frac{1}{2} \bm{\alpha}^T \big[\frac{1}{R_i} G_i(K_{i, tr})\big] \bm{\alpha} - \alpha_0 \leq 0, \quad i = 1, \dots, m, \qquad (\lambda_i) \\
&\sum_{j=1}^{n_{tr}} l_j \alpha_j = 0, \qquad (\gamma) \\
&0 \leq \alpha_j \leq C, \quad j = 1, \dots, n_{tr},
\end{aligned}
\end{equation}
\item \textbf{2-norm Soft Margin SVM} learns the coefficients through solving the following convex QCQP:
\begin{equation}\label{eq: SM2 MKL SVM Dual}
\begin{aligned}
\underset{\bm{\alpha} \in \mathbb{R}_+^{n_{tr}}, \alpha_0 \in \mathbb{R}}{\text{minimize}} \quad &\frac{1}{2} \bm{\alpha}^T \big[\frac{1}{C} I_{n_{tr}}\big] \bm{\alpha} - \mathbf{e}^T \bm{\alpha} + R \alpha_0 \\
\text{subject to} \quad &\frac{1}{2} \bm{\alpha}^T \big[\frac{1}{R_i} G_i(K_{i, tr})\big] \bm{\alpha} - \alpha_0 \leq 0, \quad i = 1, \dots, m, \qquad (\lambda_i) \\
&\sum_{j=1}^{n_{tr}} l_j \alpha_j = 0, \qquad (\gamma)
\end{aligned}
\end{equation}
\end{enumerate}
where the vector $\mathbf{e}$ denotes an $n_{tr}$-dimensional vector of all ones. Given a labeled training data set $\mathcal{S}_{tr} = \{(\mathbf{d}_j, l_j)\}_{j=1}^{n_{tr}}$ and an unlabeled test data set $\mathcal{S}_t = \{\mathbf{d}_j\}_{j=1}^{n_t}$, a matrix $K_i \in \mathbb{R}^{(n_{tr} + n_t) \times (n_{tr} + n_t)}$ can be defined on the entire data set $\mathcal{S}_{tr} \cup \mathcal{S}_t$ as
\begin{equation}\label{eq: K_matrix}
K_i \coloneqq \left(
\begin{array}{cc}
K_{i, tr} & K_{i, (tr,t)} \\
K_{i, (tr,t)}^T & K_{i, t}
\end{array}
\right).
\end{equation}
The submatrix $K_{i, tr} \in \mathbb{R}^{n_{tr} \times n_{tr}}$ is a square symmetric matrix, whose $j j^{\prime}$-th element is directly defined by a kernel function: $[K_{i, tr}]_{j j^{\prime}} \coloneqq k_i(\mathbf{d}_j, \mathbf{d}_{j^{\prime}})$ for any $\mathbf{d}_j, \mathbf{d}_{j^{\prime}}$ in $\mathcal{S}_{tr}$. The submatrices $K_{i, (tr, t)} \in \mathbb{R}^{n_{tr} \times n_t}$ and $K_{i, t} \in \mathbb{R}^{n_t \times n_t}$ are defined in the same way but with different input vectors. The matrix $G_i(K_{i, tr}) \in \mathbb{R}^{n_{tr} \times n_{tr}}$ in the quadratic constraint of \eqref{eq: SM1 MKL SVM Dual} and \eqref{eq: SM2 MKL SVM Dual} is a square symmetric matrix with its $j j^{\prime}$-th element being $[G_i(K_{i,tr})]_{j j^{\prime}} = l_j l_{j^{\prime}} [K_{i,tr}]_{j j^{\prime}}$. Note that each kernel matrix $K_{i, tr}$ is a symmetric PSD matrix (see Proposition 2 in \cite{lanckriet2004learning}), then each $G_i(K_{i,tr})$ is also a symmetric PSD matrix, since $G_i(K_{i,tr}) = L K_{i,tr} L$, where $L \coloneqq diag(l_1, \dots, l_{n_{tr}})$. Let $R_i$ denote $\text{trace}(K_i)$ for $i = 1, \dots, m$, and $R = \sum_{i=1}^m \lambda_i R_i$ can be fixed as a given number. The parameter $C$ is a fixed positive scalar from the soft margin criteria.
\par Once the optimal primal-dual solution $(\bm{\alpha}^*; \lambda_1^*, \dots, \lambda_m^*)$ is found from either \eqref{eq: SM1 MKL SVM Dual} or \eqref{eq: SM2 MKL SVM Dual}, combining with those pre-specified $k_i$'s, it can be used to label the test data set according to the following discriminant function $\mathcal{G}_{\text{MKL}}: \mathbb{R}^{n_d} \to \{-1, 1\}$:
\begin{equation}\label{eq: Induction_MKL}
\mathcal{G}_{\text{MKL}}(\mathbf{d}_{j^{\prime}}) = \text{sign}\Big(\sum_{j=1}^{n_{tr}} \alpha_j^* l_j \big[\sum_{i=1}^m \lambda_i^* k_i(\mathbf{d}_j, \mathbf{d}_{j^{\prime}})\big] + b\Big), \quad \forall \mathbf{d}_{j^{\prime}} \in \mathcal{S}_t.
\end{equation}
Compared with \eqref{eq: Induction_SKL}, the only difference is the replacement of a non-negative combination of $k_i$'s with coefficients $\lambda_1^*, \dots, \lambda_m^*$. The test set accuracy (TSA) can then be obtained by measuring the percentage of the test data points accurately labeled according to the function \eqref{eq: Induction_MKL}.
%
%Data Set
\par The formulation  \eqref{eq: SM1 MKL SVM Dual} and \eqref{eq: SM2 MKL SVM Dual} provide instances of convex QCQPs in the form of \eqref{eq: QCQP Problem Form}, and we apply the PC$^2$PM to solve them.
The first input data set we used is the \textit{Two-norm Problem} from \cite{breiman1998arcing}, which is also used in \cite{lanckriet2004learning}; however, our data set has a much larger size than in \cite{lanckriet2004learning}. We first generate $8,000$ data points, with each data point being a $20$-dimension vector, drawn from a multivariate normal distribution with a unit covariance matrix and the mean of $(a, \dots, a)$. These data points form the first class that are all labeled with $1$. Another $8,000$ points of $20$-dimension vectors are drawn from another multivariate normal distribution with also a unit covariance matrix but the mean of $(-a, \dots, -a)$. They form the second class that are all labeled with $-1$. The value of $a$ is set as $\frac{2}{\sqrt{20}}$, the same as in \cite{breiman1998arcing}. Together, these two classes of data points form our first input data set with the size of $8000 + 8000 = 16, 000$. The second input data set is the \textit{HEPMASS Data Set} from the UCI Repository\footnote{\href{https://archive.ics.uci.edu/ml/datasets/HEPMASS}{https://archive.ics.uci.edu/ml/datasets/HEPMASS}}. This data set is used in high-energy physics experiments for learning particle-producing collisions from a background source. Each data point is generated from Monte Carlo simulations of collisions, and has $28$ attributes. We randomly selected $16,000$ data points from the original $10,500,000$-sized data set as our inputs.
%Result_1
\par We use a set of pre-specified kernel functions $\{k_1, \dots, k_5\}$ that contains all Gaussian kernel functions defined in \eqref{eq: Gaussian_kernel} whose $\sigma^2$ equal to $0.01$, $0.1$, $1$, $10$ and $100$ respectively. Each matrix $K_i$ is normalized and $R_i = \text{trace}(K_i)$ is set to be $1.0$ for $i = 1, \dots, 5$. Then $R = \sum_{i=1}^5 \lambda_i R_i = \sum_{i=1}^5 \lambda_i$, is restricted to be $5.0$. The value of the parameter $C$ is fixed as $1.0$ for 2-norm soft margin SVMs, and is set as $3.0$ for Two-norm Problem and $5.0$ for HEPMASS Data Set when using 1-norm soft margin SVMs. Numerical results of both $1$-norm and $2$-norm soft margin SVMs using the above five kernel functions are summarized in Table~\ref{tab: MKL 1}.
\begin{table}[htb]
\renewcommand{\arraystretch}{1.2}
\centering
\footnotesize
\setlength\tabcolsep{2.3pt}
\begin{subtable}[c]{1.0\textwidth}
\begin{tabular}{c|c|cccccccc} \hline
\multicolumn{10}{c}{\textbf{Two-norm Problem}} \\ \hline
\textbf{SVM}&&\textbf{mem./node}&\textbf{time}&\multirow{2}{*}{$\lambda_1^*$}&\multirow{2}{*}{$\lambda_2^*$}&\multirow{2}{*}{$\lambda_3^*$}&\multirow{2}{*}{$\lambda_4^*$}&\multirow{2}{*}{$\lambda_5^*$}&\textbf{TSA} \\
\textbf{Criteria}&&(GB)&(hour)&&&&&&($\%$) \\ \hline
SM1&\textbf{PC$^2$PM}&\textcolor{blue}{$2.1$/node}&$6.06$&$0.000$&$0.000$&$0.000$&$6.543$&$0.000$&$97.84$ \\ \cdashline{2-10}
$C = 3.0$&\textbf{CPLEX 12.8.0}&\textcolor{red}{O.O.M. ($> 96$)}&N.A.&N.A.&N.A.&N.A.&N.A.&N.A.&N.A. \\ \hline
SM2&\textbf{PC$^2$PM}&\textcolor{blue}{$2.0$/node}&\textcolor{blue}{$0.72$}&$0.000$&$0.000$&$0.000$&$5.005$&$0.000$&$97.83$ \\ \cdashline{2-10}
$C = 1.0$&\textbf{CPLEX 12.8.0}&\textcolor{red}{$73.6$}&\textcolor{red}{$3.09$}&$0.001$&$0.000$&$0.000$&$4.997$&$0.002$&$97.83$ \\ \hline
\end{tabular}
\caption{For 1-norm soft margin SVM, we let PC$^2$PM converge with $res_1 < 0.015$ instead of $10^{-3}$, while still keep $res_2 < 10^{-3}$.}
\vspace{10pt}
\end{subtable}
\quad% 
\begin{subtable}[c]{1.0\textwidth}
\begin{tabular}{c|c|cccccccc} \hline
\multicolumn{10}{c}{\textbf{HEPMASS Data Set}} \\ \hline
\textbf{SVM}&&\textbf{mem./node}&\textbf{time}&\multirow{2}{*}{$\lambda_1^*$}&\multirow{2}{*}{$\lambda_2^*$}&\multirow{2}{*}{$\lambda_3^*$}&\multirow{2}{*}{$\lambda_4^*$}&\multirow{2}{*}{$\lambda_5^*$}&\textbf{TSA} \\
\textbf{Criteria}&&(GB)&(hour)&&&&&&($\%$) \\ \hline
SM1&\textbf{PC$^2$PM}&\textcolor{blue}{$2.1$/node}&$6.51$&$0.000$&$0.000$&$0.000$&$0.000$&$6.992$&$76.77$ \\ \cdashline{2-10}
$C = 5.0$&\textbf{CPLEX 12.8.0}&\textcolor{red}{O.O.M. ($> 96$)}&N.A.&N.A.&N.A.&N.A.&N.A.&N.A.&N.A. \\ \hline
SM2&\textbf{PC$^2$PM}&\textcolor{blue}{$2.0$/node}&\textcolor{blue}{$0.17$}&$0.000$&$0.000$&$0.000$&$0.000$&$5.088$&$78.43$ \\ \cdashline{2-10}
$C = 1.0$&\textbf{CPLEX 12.8.0}&\textcolor{red}{$71.5$}&\textcolor{red}{$3.12$}&$0.001$&$0.000$&$0.000$&$0.013$&$4.985$&$78.33$ \\ \hline
\end{tabular}
\caption{For 1-norm soft margin SVM, we let PC$^2$PM converge with $res_1 < 0.02$ instead of $10^{-3}$, while still keep $res_2 < 10^{-3}$.}
\vspace{4pt}
\end{subtable}
\caption{Comparison of PC$^2$PM with CPLEX 12.8.0 for solving multiple kernel learning problems using $5$ Gaussian kernel functions.}
\label{tab: MKL 1}%\vspace*{-0pt}
\end{table}
Each data set of a total number of $16,000$ data points is randomly partitioned into $80\%$ for training and $20\%$ for testing. The reported values in each row of Table~\ref{tab: MKL 1} are averaged over five different random partitions.
\par We implement PC$^2$PM using $128$ cores for primal updates and $5$ cores for dual updates, which amount to a total of $6$ compute nodes on Purdue's Brown cluster. The average elapsed wall-clock time used by PC$^2$PM to converge with a tolerance $\tau^{\text{PC}^2\text{PM}} = 10^{-3}$ is presented in Table \ref{tab: MKL 1}, along with the averaged amount of memory used by each node. We also report in Table \ref{tab: MKL 1} the average learned non-negative coefficients $\lambda_1^*, \dots \lambda_5^*$, as well as the average TSA. The performance of CPLEX 12.8.0 with the same tolerance is also presented in Table \ref{tab: MKL 1} for comparison. As shown by the values of the coefficients learned, the Gaussian kernel function $k_4$ with $\sigma^2 = 10.0$ is selected by the models of both two soft margin SVMs for the Two-norm Problem; the HEPMASS Data Set selects the Gaussian kernel function $k_5$ with $\sigma^2 = 100.0$. For 2-norm soft margin SVMs, PC$^2$PM converges much faster than CPLEX, and also uses much less memory (as expected). For TSA, both PC$^2$PM and CPLEX obtain the same value, calculated using their own optimal solution point $(\bm{\alpha}^*, \lambda_1^*, \dots, \lambda_m^*)$. For 1-norm soft margin SVMs, CPLEX fails to provide a solution due to running out of memory, while PC$^2$PM  still solves the problem. 
\par In Table~\ref{tab: MKL 1 Large},
\begin{table}[htb]
\renewcommand{\arraystretch}{1.2}
\centering
\footnotesize
\setlength\tabcolsep{2.5pt}
\begin{tabular}{cc|cccccc} \hline
\multicolumn{8}{c}{\textbf{HEPMASS Data Set}} \\ \hline
\multicolumn{8}{c}{\textbf{PC$^2$PM} $\quad$ SM2 $\quad$ $C = 1.0$} \\ \hline
\textbf{$\bm{\sigma}^2$ search}&\multirow{2}{*}{\textbf{m}}&\multirow{2}{*}{\textbf{$\text{n}_{\text{node}}$}}&\multirow{2}{*}{\textbf{$\text{n}_{\text{core}}$}}&\textbf{mem./node}&\textbf{time}&\multirow{2}{*}{\textbf{non-zero $\bm{\lambda}$'s}}&TSA \\
\textbf{range}&&&&(GB)&(hour)&&($\%$) \\ \hline
\multirow{2}{*}{$[10^{-4}, 10^4]$}&$9$&$6$&$128+9$&$2.9$/node&$0.43$&$\lambda_7^* = 8.995 \quad (\sigma^2 = 10^2)$&$79.62$ \\
&$17$&$7$&$128+17$&$4.4$/node&$1.51$&$\lambda_{12}^* = 17.001 \quad (\sigma^2 = 10^{1.5})$&$80.60$ \\ \hline
\end{tabular}
\caption{Numerical results of applying PC$^2$PM to solve 2-norm soft margin SVMs using multiple (9, 17) Gaussian kernel functions for the HEPMASS Data Set.}
\label{tab: MKL 1 Large}%\vspace*{-0pt}
\end{table}
we also report the numerical results of applying PC$^2$PM to solve $2$-norm soft margin SVMs for the HEPMASS Data Set, using $9$ Gaussian kernel functions with $\sigma^2$ equal to $10^{-4}$, $10^{-3}$, $10^{-2}$, $10^{-1}$, $10^0$, $10^1$, $10^2$, $10^3$ and $10^4$ respectively. Though the number of constraints doubles, PC$^2$PM still converges within a reasonable amount of time, and remains memory efficient. The Gaussian kernel function $k_7$ with $\sigma^2 = 100.0$ is still selected by the model. We further search the range of $[10^{-4}, 10^4]$ using $17$ Gaussian kernel functions with $\sigma^2$ equal to $10^{-4}$, $10^{-3.5}$, $10^{-3}$, $\dots$, $10^3$, $10^{3.5}$, $10^4$. The Gaussian kernel function $k_{12}$ with $\sigma^2 = 10^{1.5}$ is selected instead, and we observe a slightly increased average TSA.
\par While the numerical experiments so far have demonstrated the scalability of the PC$^2$PM algorithm due to its distributed data storage and natural decomposition to facilitate parallel computing, in the following experiments, we show the benefits of the PC$^2$PM algorithm for not requiring any matrix decompositions. In this test, we use three kernel functions, instead of five, to solve \eqref{eq: SM1 MKL SVM Dual} and \eqref{eq: SM2 MKL SVM Dual}. The three kernel functions consist of $k_1$ -- the Gaussian kernel function with $\sigma^2 = 100.0$, $k_2$ -- a linear kernel function defined in \eqref{eq: Linear_kernel}, and $k_3$ -- a polynomial kernel function defined in \eqref{eq: Polynomial_kernel}. The value of the parameter $C$ is fixed as $1.0$, and is only changed to $2.0$ when using 1-norm soft margin SVM for HEPMASS Data Set. All the other settings remain the same as in the previous experiment (except for the value of $R$, which is set as $3.0$). The numerical results are reported in Table~\ref{tab: MKL 2}.
\begin{table}[htb]
\renewcommand{\arraystretch}{1.2}
\centering
\footnotesize
\setlength\tabcolsep{5.2pt}
\begin{subtable}[c]{1.0\textwidth}
\begin{tabular}{c|c|cccccc} \hline
\multicolumn{8}{c}{\textbf{Two-norm Problem}} \\ \hline
\textbf{SVM}&&\textbf{mem./node}&\textbf{time}&\multirow{2}{*}{$\lambda_1^*$}&\multirow{2}{*}{$\lambda_2^*$}&\multirow{2}{*}{$\lambda_3^*$}&\textbf{TSA} \\
\textbf{Criteria}&&(GB)&(hour)&&&&($\%$) \\ \hline
SM1&\textbf{PC$^2$PM}&\textcolor{blue}{$1.7$/node}&$3.16$&$0.000$&$3.029$&$0.000$&$91.29$ \\ \cdashline{2-8}
$C = 1.0$&\textbf{CPLEX 12.8.0}&\textcolor{red}{Non-Con. Error}&N.A.&N.A.&N.A.&N.A.&N.A. \\ \hline
SM2&\textbf{PC$^2$PM}&\textcolor{blue}{$1.5$/node}&$1.69$&$0.000$&$3.054$&$0.000$&$97.85$ \\ \cdashline{2-8}
$C = 1.0$&\textbf{CPLEX 12.8.0}&\textcolor{red}{Non-Con. Error}&N.A.&N.A.&N.A.&N.A.&N.A. \\ \hline
\end{tabular}
%\caption{}
\vspace{10pt}
\end{subtable}
\quad% 
\begin{subtable}[c]{1.0\textwidth}
\begin{tabular}{c|c|cccccc} \hline
\multicolumn{8}{c}{\textbf{HEPMASS Data Set}} \\ \hline
\textbf{SVM}&&\textbf{mem./node}&\textbf{time}&\multirow{2}{*}{$\lambda_1^*$}&\multirow{2}{*}{$\lambda_2^*$}&\multirow{2}{*}{$\lambda_3^*$}&\textbf{TSA} \\
\textbf{Criteria}&&(GB)&(hour)&&&&($\%$) \\ \hline
SM1&\textbf{PC$^2$PM}&\textcolor{blue}{$1.6$/node}&$6.43$&$0.000$&$3.021$&$0.000$&$72.81$ \\ \cdashline{2-8}
$C = 2.0$&\textbf{CPLEX 12.8.0}&\textcolor{red}{Non-Con. Error}&N.A.&N.A.&N.A.&N.A.&N.A. \\ \hline
SM2&\textbf{PC$^2$PM}&\textcolor{blue}{$1.4$/node}&$0.83$&$0.000$&$3.019$&$0.000$&$80.51$ \\ \cdashline{2-8}
$C = 1.0$&\textbf{CPLEX 12.8.0}&\textcolor{red}{Non-Con. Error}&N.A.&N.A.&N.A.&N.A.&N.A. \\ \hline
\end{tabular}
%\caption{}
\vspace{4pt}
\end{subtable}
\caption{Comparison of PC$^2$PM with CPLEX 12.8.0 for solving multiple kernel learning problems using $3$ kernel functions.}
\label{tab: MKL 2}%\vspace*{-0pt}
\end{table}
For all groups of tests, CPLEX returns an error stating that the quadratic constraint containing $G_3(K_{3, tr})$ is not convex, which is theoretically impossible because each matrix $G_i(K_{i,tr})$ is at least a PSD matrix as we discussed previously; while PC$^2$PM solves all the instances without any issues. The error returned by CPLEX is created likely by the failure of matrix decomposition of a large-scale PSD matrix due to precision limit. Once we reduce the size of the matrices in \eqref{eq: SM1 MKL SVM Dual} and \eqref{eq: SM2 MKL SVM Dual}, CPLEX can then solve the instances without error messages. This numerical experiment illustrates that not requiring matrix decomposition in the PC$^2$PM is not just of computational convenience; it can indeed make the algorithm more robust to solve large-scale problems without facing potential issues caused by floating point arithmetic.
%
%
%
%
%%%%%%%%%%%%%%%%%%%%%%%%%%%%%%%%%%%%%%%%%%%%%%%%%%%%%%%%%%%%%%%%%%%%%%%%%%%%%%%%
\section{Conclusion and Future Works}
In this paper, we propose a novel distributed algorithm, built upon the original idea of the PCPM algorithm, that can solve non-separable convex QCQPs in a Jacobi-fashion (that is, parallel updating). Numerical results show that our algorithm, termed as PC$^2$PM, exhibits much better scalability when compared to CPLEX, which uses the IPM to solve convex QCQPs. The scalability of the algorithm is attributed to the three key features of the algorithm design: first, 
the PC$^2$PM algorithm can decompose primal (and dual) variables down to the scalar level and update them in parallel, even when the quadratic constraints are non-separable. Second, when implementing the algorithm, only the related columns of all the Hessian matrices need to be stored locally, instead of the entire matrices on each of computing unit in a parallel computing setting. Third, our algorithm does not need any matrix decomposition (unlike any semi-definite-programming-based approach), which can improve the algorithm's robustness, especially when solving convex QCQPs with PSD matrices, as demonstrated in our numerical experiments summarized in Table \ref{tab: MKL 2}. The second and the third feature together make our algorithm particularly suitable to solve extreme-dimension QCQPs, which likely will cause memory issues for other algorithms.  

In addition to the scalability of the PC$^2$PM algorithm, its ability to solve non-separable, quadratically constrained problems in  Jacobi-fashion should also be emphasized, 
as in general it is very difficult to design distributed algorithms with Jacobi-style update (as opposed to the sequential Gauss-Seidel update) to solve optimization problems with non-separable constraints. Whether the algorithm idea from 
PC$^2$PM can be extended to solve more general convex problems is certainly worth exploring. There are several other lines of research that can be done to improve the current work. First, while we proved convergence of PC$^2$PM, we cannot prove its convergence rate as of now. Second, while the parallel updating of the primal variables is a nice property of PC$^2$PM, it is still a synchronous algorithm in the sense that the algorithm needs to wait for all primal and dual updates to be done before it can move to the next iteration. An asynchronous implementation of the algorithm will no doubt make it even more suitable for distributed computing, and we defer it to our future work. Third, there have been increasing works on solving large-scale non-convex QCQPs. As mentioned in the introduction section, one algorithm idea is to solve it with a sequence of convexified QCQPs, where our algorithm is then applicable. This naturally leads to an algorithm with nested loops, where the outer loop lays out sequential convexification, and the inner loop invokes our algorithm. It would be interesting to see how such a nested algorithm performs in practice, especially with high-dimension problems.

%In this paper, we propose the P-PCPD method for solving convex QCQPs. The proposed algorithm is designed with scalability and distributed computing in mind, and hence should be best suited to solve extreme-scale convex QCQPs that may arise from many machine learning applications. Rigorous proof of convergence is also established in this paper. An immediate next step of our work is to investigate the convergence rate of the P-PCPD method, which, based on numerical experiments, appears to be of a linear convergence rate even with only PSD matrices. We also plan to compile a group of large-scale QCQPs from real-world applications and to test the practical performance of our method. 

%%%%%%%%%%%%%%%%%%%%%%%%%%%%%%%%%%%%%%%%%%%%%%%%%%%%%%%%%%%%%%%%%%%%%%%%%%%%%%%%
\begin{acknowledgements}
The authors would like to acknowledge the support of National Science Foundation grant CMMI-1832688 and the Emerging Frontiers grant from the School of Industrial Engineering at Purdue University. Specially, we wish to thank Professor Jong-Shi Pang of University of Southern California for the helpful comments and discussions. In addition, we would like to thank Purdue Rosen Center for Advanced Computing for providing the computing resources and technical support.
\end{acknowledgements}

\bibliographystyle{spmpsci}      % mathematics and physical sciences
\bibliography{main}   % name your BibTeX data base
%\begin{thebibliography}{}
% etc
%\end{thebibliography}
%
%
%
%
%%%%%%%%%%%%%%%%%%%%%%%%%%%%%%%%%%%%%%%%%%%%%%%%%%%%%%%%%%%%%%%%%%%%%%%%%%%%%%%%
\newpage
\begin{appendices}
\section{Step-size Update Rule for $\rho^{k+1}$}\label{app:StepRule}
With a given scalar $0 \leq \epsilon_0 < 1$, and a series of positive scalars $\epsilon_1, \ldots, \epsilon_8 > 0$ that satisfy $\sum_{s=1}^8 \epsilon_s \leq 1 - \epsilon_0$, we define the following function $\rho: \mathbb{X} \times \mathbb{R}^{n_2} \times \mathbb{R}_+^{m_1} \times \mathbb{R}^{m_2} \to (0, +\infty)$  to update the adaptive step size $\rho^{k+1}$ in Algorithm~\ref{alg: PC^2PM} at each iteration $k$:
\begin{equation}\label{eq: adaptive step size}
\begin{array} {lcl}
\rho^{k+1} & = &\rho(\mathbf{x}^k, \mathbf{u}^k, \bm{\lambda}^k, \bm{\gamma}^k)  \\[10pt]
& := & \min\hspace*{-3pt}\left\{\rho_1, \rho_2(\mathbf{x}^k, \mathbf{u}^k, \bm{\lambda}^k), \rho_3(\mathbf{x}^k, \bm{\lambda}^k, \bm{\gamma}^k), \rho_4, \rho_5(\mathbf{x}^k), \rho_6, \rho_7, \rho_8\right\}\hspace*{-3pt},
\end{array}
\end{equation}
where
\begin{enumerate}[label=(\roman*)]
\item
$\rho_1 = \left\{
\begin{array}{ll}
\displaystyle\frac{\epsilon_1}{\lVert P_0 \rVert_F}, &\text{if }\lVert P_0 \rVert_F \not= 0 \\
\epsilon_1, &\text{if }\lVert P_0 \rVert_F = 0, 
\end{array}
\right.$
with $\Vert \cdot \rVert_F$ representing the Frobenius norm of a matrix;
\vspace{8pt}
\item $\rho_2(\mathbf{x}^k, \mathbf{u}^k, \bm{\lambda}^k) =  \min_i\{\rho_{2i}(\mathbf{x}^k, \mathbf{u}^k, \bm{\lambda}^k)\}$, where 
$$\rho_{2i}(\mathbf{x}^k, \mathbf{u}^k, \bm{\lambda}^k) \coloneqq \left\{
\begin{array}{ll}
\displaystyle\frac{-b_i + \sqrt{b_i^2 + 4 a_i c_i}}{2 a_i}, &\text{if }a_i > 0 \\
\displaystyle\frac{c_i}{b_i}, &\text{if }a_i = 0, b_i > 0 \\[8pt]
M, &\text{if }a_i = 0, b_i = 0,
\end{array}
\right.$$
for all $i = 1, \dots, m_1$, with $a_i = \lvert \frac{1}{2} (\mathbf{x}^k)^T P_i \mathbf{x}^k + \mathbf{q}_i^T \mathbf{x}^k + \mathbf{c}_i^T \mathbf{u}^k + r_i \rvert \geq 0$, and $b_i = \lambda_i^k \geq 0$. For $c_i$, if $\lVert P_i \rVert_F \not= 0$, $c_i = \frac{\epsilon_2}{m_1 \lVert P_i \rVert_F} > 0$; otherwise $c_i = \frac{\epsilon_2}{m_1} > 0$. The constant $M > 0$ can be any fixed, arbitrarily large scalar;
\vspace{8pt}
\item 
$\rho_{3}(\mathbf{x}^k, \bm{\lambda}^k, \bm{\gamma}^k) = $
$$\hspace*{45pt} \left\{
\begin{array}{ll}
\displaystyle\min\{2 \epsilon_3, \frac{-b + \sqrt{b^2 + 4 a c}}{2 a}\}, &\text{if }a > 0 \\ 
\displaystyle\min\{2 \epsilon_3, \frac{c}{b}\}, &\text{if }a = 0, b > 0 \\[6pt]
2 \epsilon_3, &\text{if }a = 0, b = 0,
\end{array}
\right.$$
where $a = \lVert P_0 \mathbf{x}^k + \mathbf{q}_0 + \sum_{i=1}^{m_1} \lambda_i^k (P_i \mathbf{x}^k + \mathbf{q}_i) + A^T \bm{\gamma}^k\rVert_2 \geq 0$, $b = 2 \lVert \mathbf{x}^k \rVert_2 \geq 0$ and $c = \frac{2 \epsilon_3}{\lVert P \rVert_F} > 0$ with $P \in \mathbb{R}^{m_1 n_1 \times n_1}$ denoting the stacked matrix $\left(\begin{array}{c}P_1 \\[-2pt] \vdots \\[-2pt] P_{m_1}\end{array}\right)$;
\vspace{8pt}
\item $\rho_4 =
\left\{
\begin{array}{ll}
\displaystyle\frac{\epsilon_4}{\lVert Q \rVert_F}, &\text{if }\lVert Q \rVert_F \not= 0 \\[8pt]
\epsilon_4, &\text{if }\lVert Q \rVert_F = 0
\end{array}
\right.$,
where $Q \in \mathbb{R}^{m_1 \times n_1}$ denotes matrix $\left(\begin{array}{c}\mathbf{q}_1^T \\[-2pt] \vdots \\ \mathbf{q}_{m_1}^T\end{array}\right)$, with the $\mathbf{q}_i$'s being the vectors in the linear terms of $\mathbf{x}$ in the QCQP \eqref{eq: QCQP Problem Form};
\vspace{8pt}
\item $\rho_5(\mathbf{x}^k) = \left\{
\begin{array}{ll}
\displaystyle\frac{\epsilon_5}{\lVert \mathbf{x}^k \rVert_2 \lVert P \rVert_F}, &\text{if }\lVert \mathbf{x}^k \rVert_2 \not= 0 \\[8pt]
\epsilon_5, &\text{if }\lVert \mathbf{x}^k \rVert_2 = 0
\end{array}
\right.$;
\vspace{8pt}
\item $\rho_6 = \left\{
\begin{array}{ll}
\displaystyle\frac{\epsilon_6}{\lVert C \rVert_F}, &\text{if }\lVert C \rVert_F \not= 0 \\[8pt]
\epsilon_6, &\text{if }\lVert C \rVert_F = 0 
\end{array}
\right.$,
where $C \in \mathbb{R}^{m_2 \times n_2}$ denotes matrix $\left(\begin{array}{c}\mathbf{c}_1^T \\[-2pt] \vdots \\ \mathbf{c}_{m_2}^T\end{array}\right)$, with the $\mathbf{c}_j$'s being the vectors in the linear terms of $\mathbf{u}$ in the QCQP \eqref{eq: QCQP Problem Form};
\vspace{8pt}
\item $\rho_7 = \left\{
\begin{array}{ll}
\displaystyle\frac{\epsilon_7}{\lVert A \rVert_F}, &\text{if }\lVert A \rVert_F \not= 0 \\[8pt]
\epsilon_7, &\text{if }\lVert A \rVert_F = 0
\end{array}
\right.$, 
where $A$ is the matrix in the linear constraint $A \mathbf{x} + B \mathbf{u} = \mathbf{b}$ in \eqref{eq: QCQP Problem Form};
\vspace{8pt}
\item $\rho_8 = \left\{
\begin{array}{ll}
\displaystyle\frac{\epsilon_8}{\lVert B \rVert_F}, &\text{if }\lVert B \rVert_F \not= 0 \\[8pt]
\epsilon_8, &\text{if }\lVert B \rVert_F = 0
\end{array}
\right.$, where $B$ is the matrix in the linear constraint $A \mathbf{x} + B \mathbf{u} = \mathbf{b}$ in \eqref{eq: QCQP Problem Form}.
\end{enumerate}
$\hfill\Box$

\par While the rules to update the step-size $\rho^{k+1}$ may appear to be very cumbersome, the calculations are actually quite straightforward. Since the Frobenius norm of all matrices can be obtained in advance, the values of $\rho_1$, $\rho_4$, $\rho_6$, $\rho_7$ and $\rho_8$ are pre-determined. Given a current solution $(\mathbf{x}^k, \mathbf{u}^k, \bm{\lambda}^k, \bm{\gamma}^k)$, $\rho_2$, $\rho_{3}$ and $\rho_5$ can also be easily calculated. The minimum of all the $\rho_s$'s then determines the value of the adaptive step size $\rho^{k+1}$. 

\section{Proofs in Section 3}\label{app: Proof_Sec3}
\subsection{Proof of Proposition~\ref{prp: Distance}}
We first prove the inequality \eqref{eq: primal distance}. Consider the linear approximation of the Lagrangian function of a QCQP, as defined in \eqref{eq: R_func}, with a given point $\zeta^k \equiv (\mathbf{x}^k, \lambda^k, \gamma^k)$. Let $\widehat{\mathbf{z}} = (\mathbf{y}^{k+1}, \mathbf{v}^{k+1})$, the $(k+1)$-th iteration of the primal predictor of $\mathbf{x}^k$ and $\mathbf{u}^k$ in the PC$^2$PM algorithm, as given in \eqref{eq: x primal predictor update} and \eqref{eq: u primal predictor update}, respectively. By Lemma \ref{lem: Lemma2}, we know that $\widehat{\mathbf{z}}$ is the unique minimizer of the corresponding proximal minimization problem in \eqref{eq: y_k+1}. By defining $\bar{\mathbf{z}} = (\mathbf{x}^k, \mathbf{u}^k)$ and $\mathbf{z} = (\mathbf{x}^{k+1}, \mathbf{u}^{k+1})$, and using Lemma \ref{lem: Lemma1}, we have that 
\begin{equation}
2\rho^{k+1}\Bigg[\mathcal{R}(\widehat{\mathbf{z}}; \zeta^k) - \mathcal{R}(\mathbf{z};\zeta^k)\Bigg] \leq \lVert \bar{\mathbf{z}} - \mathbf{z} \rVert_2^2 - \lVert \widehat{\mathbf{z}} - \mathbf{z} \rVert_2^2 - \lVert \widehat{\mathbf{z}} - \bar{\mathbf{z}} \rVert_2^2, 
\end{equation}
which leads to the following expanded inequality
\begingroup
\begin{align}
&2 \rho^{k+1}\Bigg\{(P_0 \mathbf{x}^k + \mathbf{q}_0)^T \mathbf{y}^{k+1} + \mathbf{c}_0^T \mathbf{v}^{k+1} + r_0 \nonumber \\
&\hspace*{32pt} + \sum_{i=1}^{m_1} \lambda_i^k \big[(P_i \mathbf{x}^k + \mathbf{q}_i)^T \mathbf{y}^{k+1} + \mathbf{c}_i^T \mathbf{v}^{k+1} + r_i\big] \nonumber \\
&\hspace*{32pt} + (\bm{\gamma}^k)^T (A \mathbf{y}^{k+1} + B \mathbf{v}^{k+1} - \mathbf{b})\Bigg\} \nonumber \\
- &2 \rho^{k+1}\Bigg\{(P_0 \mathbf{x}^k + \mathbf{q}_0)^T \mathbf{x}^{k+1} + \mathbf{c}_0^T \mathbf{u}^{k+1} + r_0 \nonumber \\
&\hspace*{32pt} + \sum_{i=1}^{m_1} \lambda_i^k \big[(P_i \mathbf{x}^k + \mathbf{q}_i)^T \mathbf{x}^{k+1} + \mathbf{c}_i^T \mathbf{u}^{k+1} + r_i\big] \nonumber \\
&\hspace*{32pt} + (\bm{\gamma}^k)^T (A \mathbf{x}^{k+1} + B \mathbf{u}^{k+1} - \mathbf{b})\Bigg\} \nonumber \\[8pt]
\leq &\ \Big(\lVert \mathbf{x}^k - \mathbf{x}^{k+1} \rVert_2^2  + \lVert \mathbf{u}^k - \mathbf{u}^{k+1} \rVert_2^2 \Big) \nonumber \\
- &\ \Big(\lVert \mathbf{y}^{k+1} + \mathbf{x}^{k+1} \rVert_2^2 + \lVert \mathbf{v}^{k+1} - \mathbf{u}^{k+1} \rVert_2^2 \Big) \nonumber \\
- &\ \Big(\lVert \mathbf{y}^{k+1} - \mathbf{x}^k \rVert_2^2 + \lVert \mathbf{v}^{k+1} - \mathbf{u}^k \rVert_2^2\Big). \label{eq: Ineq1}
\end{align}
\endgroup
Now consider the $\mathcal{R}$ function at a different given point $\zeta^{k+1} \hspace*{-3pt}\equiv\hspace*{-3pt} (\mathbf{y}^{k+1}\hspace*{-1pt}, \hspace*{-1pt}\mu^{k+1}\hspace*{-1pt}, \hspace*{-1pt}\bm{\nu}^{k+1})$. 
With a slight abuse of notation, we now let $\widehat{\mathbf{z}} = (\mathbf{x}^{k+1}, \mathbf{u}^{k+1})$, the primal correctors at the $(k+1)$-th iteration of the PC$^2$PM algorithm. Also letting $\mathbf{z} = (\mathbf{x}^*, \mathbf{u}^*)$, but keeping $\bar{\mathbf{z}} = (\mathbf{x}^k, \mathbf{u}^k)$, by \eqref{eq: x_k+1} in Lemma \ref{lem: Lemma2} and Lemma \ref{lem: Lemma1}, we have that:
$$
2\rho^{k+1}\Bigg[\mathcal{R}(\widehat{\mathbf{z}}; \zeta^{k+1}) - \mathcal{R}(\mathbf{z};\zeta^{k+1})\Bigg] \leq \lVert \bar{\mathbf{z}} - \mathbf{z} \rVert_2^2 - \lVert \widehat{\mathbf{z}} - \mathbf{z} \rVert_2^2 - \lVert \widehat{\mathbf{z}} - \bar{\mathbf{z}} \rVert_2^2, 
$$
which leads to the following expanded inequality 
\begingroup
\begin{align}
&2 \rho^{k+1}\Bigg\{(P_0 \mathbf{y}^{k+1} + \mathbf{q}_0)^T \mathbf{x}^{k+1} + \mathbf{c}_0^T \mathbf{u}^{k+1} + r_0 \nonumber \\
&\hspace*{32pt} +\sum_{i=1}^{m_1} \mu_i^{k+1} \big[(P_i \mathbf{y}^{k+1} + \mathbf{q}_i)^T \mathbf{x}^{k+1} + \mathbf{c}_i^T \mathbf{u}^{k+1} + r_i\big] \nonumber \\
&\hspace*{32pt} + (\bm{\nu}^{k+1})^T (A \mathbf{x}^{k+1} + B \mathbf{u}^{k+1} - \mathbf{b})\Bigg\} \nonumber \\
- &2 \rho^{k+1}\Bigg\{(P_0 \mathbf{y}^{k+1} + \mathbf{q}_0)^T \mathbf{x}^* + \mathbf{c}_0^T \mathbf{u}^* + r_0 \nonumber \\
&\hspace*{32pt} +\sum_{i=1}^{m_1} \mu_i^{k+1} \big[(P_i \mathbf{y}^{k+1} + \mathbf{q}_i)^T \mathbf{x}^* + \mathbf{c}_i^T \mathbf{u}^* + r_i\big] \nonumber \\
&\hspace*{32pt} + (\bm{\nu}^{k+1})^T (A \mathbf{x}^* + B \mathbf{u}^* - \mathbf{b})\Bigg\} \nonumber \\[8pt]
\leq &\ \Big(\lVert \mathbf{x}^k - \mathbf{x}^* \rVert_2^2 + \lVert \mathbf{u}^k - \mathbf{u}^* \rVert_2^2 \Big) \nonumber \\
- &\ \Big(\lVert \mathbf{x}^{k+1} - \mathbf{x}^* \rVert_2^2 + \lVert \mathbf{u}^{k+1} - \mathbf{u}^* \rVert_2^2 \Big) \nonumber \\
- &\ \Big(\lVert \mathbf{x}^{k+1} - \mathbf{x}^k \rVert_2^2 + \lVert \mathbf{u}^{k+1} - \mathbf{u}^k \rVert_2^2\Big). \label{eq: Ineq2}
\end{align}
\endgroup
The final piece to derive inequality \eqref{eq: primal distance} is to utilize Lemma \ref{lem: Lemma3}. Let $(\mathbf{x}^*\hspace*{-1pt}, \hspace*{-1pt}\mathbf{u}^*\hspace*{-1pt}, \hspace*{-1pt}\bm{\lambda}^*\hspace*{-1pt}, \hspace*{-1pt}\bm{\gamma}^*\hspace*{-1pt})$ be a saddle point of QCQP \eqref{eq: QCQP Problem Form}, and again, $\zeta^{k+1} = (\mathbf{y}^{k+1}, \mu^{k+1}, \bm{\nu}^{k+1})$. By Lemma \ref{lem: Lemma3}, we have that
\begin{equation}
\begin{aligned}
&\mathcal{R}(\mathbf{x}^*, \mathbf{u}^*; \zeta^{k+1}) - 
\mathcal{R}(\mathbf{y}^{k+1}, \mathbf{v}^{k+1} ;\zeta^{k+1}) \\[8pt]
\leq & \displaystyle \sum_{i=1}^m \bigg[ (\lambda_i^* - \mu_i^{k+1}) \bigg(\frac{1}{2} \mathbf{y}^{{k+1}^T} P_i \mathbf{y}^{k+1} + \mathbf{q}_i^T \mathbf{y}^{k+1} + \mathbf{c}_i^T \mathbf{v}^{k+1} + r_i\bigg)\bigg] \\[15pt]
&+ (\bm{\gamma}^* - \bm{\nu}^{k+1})^T (A \mathbf{y}^{k+1} + B \mathbf{v}^{k+1} - \mathbf{b}).
\end{aligned}
\end{equation}
Multiplying both sides by $2 \rho^{k+1}$ and expanding the $\mathcal{R}$ function, we have that
\begingroup
\begin{align}
&2 \rho^{k+1} \Bigg\{(P_0 \mathbf{y}^{k+1} + \mathbf{q}_0)^T \mathbf{x}^* + \mathbf{c}_0^T \mathbf{u}^* + r_0 \nonumber \\
&\hspace*{32pt} +\sum_{i=1}^m \mu_i^{k+1} \big[(P_i \mathbf{y}^{k+1} + \mathbf{q}_i)^T \mathbf{x}^* + \mathbf{c}_i^T \mathbf{u}^* + r_i\big] \nonumber \\
&\hspace*{32pt} + (\bm{\nu}^{k+1})^T (A \mathbf{x}^* + B \mathbf{u}^* - \mathbf{b})\Bigg\} \nonumber \\
- &2 \rho^{k+1} \Bigg\{(P_0 \mathbf{y}^{k+1} + \mathbf{q}_0)^T \mathbf{y}^{k+1} + \mathbf{c}_0^T \mathbf{v}^{k+1} + r_0 \nonumber \\
&\hspace*{32pt} + \sum_{i=1}^m \mu_i^{k+1} \big[(P_i \mathbf{y}^{k+1} + \mathbf{q}_i)^T \mathbf{y}^{k+1} + \mathbf{c}_i^T \mathbf{v}^{k+1} + r_i\big] \nonumber \\
& \hspace*{32pt} + (\bm{\nu}^{k+1})^T (A \mathbf{y}^{k+1} + B \mathbf{v}^{k+1} - \mathbf{b})\Bigg\} \nonumber \\[8pt]
\leq &\ 2 \rho^{k+1} \Bigg\{\sum_{i=1}^m (\lambda_i^* - \mu_i^{k+1}) \big[\frac{1}{2} (\mathbf{y}^{k+1})^T P_i \mathbf{y}^{k+1} + \mathbf{q}_i^T \mathbf{y}^{k+1} + \mathbf{c}_i^T \mathbf{v}^{k+1} + r_i\big] \nonumber \\ 
&\hspace*{35pt} + (\bm{\gamma}^* - \bm{\nu}^{k+1})^T (A \mathbf{y}^{k+1} + B \mathbf{v}^{k+1} - \mathbf{b})\Bigg\}.  \label{eq: Ineq3}
\end{align}
\endgroup
Adding the three inequalities \eqref{eq: Ineq1}, \eqref{eq: Ineq2} and \eqref{eq: Ineq3} yields the inequality \eqref{eq: primal distance} in Proposition~\ref{prp: Distance}. 
\par To prove the second inequality, \eqref{eq: dual distance}, in Proposition \ref{prp: Distance}, we use a similar approach as above, just replacing the linear approximation function $\mathcal{R}$ with the original Lagrangian function $\mathcal{L}$. More specifically, let $\widehat{\mathbf{z}} = (\bm{\mu}^{k+1}, \bm{\nu}^{k+1})$. By \eqref{eq: mu_k+1} in Lemma \ref{lem: Lemma2}, we know that 
\begin{equation}
\begin{aligned}
&\widehat{\mathbf{z}} \coloneqq (\bm{\mu}^{k+1},\ \bm{\nu}^{k+1}) \\
= &\underset{\bm{\lambda} \in \mathbb{R}_+^{m_1},\ \bm{\gamma} \in \mathbb{R}^{m_2}}{\argmin} -\mathcal{L}(\mathbf{x}^k, \mathbf{u}^k, \bm{\lambda}, \bm{\gamma}) + \frac{1}{2 \rho^{k+1}} \lVert \bm{\lambda} - \bm{\lambda}^k \rVert_2^2 + \frac{1}{2 \rho^{k+1}} \lVert \bm{\gamma} - \bm{\gamma}^k \rVert_2^2.
\end{aligned}
\end{equation}
Letting $\bar{\mathbf{z}} = (\bm{\lambda}^k, \bm{\gamma}^k)$ and choosing a specific $\mathbf{z} = (\bm{\lambda}^{k+1}, \bm{\gamma}^{k+1})$, we use Lemma \ref{lem: Lemma1} to obtain that
\begin{equation}
2\rho^{k+1}\Bigg[ \bigg(-\mathcal{L} (\mathbf{x}^k,\mathbf{u}^k;\widehat{\mathbf{z}}) \bigg) - \Bigg(-\mathcal{L} (\mathbf{x}^k,\mathbf{u}^k; \mathbf{z}) \Bigg) \Bigg] \leq \lVert \bar{\mathbf{z}} - \mathbf{z} \rVert_2^2 - \lVert \widehat{\mathbf{z}} - \mathbf{z} \rVert_2^2 - \lVert \widehat{\mathbf{z}} - \bar{\mathbf{z}} \rVert_2^2, 
\end{equation}
which yields the following expanded inequality: 
\begingroup
\begin{align}
&2 \rho^{k+1}\Bigg\{\sum_{i=1}^{m_1} (\lambda_i^{k+1} - \mu_i^{k+1}) \bigg[\frac{1}{2} (\mathbf{x}^k)^T P_i \mathbf{x}^k + \mathbf{q}_0^T \mathbf{x}^k + \mathbf{c}_0^T \mathbf{u}^k + r_i\bigg] \nonumber \\
&\hspace*{32pt} + (\bm{\gamma}^{k+1} - \bm{\nu}^{k+1})^T (A \mathbf{x}^k + B \mathbf{u}^k - \mathbf{b})\Bigg\} \nonumber \\
\leq &\Big(\lVert \bm{\lambda}^k - \bm{\lambda}^{k+1} \rVert_2^2 + \lVert \bm{\gamma}^k - \bm{\gamma}^{k+1} \rVert_2^2\Big) \nonumber \\
- &\Big(\lVert \bm{\mu}^{k+1} - \bm{\lambda}^{k+1} \rVert_2^2 + \lVert \bm{\nu}^{k+1} - \bm{\gamma}^{k+1} \rVert_2^2\Big) - \Big(\lVert \bm{\mu}^{k+1} - \bm{\lambda}^k \rVert_2^2 + \lVert \bm{\nu}^{k+1} - \bm{\gamma}^k \rVert_2^2\Big). \label{eq: primal_dist_bound}
\end{align}
\endgroup
Similarly, again with some abuse of notation, letting  $\widehat{\mathbf{z}} = (\bm{\lambda}^{k+1}, \bm{\gamma}^{k+1})$, by \eqref{eq: lambda_k+1} in Lemma \ref{lem: Lemma2}, we have that 
\begin{equation}
\begin{aligned}
&\widehat{\mathbf{z}} \coloneqq (\bm{\lambda}^{k+1}, \bm{\gamma}^{k+1}) \\
= &\underset{\bm{\lambda} \in \mathbb{R}_+^{m_1},\ \bm{\gamma} \in \mathbb{R}^{m_2}}{\argmin} \hspace*{-4pt}-\mathcal{L}(\mathbf{y}^{k+1}, \mathbf{v}^{k+1}, \bm{\lambda}, \bm{\gamma}) + \frac{1}{2 \rho^{k+1}} \lVert \bm{\lambda} - \bm{\lambda}^k \rVert_2^2 + \frac{1}{2 \rho^{k+1}} \lVert \bm{\gamma} - \bm{\gamma}^k \rVert_2^2.
\end{aligned}
\end{equation}
By choosing $\mathbf{z}$ to be $(\bm{\lambda}^*, \bm{\gamma}^*)$, while keeping $\bar{\mathbf{z}}$ at $(\bm{\lambda}^k, \bm{\gamma}^k)$, we have from Lemma~\ref{lem: Lemma1} that
\begin{equation}
\begin{aligned}
&2\rho^{k+1}\Bigg[ \bigg(-\mathcal{L} (\mathbf{y}^{k+1},\mathbf{v}^{k+1};\widehat{\mathbf{z}}) \bigg) - 
\Bigg(-\mathcal{L} (\mathbf{y}^{k+1},\mathbf{v}^{k+1}; \mathbf{z}) \Bigg) \Bigg] \\[10pt]
\leq &\lVert \bar{\mathbf{z}} - \mathbf{z} \rVert_2^2 - \lVert \widehat{\mathbf{z}} - \mathbf{z} \rVert_2^2 - \lVert \widehat{\mathbf{z}} - \bar{\mathbf{z}} \rVert_2^2,
\end{aligned}
\end{equation}
which yields the following expanded inequality:
\begin{align}
&2 \rho^{k+1}\Bigg\{\sum_{i=1}^{m_1} (\lambda_i^* - \lambda_i^{k+1}) \bigg[\frac{1}{2} (\mathbf{y}^{k+1})^T P_i \mathbf{y}^{k+1} + \mathbf{q}_0^T \mathbf{y}^{k+1} + \mathbf{c}_0^T \mathbf{v}^{k+1} + r_i\bigg] \nonumber \\
&\hspace*{32pt} + (\bm{\gamma}^* - \bm{\gamma}^{k+1})^T (A \mathbf{y}^{k+1} + B \mathbf{v}^{k+1} - \mathbf{b})\Bigg\} \nonumber \\
\leq &\ \Big(\lVert \bm{\lambda}^k - \bm{\lambda}^* \rVert_2^2 + \lVert \bm{\gamma}^k - \bm{\gamma}^* \rVert_2^2\Big) \nonumber \\
- &\ \Big(\lVert \bm{\lambda}^{k+1} - \bm{\lambda}^* \rVert_2^2 + \lVert \bm{\gamma}^{k+1} - \bm{\gamma}^* \rVert_2^2\Big) - \Big(\lVert \bm{\lambda}^{k+1} - \bm{\lambda}^k \rVert_2^2 + \lVert \bm{\gamma}^{k+1} - \bm{\gamma}^k \rVert_2^2\Big). \label{eq: dual_dist_bound} 
\end{align}
Adding the two inequalities \eqref{eq: primal_dist_bound} and \eqref{eq: dual_dist_bound} leads to the second inequality, \eqref{eq: dual distance}, in Proposition \ref{prp: Distance}.
\par \hfill$\Box$
\subsection{Proof of Theorem \ref{thm: Converge}}
By adding the two inequalities \eqref{eq: primal distance} and \eqref{eq: dual distance} in Proposition \ref{prp: Distance}, we have that
\begingroup
\begin{align}
&\lVert \mathbf{x}^{k+1} - \mathbf{x}^* \rVert_2^2 + \lVert \mathbf{u}^{k+1} - \mathbf{u}^* \rVert_2^2 + \lVert \bm{\lambda}^{k+1} - \bm{\lambda}^* \rVert_2^2 + \lVert \bm{\gamma}^{k+1} - \bm{\gamma}^* \rVert_2^2 \nonumber \\
\leq &\ \lVert \mathbf{x}^k - \mathbf{x}^* \rVert_2^2 + \lVert \mathbf{u}^k - \mathbf{u}^* \rVert_2^2 + \lVert \bm{\lambda}^k - \bm{\lambda}^* \rVert_2^2 + \lVert \bm{\gamma}^k - \bm{\gamma}^* \rVert_2^2 \nonumber \\
- &\ \Big(\lVert \mathbf{y}^{k+1} - \mathbf{x}^{k+1} \rVert_2^2 + \lVert \mathbf{v}^{k+1} - \mathbf{u}^{k+1} \rVert_2^2 + \lVert \mathbf{y}^{k+1} - \mathbf{x}^k \rVert_2^2 + \lVert \mathbf{v}^{k+1} - \mathbf{u}^k \rVert_2^2\Big) \nonumber \\
- &\ \Big(\lVert \bm{\mu}^{k+1} - \bm{\lambda}^{k+1} \rVert_2^2 + \lVert \bm{\nu}^{k+1} - \bm{\gamma}^{k+1} \rVert_2^2 + \lVert \bm{\mu}^{k+1} - \bm{\lambda}^k \rVert_2^2 + \lVert \bm{\nu}^{k+1} - \bm{\gamma}^k \rVert_2^2\Big) \nonumber \\
+ &\underbrace{2 \rho^{k+1} (\mathbf{y}^{k+1} - \mathbf{x}^{k+1})^T P_0 (\mathbf{y}^{k+1} - \mathbf{x}^k)}_\text{(a)} \nonumber \\
+ &\sum_{i=1}^{m_1} \underbrace{2 \rho^{k+1} \mu_i^{k+1} (\mathbf{y}^{k+1} - \mathbf{x}^{k+1})^T P_i (\mathbf{y}^{k+1} - \mathbf{x}^k)}_{\text{(b)}_i} \nonumber \\[-4pt]
+ &\underbrace{2\rho^{k+1} \sum_{i=1}^{m_1} (\lambda_i^{k+1} - \mu_i^{k+1}) \big[\frac{1}{2} (\mathbf{y}^{k+1})^T P_i \mathbf{y}^{k+1} - \frac{1}{2} (\mathbf{x}^k)^T P_i \mathbf{x}^k\big]}_\text{(c)} \nonumber \\[-6pt]
+ &\underbrace{2 \rho^{k+1} \sum_{i=1}^{m_1} (\lambda_i^{k+1} - \mu_i^{k+1}) \mathbf{q}_i^T (\mathbf{y}^{k+1} - \mathbf{x}^k)}_\text{(d)} \nonumber \\[-6pt]
+ &\underbrace{2 \rho^{k+1} \sum_{i=1}^{m_1} (\mu_i^{k+1} - \lambda_i^k) \mathbf{q}_i^T (\mathbf{y}^{k+1} - \mathbf{x}^{k+1})}_\text{(e)} \nonumber \\[-6pt]
+ &\underbrace{2 \rho^{k+1} \sum_{i=1}^{m_1} (\mu_i^{k+1} - \lambda_i^k) (P_i \mathbf{x}^k)^T (\mathbf{y}^{k+1} - \mathbf{x}^{k+1})}_\text{(f)} \nonumber \\[-6pt]
+ &\underbrace{2 \rho^{k+1} \sum_{i=1}^{m_1} (\lambda_i^{k+1} - \mu_i^{k+1}) \mathbf{c}_i^T (\mathbf{v}^{k+1} - \mathbf{u}^k)}_\text{(g)} \nonumber \\[-6pt]
+ &\underbrace{2 \rho^{k+1} \sum_{i=1}^{m_1} (\mu_i^{k+1} - \lambda_i^k) \mathbf{c}_i^T (\mathbf{v}^{k+1} - \mathbf{u}^{k+1})}_\text{(h)} \nonumber\\
+ &\underbrace{2 \rho^{k+1} (\bm{\gamma}^{k+1} - \bm{\nu}^{k+1})^T A (\mathbf{y}^{k+1} - \mathbf{x}^k)}_\text{(i)} + \underbrace{2 \rho^{k+1} (\bm{\nu}^{k+1} - \bm{\gamma}^k)^T A (\mathbf{y}^{k+1} - \mathbf{x}^{k+1})}_\text{(j)} \nonumber \\
+ &\underbrace{2 \rho^{k+1} (\bm{\gamma}^{k+1} - \bm{\nu}^{k+1})^T B (\mathbf{v}^{k+1} - \mathbf{u}^k)}_\text{(k)} + \underbrace{2 \rho^{k+1} (\bm{\nu}^{k+1} - \bm{\gamma}^k)^T B (\mathbf{v}^{k+1} - \mathbf{u}^{k+1})}_{\text{(l)}}. \label{eq: adding two inequalities}
\end{align}
\endgroup
Next, we establish an upper bound for each term of the term from (a) to (l) in \eqref{eq: adding two inequalities} using the adaptive step size $\rho^{k+1} = \rho(\mathbf{x}^k, \mathbf{u}^k, \bm{\lambda}^k, \bm{\gamma}^k)$, as defined in \eqref{eq: adaptive step size}.
\begin{enumerate}[label=(\alph*)]
\item First, we want to show that 
\begin{equation}\label{eq: (a)}
\text{(a)} \leq \epsilon_1 \Big(\lVert \mathbf{y}^{k+1} - \mathbf{x}^{k+1} \rVert_2^2 + \lVert \mathbf{y}^{k+1} - \mathbf{x}^k \rVert_2^2\Big).
\end{equation}
To prove this (and several inequalities below), we first show an extension of the Young's inequality\footnote{Young's inequality states that if $a$ and $b$ are two non-negative real numbers, and $p$ and $q$ are real numbers greater than 1 such that $\frac{1}{p} + \frac{1}{q} = 1$, then $ab < \frac{a^p}{p} + \frac{b^q}{q}$.} on vector products that will play a key role in the following proof.

\parindent=20pt Given any two vectors $\mathbf{z}_1, \mathbf{z}_2 \in \mathbb{R}^n$, we have that
\begin{equation}
\mathbf{z}_1^T \mathbf{z}_2 = \sum_{j=1}^n z_{1j} z_{2j} = \sum_{j=1}^n \Big(\frac{1}{\delta} z_{1j}\Big) \Big(\delta z_{2j}\Big) \leq \sum_{j=1}^n \Big\lvert \frac{1}{\delta} z_{1j} \Big\rvert \Big\lvert \delta z_{2j} \Big\rvert,
\end{equation}
where $\delta$ is a non-zero real number. Applying Young's inequality on each summation term with $p = q = 2$, we obtain that 
\begin{equation}\label{eq: Young's Inequality}
\mathbf{z}_1^T \mathbf{z}_2 \leq \sum_{j=1}^n \left[\frac{1}{2} \Big(\frac{1}{\delta} z_{1j}\Big)^2 + \frac{1}{2} \Big(\delta z_{2j}\Big)^2\right] = \frac{1}{2 \delta^2} \lVert \mathbf{z}_1 \rVert_2^2 + \frac{\delta^2}{2} \lVert \mathbf{z}_2 \rVert_2^2.
\end{equation}
Applying \eqref{eq: Young's Inequality} on (a) yields
\begingroup
\begin{align}
\text{(a)} \leq &2 \rho^{k+1} \Big(\frac{1}{2 \delta^2} \lVert \mathbf{y}^{k+1} - \mathbf{x}^{k+1} \rVert_2^2 + \frac{\delta^2}{2} \lVert P_0 (\mathbf{y}^{k+1} - \mathbf{x}^k) \rVert_2^2\Big) \nonumber \\
\leq &2 \rho^{k+1} \Big(\frac{1}{2 \delta^2} \lVert \mathbf{y}^{k+1} - \mathbf{x}^{k+1} \rVert_2^2 + \frac{\delta^2}{2} \vertiii{P_0}_2^2 \lVert \mathbf{y}^{k+1} - \mathbf{x}^k \rVert_2^2\Big) \nonumber \\
\leq &2 \rho^{k+1} \Big(\frac{1}{2 \delta^2} \lVert \mathbf{y}^{k+1} - \mathbf{x}^{k+1} \rVert_2^2 + \frac{\delta^2}{2} \lVert P_0 \rVert_F^2 \lVert \mathbf{y}^{k+1} - \mathbf{x}^k \rVert_2^2\Big). \label{eq:(a)_inter}
\end{align}
\endgroup
The second inequality holds due to the property that given a matrix $A \in \mathbb{R}^{m \times n}$ and a vector $\mathbf{z} \in \mathbb{R}^n$, $\lVert A \mathbf{z} \rVert_2 \leq \vertiii{A}_2 \lVert \mathbf{z} \rVert_2$ (see Theorem~5.6.2 in \cite{horn2012matrix}), where we use the notation $\vertiii{\cdot}_2$ to denote the matrix norm $\vertiii{A}_2 \coloneqq \underset{\mathbf{z} \not= \mathbf{0}}{\sup} \frac{\lVert A \mathbf{z} \rVert_2}{\lVert \mathbf{z} \rVert_2}$. The last inequality holds due to the property $\vertiii{A}_2 \leq \lVert A \rVert_F$ \cite{golub2013matrix}, where $\lVert A \rVert_F \coloneqq \big(\sum_{i=1}^m \sum_{j=1}^n \lvert A_{ij} \rvert^2\big)^{\frac{1}{2}}$ denotes the Frobenius norm.
% which can be calculated in a distributed fashion for a large-scale matrix.

\vspace{4pt}
\parindent=20pt From \eqref{eq:(a)_inter}, if $\lVert P_0 \rVert_F \not= 0$, then letting $\delta^2 = \frac{1}{\lVert P_0 \rVert_F}$ yields 
\begin{equation}
\text{(a)} \leq \rho^{k+1} \lVert P_0 \rVert_F \Big(\lVert \mathbf{y}^{k+1} - \mathbf{x}^{k+1} \rVert_2^2 + \lVert \mathbf{y}^{k+1} - \mathbf{x}^k \rVert_2^2\Big).
\end{equation}
Since $\rho^{k+1} \leq \rho_1 = \frac{\epsilon_1}{\lVert P_0 \rVert_F}$, we obtain \eqref{eq: (a)}.
If, on the other hand, $\lVert P_0 \rVert_F = 0$, then letting $\delta^2 = 1$ yields
\begin{equation}
\text{(a)} \leq \rho^{k+1} \Big(\lVert \mathbf{y}^{k+1} - \mathbf{x}^{k+1} \rVert_2^2 + \lVert \mathbf{y}^{k+1} - \mathbf{x}^k \rVert_2^2\Big).
\end{equation}
Since $\rho^{k+1} \leq \rho_1 = \epsilon_1$, \eqref{eq: (a)} is also obtained.
\vspace{10pt}
\item Here we want to show that 
\begin{equation}\label{eq: (b)}
\sum_{i=1}^{m_1} \text{(b)}_i \leq \epsilon_2 \Big(\lVert \mathbf{y}^{k+1} - \mathbf{x}^{k+1} \rVert_2^2 + \lVert \mathbf{y}^{k+1} - \mathbf{x}^k \rVert_2^2\Big).
\end{equation}
Applying \eqref{eq: Young's Inequality} on each term $\text{(b)}_i$ yields
\begin{equation}
\begin{aligned} 
\text{(b)}_i \leq &2 \rho^{k+1} \mu_i^{k+1} \Big(\frac{1}{2 \delta_i^2} \lVert \mathbf{y}^{k+1} - \mathbf{x}^{k+1} \rVert_2^2 + \frac{\delta_i^2}{2} \vertiii{P_i}_2^2 \lVert \mathbf{y}^{k+1} - \mathbf{x}^k \rVert_2^2\Big) \\
\leq &2 \rho^{k+1} \mu_i^{k+1} \Big(\frac{1}{2 \delta_i^2} \lVert \mathbf{y}^{k+1} - \mathbf{x}^{k+1} \rVert_2^2 + \frac{\delta_i^2}{2} \lVert P_i \rVert_F^2 \lVert \mathbf{y}^{k+1} - \mathbf{x}^k \rVert_2^2\Big).
\end{aligned}
\end{equation}
\begin{itemize}[label=$\bullet$]
\vspace{4pt}
\item If $\lVert P_i \rVert_F \not= 0$, then letting $\delta_i^2 = \frac{1}{\lVert P_i \rVert_F}$ yields 
\begin{equation}
\begin{aligned}
\text{(b)}_i \leq &\rho^{k+1} \mu_i^{k+1} \lVert P_i \rVert_F \Big(\lVert \mathbf{y}^{k+1} - \mathbf{x}^{k+1} \rVert_2^2 + \lVert \mathbf{y}^{k+1} - \mathbf{x}^k \rVert_2^2\Big) \\
\leq &\rho^{k+1} \tilde{\mu}_i^{k+1} \lVert P_i \rVert_F \Big(\lVert \mathbf{y}^{k+1} - \mathbf{x}^{k+1} \rVert_2^2 + \lVert \mathbf{y}^{k+1} - \mathbf{x}^k \rVert_2^2\Big),
\end{aligned}
\end{equation}
where $\tilde{\mu}_i^{k+1} \coloneqq \lambda_i^k + \rho^{k+1} \lvert \frac{1}{2} (\mathbf{x}^k)^T P_i \mathbf{x}^k + \mathbf{q}_i^T \mathbf{x}^k + \mathbf{c}_i^T \mathbf{u}^k + r_i \rvert \geq \mu_i^{k+1}$. If we can bound $\rho^{k+1} \tilde{\mu}_i^{k+1} \leq \frac{\epsilon_2}{m_1 \lVert P_i \rVert_F}$, then we can achieve 
\begin{equation}\label{eq: (b)_i}
\text{(b)}_i \leq \frac{\epsilon_2}{m_1} \Big(\lVert \mathbf{y}^{k+1} - \mathbf{x}^{k+1} \rVert_2^2 + \lVert \mathbf{y}^{k+1} - \mathbf{x}^k \rVert_2^2\Big).
\end{equation}
By substituting $a_i = \lvert \frac{1}{2} (\mathbf{x}^k)^T P_i \mathbf{x}^k + \mathbf{q}_i^T \mathbf{x}^k + \mathbf{c}_i^T \mathbf{u}^k + r_i \rvert \geq 0$, $b_i = \lambda_i^k \geq 0$ and $c_i = \frac{\epsilon_2}{m_1 \lVert P_i \rVert_F} > 0$, we can rewrite $\rho^{k+1} \tilde{\mu}_i^{k+1} - \frac{\epsilon_2}{m_1 \lVert P_i \rVert_F}$ as $a_i(\rho^{k+1})^2 + b_i \rho^{k+1} - c_i$, which is simply a quadratic function of $\rho^{k+1}$ with parameters $a_i$, $b_i$ and $c_i$. To bound $\rho^{k+1} \tilde{\mu}_i^{k+1} \leq \frac{\epsilon_2}{m_1 \lVert P_i \rVert_F}$ is equivalent to find proper values of $\rho^{k+1}$ that keep the quadratic function stay below zero.
\begin{itemize}
\item If $a_i = 0$ and $b_i = 0$, then $\rho^{k+1} \in (0, +\infty)$.
\vspace{2pt}
\item If $a_i = 0$ and $b_i > 0$, then $\rho^{k+1} \in (0, \frac{c_i}{b_i}]$.
\vspace{2pt}
\item If $a_i > 0$, then $\rho^{k+1} \in (0, \frac{-b_i + \sqrt{b_i^2 + 4 a_i c_i}}{2 a_i}]$.
\end{itemize}
\vspace{2pt}
Since $\rho^{k+1} \leq \rho_2(\mathbf{x}^k, \mathbf{u}^k, \bm{\lambda}^k) \leq \rho_{2i}(\mathbf{x}^k, \mathbf{u}^k, \bm{\lambda}^k)$, it satisfies all the above three conditions, we then obtain \eqref{eq: (b)_i}, and hence \eqref{eq: (b)}.
\vspace{4pt}
\item If $\lVert P_i \rVert_F = 0$, then letting $\delta_i^2 = 1$ yields 
\begin{equation}
\begin{aligned}
\text{(b)}_i \leq &\rho^{k+1} \mu_i^{k+1} \Big(\lVert \mathbf{y}^{k+1} - \mathbf{x}^{k+1} \rVert_2^2 + \lVert \mathbf{y}^{k+1} - \mathbf{x}^k \rVert_2^2\Big) \\
\leq &\rho^{k+1} \tilde{\mu}_i^{k+1} \Big(\lVert \mathbf{y}^{k+1} - \mathbf{x}^{k+1} \rVert_2^2 + \lVert \mathbf{y}^{k+1} - \mathbf{x}^k \rVert_2^2\Big).
\end{aligned}
\end{equation}
Similarly, if we can bound $\rho^{k+1} \tilde{\mu}_i^{k+1} \leq \frac{\epsilon_2}{m_1}$, then we can also achieve \eqref{eq: (b)_i}. By substituting $a_i = \lvert \frac{1}{2} (\mathbf{x}^k)^T P_i \mathbf{x}^k + \mathbf{q}_i^T \mathbf{x}^k + \mathbf{c}_i^T \mathbf{u}^k + r_i \rvert \geq 0$, $b_i = \lambda_i^k \geq 0$ and $c_i = \frac{\epsilon_2}{m_1} > 0$, we can rewrite $\rho^{k+1} \tilde{\mu}_i^{k+1} - \frac{\epsilon_2}{m_1}$ as $a_i(\rho^{k+1})^2 + b_i \rho^{k+1} - c_i$. The same analysis can be followed as discussed in the case of $\lVert P_i \rVert_F \not= 0$.
\end{itemize}
\vspace{10pt}
\item Next, we want to show that 
\begin{equation}\label{eq: (c)}
\text{(c)} \leq \epsilon_3 \Big(\lVert \bm{\lambda}^{k+1} - \bm{\mu}^{k+1} \rVert_2^2 + \lVert \mathbf{y}^{k+1} - \mathbf{x}^k \rVert_2^2\Big).
\end{equation}
By using $P$ to denote $\left(\begin{array}{c}P_1 \\[-2pt] \vdots \\[-2pt] P_{m_1}\end{array}\right)$, we can rewrite 
\begin{equation}\label{eq:(c)_rewrite}
\text{(c)} = \rho^{k+1} \Big\{(\bm{\lambda}^{k+1} - \bm{\mu}^{k+1})^T \Big[I_{m_1 \times m_1} \otimes (\mathbf{x}^k + \mathbf{y}^{k+1})^T\Big] P (\mathbf{y}^{k+1} - \mathbf{x}^k)\Big\},
\end{equation}
where $\otimes$ denotes the Kronecker product; that is, given a matrix $A \in \mathbb{R}^{m_1 \times n_1}$ and a matrix $B \in \mathbb{R}^{m_2 \times n_2}$, $A \otimes B \coloneqq \left(\begin{array}{ccc}a_{11} B&\cdots&a_{1 m_1} B \\[-2pt] \vdots&&\vdots \\[-2pt] a_{m_1 1} B&\cdots&a_{m_1 n_1} B\end{array}\right)$.
Applying \eqref{eq: Young's Inequality} to \eqref{eq:(c)_rewrite} yields 
\begin{equation}
\begin{aligned}
\text{(c)} \leq &\ \rho^{k+1} \Big(\frac{1}{2 \delta^2} \lVert \bm{\lambda}^{k+1} - \bm{\mu}^{k+1} \rVert_2^2 \\
&\hspace*{28pt} + \frac{\delta^2}{2} \vertiii{I_{m_1 \times m_1} \otimes (\mathbf{x}^k + \mathbf{y}^{k+1})^T}_2^2 \vertiii{P}_2^2 \lVert \mathbf{y}^{k+1} - \mathbf{x}^k \rVert_2^2\Big) \\
\leq &\ \rho^{k+1} \Big(\frac{1}{2 \delta^2} \lVert \bm{\lambda}^{k+1} - \bm{\mu}^{k+1} \rVert_2^2 + \frac{\delta^2}{2} \lVert \mathbf{x}^k + \mathbf{y}^{k+1} \rVert_2^2 \lVert P \rVert_F^2 \lVert \mathbf{y}^{k+1} - \mathbf{x}^k \rVert_2^2\Big).
\end{aligned}
\end{equation}
Since we have the property that $\vertiii{A \otimes B}_2 = \vertiii{A}_2 \vertiii{B}_2$ (see Theorem~8 in \cite{lancaster1972norms}), the last inequality holds due to 
\begin{equation}
\vertiii{I_{m_1 \times m_1} \otimes (\mathbf{x}^k + \mathbf{y}^{k+1})^T}_2^2 =  \vertiii{I_{m_1 \times m_1}}_2^2 \vertiii{(\mathbf{x}^k + \mathbf{y}^{k+1})^T}_2^2,
\end{equation}
together with $\vertiii{I_{m_1 \times m_1}}_2 = 1$ and $\vertiii{(\mathbf{x}^k + \mathbf{y}^{k+1})^T}_2 \leq \lVert (\mathbf{x}^k + \mathbf{y}^{k+1})^T \rVert_F = \lVert \mathbf{x}^k + \mathbf{y}^{k+1} \rVert_2$. Note that $\lVert P \rVert_F \not= 0$, otherwise the QCQP is simply a QP.
\begin{itemize}[label=$\bullet$]
\vspace{4pt}
\item If $\lVert \mathbf{x}^k + \mathbf{y}^{k+1} \rVert_2 \not= 0$, then letting $\delta^2 = \frac{1}{\lVert \mathbf{x}^k + \mathbf{y}^{k+1} \rVert_2 \lVert P \rVert_F}$ yields
\begin{equation}
\begin{aligned}
\text{(c)} \leq &\frac{1}{2} \rho^{k+1} \lVert \mathbf{x}^k + \mathbf{y}^{k+1} \rVert_2 \lVert P \rVert_F \Big(\lVert \bm{\lambda}^{k+1} - \bm{\mu}^{k+1} \rVert_2^2 + \lVert \mathbf{y}^{k+1} - \mathbf{x}^k \rVert_2^2\Big) \\
\leq &\frac{1}{2} \rho^{k+1} \lVert \mathbf{x}^k + \tilde{\mathbf{y}}^{k+1} \rVert_2 \lVert P \rVert_F \Big(\lVert \bm{\lambda}^{k+1} - \bm{\mu}^{k+1} \rVert_2^2 + \lVert \mathbf{y}^{k+1} - \mathbf{x}^k \rVert_2^2\Big),
\end{aligned}
\end{equation}
where $\tilde{y}_j^{k+1} \coloneqq x_j^k + \rho \Big\lvert \big[P_0 \mathbf{x}^k + \mathbf{q}_0 + \sum_{i=1}^{m_1} \lambda_i^k \big(P_i \mathbf{x}^k + \mathbf{q}_i\big) + A^T \bm{\gamma}^k\big]_j \Big\rvert \geq y_j^{k+1}$. If we can bound $\rho^{k+1} \lVert \mathbf{x}^k + \tilde{\mathbf{y}}^{k+1} \rVert_2 \leq \frac{2 \epsilon_3}{\lVert P^T \rVert_F}$, then \eqref{eq: (c)} can be obtained. We first bound
\begingroup
\begin{align}
&\rho^{k+1} \lVert \mathbf{x}^k + \tilde{\mathbf{y}}^{k+1} \rVert_2 - \frac{2 \epsilon_3}{\lVert P \rVert_F} \nonumber \\[-4pt]
\leq &\rho^{k+1} \Big[2 \lVert \mathbf{x}^k \rVert_2  + \rho^{k+1} \lVert P_0 \mathbf{x}^k + \mathbf{q}_0 + \sum_{i=1}^{m_1} \lambda_i^k \big(P_i \mathbf{x}^k + \mathbf{q}_i\big) + A^T \bm{\gamma}^k \rVert_2\Big] \nonumber \\[-4pt]
- &\frac{2 \epsilon_3}{\lVert P \rVert_F}.
\end{align}
\endgroup
By substituting $a = \lVert P_0 \mathbf{x}^k + \mathbf{q}_0 + \sum_{i=1}^{m_1} \lambda_i^k \big(P_i \mathbf{x}^k + \mathbf{q}_i\big) + A^T \bm{\gamma}^k \rVert_2 \geq 0$, $b = 2 \lVert \mathbf{x}^k \rVert_2 \geq 0$ and $c = \frac{2 \epsilon_3}{\lVert P \rVert_F} > 0$, we can bound $\rho^{k+1} \lVert \mathbf{x}^k + \tilde{\mathbf{y}}^{k+1} \rVert_2 - \frac{2 \epsilon_3}{\lVert P \rVert_F}$ using $a (\rho^{k+1})^2 + b \rho^{k+1} - c$, which is simply a quadratic function of $\rho^{k+1}$ with parameters $a$, $b$ and $c$. Bounding $\rho^{k+1} \lVert \mathbf{x}^k + \tilde{\mathbf{y}}^{k+1} \rVert_2 \leq \frac{2 \epsilon_3}{\lVert P \rVert_F}$ can be guaranteed by finding the proper values of $\rho^{k+1}$ that keep the quadratic function stay below zero.
\begin{itemize}
\vspace{2pt}
\item If $a = 0$ and $b = 0$, then $\rho^{k+1} \in (0, +\infty)$.
\vspace{2pt}
\item If $a = 0$ and $b > 0$, then $\rho^{k+1} \in (0, \frac{c}{b}]$.
\vspace{2pt}
\item If $a > 0$, then $\rho^{k+1} \in (0, \frac{-b + \sqrt{b^2 + 4 a c}}{2 a}]$.
\end{itemize}
\vspace{2pt}
Since $\rho^{k+1} \leq \rho_3(\mathbf{x}^k, \bm{\lambda}^k, \bm{\gamma}^k)$, it satisfies all the above three conditions, we obtain \eqref{eq: (c)}.
\vspace{4pt}
\item If $\lVert \mathbf{x}^k + \mathbf{y}^{k+1} \rVert_2 = 0$, then letting $\delta^2 = 1$ yields
\begin{equation}
(c) \leq \frac{1}{2} \rho^{k+1} \Big(\lVert \bm{\lambda}^{k+1} - \bm{\mu}^{k+1} \rVert_2^2 + \lVert \mathbf{y}^{k+1} - \mathbf{x}^k \rVert_2^2\Big)
\end{equation}
Since $\rho^{k+1} \leq \rho_3(\mathbf{x}^k, \bm{\lambda}^k, \bm{\gamma}^k) \leq 2 \epsilon_3$, \eqref{eq: (c)} is also obtained.
\end{itemize}
\vspace{10pt}
\item To show that 
\begin{equation}\label{eq: (d)}
\text{(d)} \leq \epsilon_4 \Big(\lVert \bm{\lambda}^{k+1} - \bm{\mu}^{k+1} \rVert_2^2 + \lVert \mathbf{y}^{k+1} - \mathbf{x}^k \rVert_2^2\Big),
\end{equation}
by letting $Q = \left(\begin{array}{c}\mathbf{q}_1^T \\[-2pt] \vdots \\ \mathbf{q}_{m_1}^T\end{array}\right)$, we can rewrite that 
\begin{equation}\label{eq: (d)_rewrite}
\text{(d)} = 2 \rho^{k+1} \Big[(\bm{\lambda}^{k+1} - \bm{\mu}^{k+1})^T Q^T (\mathbf{y}^{k+1} - \mathbf{x}^k)\Big].
\end{equation}
Applying \eqref{eq: Young's Inequality} to \eqref{eq: (d)_rewrite} yields
\begin{equation}
\begin{aligned}
\text{(d)} \leq &2 \rho^{k+1} \Big(\frac{1}{2 \delta^2} \lVert \bm{\lambda}^{k+1} - \bm{\mu}^{k+1} \rVert_2^2 + \frac{\delta^2}{2} \vertiii{Q}_2^2 \lVert \mathbf{y}^{k+1} - \mathbf{x}^k \rVert_2^2\Big) \\
\leq &2 \rho^{k+1} \big(\frac{1}{2 \delta^2} \lVert \bm{\lambda}^{k+1} - \bm{\mu}^{k+1} \rVert_2^2 + \frac{\delta^2}{2} \lVert Q \rVert_F^2 \lVert \mathbf{y}^{k+1} - \mathbf{x}^k \rVert_2^2\big).
\end{aligned}
\end{equation}
\begin{itemize}[label=$\bullet$]
\item If $\lVert Q \rVert_F \not= 0$, then letting $\delta^2 = \frac{1}{\lVert Q \rVert_F}$ yields
\begin{equation}
\text{(d)} \leq \rho^{k+1} \lVert Q \rVert_F \Big(\lVert \bm{\lambda}^{k+1} - \bm{\mu}^{k+1} \rVert_2^2 + \lVert \mathbf{y}^{k+1} - \mathbf{x}^k \rVert_2^2\Big).
\end{equation}
Since $\rho^{k+1} \leq \rho_4 = \frac{\epsilon_4}{\lVert Q \rVert_F}$, we obtain \eqref{eq: (d)}.
\vspace{4pt}
\item If $\lVert Q \rVert_F = 0$, then letting $\delta^2 = 1$ yields
\begin{equation}
\text{(d)} \leq \rho^{k+1} \Big(\lVert \bm{\lambda}^{k+1} - \bm{\mu}^{k+1} \rVert_2^2 + \lVert \mathbf{y}^{k+1} - \mathbf{x}^k \rVert_2^2\Big).  
\end{equation}
Since $\rho^{k+1} \leq \rho_4 = \epsilon_4$, \eqref{eq: (d)} is also obtained.
\end{itemize}
\vspace{10pt}
\item Similarly, to show 
\begin{equation}\label{eq: (e)}
\text{(e)} \leq \epsilon_4 \Big(\lVert \bm{\mu}^{k+1} - \bm{\lambda}^k \rVert_2^2 + \lVert \mathbf{y}^{k+1} - \mathbf{x}^{k+1} \rVert_2^2\Big),
\end{equation}
we can rewrite that 
\begin{equation}\label{eq:(e)_rewrite}
\text{(e)} = 2 \rho^{k+1} \Big[(\bm{\mu}^{k+1} - \bm{\lambda}^k)^T Q (\mathbf{y}^{k+1} - \mathbf{x}^{k+1})\Big].
\end{equation}
Applying \eqref{eq: Young's Inequality} to \eqref{eq:(e)_rewrite} yields that 
\begin{equation}
\begin{aligned}
\text{(e)} \leq &2 \rho^{k+1} \Big(\frac{1}{2 \delta^2} \lVert \bm{\mu}^{k+1} - \bm{\lambda}^k \rVert_2^2 + \frac{\delta^2}{2} \vertiii{Q}_2^2 \lVert \mathbf{y}^{k+1} - \mathbf{x}^{k+1} \rVert_2^2\Big) \\
\leq &2 \rho^{k+1} \Big(\frac{1}{2 \delta^2} \lVert \bm{\mu}^{k+1} - \bm{\lambda}^k \rVert_2^2 + \frac{\delta^2}{2} \lVert Q \rVert_F^2 \lVert \mathbf{y}^{k+1} - \mathbf{x}^{k+1} \rVert_2^2\Big).
\end{aligned}
\end{equation}
\begin{itemize}[label=$\bullet$]
\item If $\lVert Q \rVert_F \not= 0$, then letting $\delta^2 = \frac{1}{\lVert Q \rVert_F}$ yields
\begin{equation}
\text{(e)} \leq \rho^{k+1} \lVert Q \rVert_F \Big(\lVert \bm{\mu}^{k+1} - \bm{\lambda}^k \rVert_2^2 + \lVert \mathbf{y}^{k+1} - \mathbf{x}^{k+1} \rVert_2^2\Big).  
\end{equation}
Since $\rho^{k+1} \leq \rho_4 = \frac{\epsilon_4}{\lVert Q \rVert_F}$, we obtain \eqref{eq: (e)}.
\vspace{4pt}
\item If $\lVert Q \rVert_F = 0$, then letting $\delta^2 = 1$ yields
\begin{equation}
\text{(e)} \leq \rho^{k+1} \Big(\lVert \bm{\mu}^{k+1} - \bm{\lambda}^k \rVert_2^2 + \lVert \mathbf{y}^{k+1} - \mathbf{x}^{k+1} \rVert_2^2\Big).  
\end{equation}
Since $\rho^{k+1} \leq \rho_4 = \epsilon_4$, \eqref{eq: (e)} is also obtained.
\end{itemize}
\vspace{10pt}
\item To show 
\begin{equation}\label{eq: (f)}
\text{(f)} \leq \epsilon_5 \Big(\lVert \bm{\mu}^{k+1} - \bm{\lambda}^k \rVert_2^2 + \lVert \mathbf{y}^{k+1} - \mathbf{x}^{k+1} \rVert_2^2\Big),
\end{equation}
 we can rewrite
\begin{equation}
\text{(f)} = 2 \rho^{k+1} \Big\{(\bm{\mu}^{k+1} - \bm{\lambda}^k)^T \Big[I_{m_1 \times m_1} \otimes (\mathbf{x}^k)^T\Big] P (\mathbf{y}^{k+1} - \mathbf{x}^{k+1})\Big\}.
\end{equation}
Applying \eqref{eq: Young's Inequality}, we have that 
\begin{equation}
\begin{aligned}
\text{(f)} \leq &2 \rho^{k+1} \Big(\frac{1}{2 \delta^2} \lVert \bm{\mu}^{k+1} - \bm{\lambda}^k \rVert_2^2 \\
&\hspace*{30pt} + \frac{\delta^2}{2} \vertiii{I_{m_1 \times m_1} \otimes (\mathbf{x}^k)^T}_2^2 \vertiii{P}_2^2 \lVert \mathbf{y}^{k+1} - \mathbf{x}^{k+1} \rVert_2^2\Big) \\
\leq &2 \rho^{k+1} \Big(\frac{1}{2 \delta^2} \lVert \bm{\mu}^{k+1} - \bm{\lambda}^k \rVert_2^2 + \frac{\delta^2}{2} \lVert \mathbf{x}^k \rVert_2^2 \lVert P \rVert_F^2 \lVert \mathbf{y}^{k+1} - \mathbf{x}^{k+1} \rVert_2^2\Big).
\end{aligned}
\end{equation}
Similarly, the last inequality holds due to 
$$
\vertiii{I_{m_1 \times m_1} \otimes (\mathbf{x}^k)^T}_2^2 = \vertiii{I_{m_1 \times m_1}}_2^2 \vertiii{(\mathbf{x}^k)^T}_2^2.
$$
\begin{itemize}[label=$\bullet$]
\item If $\lVert \mathbf{x}^k \rVert_2 \not= 0$, then letting $\delta^2 = \frac{1}{\lVert \mathbf{x}^k \rVert_2 \lVert P \rVert_F}$ yields
\begin{equation}
\text{(f)} \leq \rho^{k+1} \lVert \mathbf{x}^k \rVert_2 \lVert P \rVert_F \Big(\lVert \bm{\mu}^{k+1} - \bm{\lambda}^k \rVert_2^2 + \lVert \mathbf{y}^{k+1} - \mathbf{x}^{k+1} \rVert_2^2\Big).
\end{equation}
Since $\rho^{k+1} \leq \rho_5(\mathbf{x}^k) = \frac{\epsilon_5}{\lVert \mathbf{x}^k \rVert_2 \lVert P \rVert_F}$, we obtain \eqref{eq: (f)}.
\vspace{4pt}
\item If $\lVert \mathbf{x}^k \rVert_2 = 0$, then letting $\delta^2 = 1$ yields
\begin{equation}
\text{(f)} \leq \rho^{k+1} \Big(\lVert \bm{\mu}^{k+1} - \bm{\lambda}^k \rVert_2^2 + \lVert \mathbf{y}^{k+1} - \mathbf{x}^{k+1} \rVert_2^2\Big).
\end{equation}
Since $\rho^{k+1} \leq \rho_5(\mathbf{x}^k) = \epsilon_5$, \eqref{eq: (f)} is also obtained.
\end{itemize}
\vspace{10pt}
\item To show
\begin{equation}\label{eq: (g)}
\text{(g)} \leq \epsilon_6 \Big(\lVert \bm{\lambda}^{k+1} - \bm{\mu}^{k+1} \rVert_2^2 + \lVert \mathbf{v}^{k+1} - \mathbf{u}^k \rVert_2^2\Big), 
\end{equation}
By letting $C = \left(\begin{array}{c}\mathbf{c}_1^T \\[-2pt] \vdots \\ \mathbf{c}_{m_2}^T\end{array}\right)$, we can rewrite
\begin{equation}
\text{(g)} = 2 \rho^{k+1} \Big[(\bm{\lambda}^{k+1} - \bm{\mu}^{k+1})^T C (\mathbf{v}^{k+1} - \mathbf{u}^k)\Big].
\end{equation}
Applying \eqref{eq: Young's Inequality}, we have that 
\begin{equation}
\begin{aligned}
\text{(g)} \leq &2 \rho^{k+1} \Big(\frac{1}{2 \delta^2} \lVert \bm{\lambda}^{k+1} - \bm{\mu}^{k+1} \rVert_2^2 + \frac{\delta^2}{2} \vertiii{C}_2^2 \lVert \mathbf{v}^{k+1} - \mathbf{u}^k \rVert_2^2\Big) \\
\leq &2 \rho^{k+1} \Big(\frac{1}{2 \delta^2} \lVert \bm{\lambda}^{k+1} - \bm{\mu}^{k+1} \rVert_2^2 + \frac{\delta^2}{2} \lVert C \rVert_F^2 \lVert \mathbf{v}^{k+1} - \mathbf{u}^k \rVert_2^2\Big).
\end{aligned}
\end{equation}
\begin{itemize}[label=$\bullet$]
\item If $\lVert C \rVert_F \not= 0$, then letting $\delta^2 = \frac{1}{\lVert C \rVert_F}$ yields
\begin{equation}
\text{(g)} \leq \rho^{k+1} \lVert C \rVert_F \Big(\lVert \bm{\lambda}^{k+1} - \bm{\mu}^{k+1} \rVert_2^2 + \lVert \mathbf{v}^{k+1} - \mathbf{u}^k \rVert_2^2\Big).
\end{equation}
Since $\rho^{k+1} \leq \rho_6 = \frac{\epsilon_6}{\lVert C \rVert_F}$, we obtain \eqref{eq: (g)}.
\vspace{4pt}
\item If $\lVert C \rVert_F = 0$, then letting $\delta^2 = 1$ yields
\begin{equation}
\text{(g)} \leq \rho^{k+1} \Big(\lVert \bm{\lambda}^{k+1} - \bm{\mu}^{k+1} \rVert_2^2 + \lVert \mathbf{v}^{k+1} - \mathbf{u}^k \rVert_2^2\Big).  
\end{equation}
Since $\rho^{k+1} \leq \rho_6 = \epsilon_6$, \eqref{eq: (g)} is also obtained.
\end{itemize}
\vspace{10pt}
\item Next, we want to show that 
\begin{equation}\label{eq: (h)}
\text{(h)} \leq \epsilon_6 \Big(\lVert \bm{\mu}^{k+1} - \bm{\lambda}^k \rVert_2^2 + \lVert \mathbf{v}^{k+1} - \mathbf{u}^{k+1} \rVert_2^2\Big).
\end{equation}
Similarly, we can rewrite
\begin{equation}
\text{(h)} = 2 \rho^{k+1} \Big[(\bm{\mu}^{k+1} - \bm{\lambda}^k)^T C (\mathbf{v}^{k+1} - \mathbf{u}^{k+1})\Big].
\end{equation}
Applying \eqref{eq: Young's Inequality} on the above equality leads to  
\begin{equation}
\begin{aligned}
\text{(h)} \leq &2 \rho^{k+1} \Big(\frac{1}{2 \delta^2} \lVert \bm{\mu}^{k+1} - \bm{\lambda}^k \rVert_2^2 + \frac{\delta^2}{2} \vertiii{C}_2^2 \lVert \mathbf{v}^{k+1} - \mathbf{u}^{k+1} \rVert_2^2\Big) \\
\leq &2 \rho^{k+1} \Big(\frac{1}{2 \delta^2} \lVert \bm{\mu}^{k+1} - \bm{\lambda}^k \rVert_2^2 + \frac{\delta^2}{2} \lVert C \rVert_F^2 \lVert \mathbf{v}^{k+1} - \mathbf{u}^{k+1} \rVert_2^2\Big).
\end{aligned}
\end{equation}
\begin{itemize}[label=$\bullet$]
\item If $\lVert C \rVert_F \not= 0$, then letting $\delta^2 = \frac{1}{\lVert C \rVert_F}$ yields
\begin{equation}
\text{(h)} \leq \rho^{k+1} \lVert C \rVert_F \Big(\lVert \bm{\mu}^{k+1} - \bm{\lambda}^k \rVert_2^2 + \lVert \mathbf{v}^{k+1} - \mathbf{u}^{k+1} \rVert_2^2\Big).  
\end{equation}
Since $\rho^{k+1} \leq \rho_6 = \frac{\epsilon_6}{\lVert C \rVert_F}$, we obtain \eqref{eq: (h)}.
\item If $\lVert C \rVert_F = 0$, then letting $\delta^2 = 1$ yields
\begin{equation}
\text{(h)} \leq \rho^{k+1} \Big(\lVert \bm{\mu}^{k+1} - \bm{\lambda}^k \rVert_2^2 + \lVert \mathbf{v}^{k+1} - \mathbf{u}^{k+1} \rVert_2^2\Big).  
\end{equation}
Since $\rho^{k+1} \leq \rho_6 = \epsilon_6$, \eqref{eq: (h)} is also obtained.
\end{itemize}
\vspace{10pt}
\item To show 
\begin{equation}\label{eq: (i)}
\text{(i)} \leq \epsilon_7 \Big(\lVert \bm{\gamma}^{k+1} - \bm{\nu}^{k+1} \rVert_2^2 + \lVert \mathbf{y}^{k+1} - \mathbf{x}^k \rVert_2^2\Big),
\end{equation}
we apply \eqref{eq: Young's Inequality} on the rewriting of (i), which leads to 
\begin{equation}
\begin{aligned}
\text{(i)} \leq &2 \rho^{k+1} \Big(\frac{1}{2 \delta^2} \lVert \bm{\gamma}^{k+1} - \bm{\nu}^{k+1} \rVert_2^2 + \frac{\delta^2}{2} \vertiii{A}_2^2 \lVert \mathbf{y}^{k+1} - \mathbf{x}^k \rVert_2^2\Big) \\
\leq &2 \rho^{k+1} \Big(\frac{1}{2 \delta^2} \lVert \bm{\gamma}^{k+1} - \bm{\nu}^{k+1} \rVert_2^2 + \frac{\delta^2}{2} \lVert A \rVert_F^2 \lVert \mathbf{y}^{k+1} - \mathbf{x}^k \rVert_2^2\Big).
\end{aligned}
\end{equation}
\begin{itemize}[label=$\bullet$]
\item If $\lVert A \rVert_F \not= 0$, then letting $\delta^2 = \frac{1}{\lVert A \rVert_F}$ yields
\begin{equation}
\text{(i)} \leq \rho^{k+1} \lVert A \rVert_F \Big(\lVert \bm{\gamma}^{k+1} - \bm{\nu}^{k+1} \rVert_2^2 + \lVert \mathbf{y}^{k+1} - \mathbf{x}^k \rVert_2^2\Big).
\end{equation}
Since $\rho^{k+1} \leq \rho_7 = \frac{\epsilon_7}{\lVert A \rVert_F}$, we obtain \eqref{eq: (i)}.
\vspace{4pt}
\item If $\lVert A \rVert_F = 0$, then letting $\delta^2 = 1$ yields
\begin{equation}
\text{(i)} \leq \rho^{k+1} \Big(\lVert \bm{\gamma}^{k+1} - \bm{\nu}^{k+1} \rVert_2^2 + \lVert \mathbf{y}^{k+1} - \mathbf{x}^k \rVert_2^2\Big).  
\end{equation}
Since $\rho^{k+1} \leq \rho_7 = \epsilon_7$, \eqref{eq: (i)} is also obtained.
\end{itemize}
\vspace{10pt}
\item Similarly, to show 
\begin{equation}\label{eq: (j)}
\text{(j)} \leq \epsilon_7 \Big(\lVert \bm{\nu}^{k+1} - \bm{\gamma}^k \rVert_2^2 + \lVert \mathbf{y}^{k+1} - \mathbf{x}^{k+1} \rVert_2^2\Big).
\end{equation}
we apply \eqref{eq: Young's Inequality} on the rewriting of (j), which yields 
\begin{equation}
\begin{aligned}
\text{(j)} \leq &2 \rho^{k+1} \Big(\frac{1}{2 \delta^2} \lVert \bm{\nu}^{k+1} - \bm{\gamma}^k \rVert_2^2 + \frac{\delta^2}{2} \vertiii{A}_2^2 \lVert \mathbf{y}^{k+1} - \mathbf{x}^{k+1} \rVert_2^2\Big) \\
\leq &2 \rho^{k+1} \Big(\frac{1}{2 \delta^2} \lVert \bm{\nu}^{k+1} - \bm{\gamma}^k \rVert_2^2 + \frac{\delta^2}{2} \lVert A \rVert_F^2 \lVert \mathbf{y}^{k+1} - \mathbf{x}^{k+1} \rVert_2^2\Big).
\end{aligned}
\end{equation}
\begin{itemize}[label=$\bullet$]
\item If $\lVert A \rVert_F \not= 0$, then letting $\delta^2 = \frac{1}{\lVert A \rVert_F}$ yields
\begin{equation}
\text{(j)} \leq \rho^{k+1} \lVert A \rVert_F \Big(\lVert \bm{\nu}^{k+1} - \bm{\gamma}^k \rVert_2^2 + \lVert \mathbf{y}^{k+1} - \mathbf{x}^{k+1} \rVert_2^2\Big).  
\end{equation}
Since $\rho^{k+1} \leq \rho_7 = \frac{\epsilon_7}{\lVert A \rVert_F}$, we obtain \eqref{eq: (j)}.
\vspace{4pt}
\item If $\lVert A \rVert_F = 0$, then letting $\delta^2 = 1$ yields
\begin{equation}
\text{(j)} \leq \rho^{k+1} \Big(\lVert \bm{\nu}^{k+1} - \bm{\gamma}^k \rVert_2^2 + \lVert \mathbf{y}^{k+1} - \mathbf{x}^{k+1} \rVert_2^2\Big).  
\end{equation}
Since $\rho^{k+1} \leq \rho_7 = \epsilon_7$, \eqref{eq: (j)} is also obtained.
\end{itemize}
\vspace{10pt}
\item Next, to show 
\begin{equation}\label{eq: (k)}
\text{(k)} \leq \epsilon_8 \Big(\lVert \bm{\gamma}^{k+1} - \bm{\nu}^{k+1} \rVert_2^2 + \lVert \mathbf{v}^{k+1} - \mathbf{u}^k \rVert_2^2\Big),
\end{equation}
we apply \eqref{eq: Young's Inequality} on the rewriting of (k): 
\begin{equation}
\begin{aligned}
\text{(k)} \leq &2 \rho^{k+1} \Big(\frac{1}{2 \delta^2} \lVert \bm{\gamma}^{k+1} - \bm{\nu}^{k+1} \rVert_2^2 + \frac{\delta^2}{2} \vertiii{B}_2^2 \lVert \mathbf{v}^{k+1} - \mathbf{u}^k \rVert_2^2\Big) \\
\leq &2 \rho^{k+1} \Big(\frac{1}{2 \delta^2} \lVert \bm{\gamma}^{k+1} - \bm{\nu}^{k+1} \rVert_2^2 + \frac{\delta^2}{2} \lVert B \rVert_F^2 \lVert \mathbf{v}^{k+1} - \mathbf{u}^k \rVert_2^2\Big).
\end{aligned}
\end{equation}
\begin{itemize}[label=$\bullet$]
\item If $\lVert B \rVert_F \not= 0$, then letting $\delta^2 = \frac{1}{\lVert B \rVert_F}$ yields
\begin{equation}
\text{(k)} \leq \rho^{k+1} \lVert B \rVert_F \Big(\lVert \bm{\gamma}^{k+1} - \bm{\nu}^{k+1} \rVert_2^2 + \lVert \mathbf{v}^{k+1} - \mathbf{u}^k \rVert_2^2\Big).
\end{equation}
Since $\rho^{k+1} \leq \rho_8 = \frac{\epsilon_8}{\lVert B \rVert_F}$, we obtain \eqref{eq: (k)}.
\vspace{4pt}
\item If $\lVert B \rVert_F = 0$, then letting $\delta^2 = 1$ yields
\begin{equation}
\text{(k)} \leq \rho^{k+1} \Big(\lVert \bm{\gamma}^{k+1} - \bm{\nu}^{k+1} \rVert_2^2 + \lVert \mathbf{v}^{k+1} - \mathbf{u}^k \rVert_2^2\Big).  
\end{equation}
Since $\rho^{k+1} \leq \rho_8 = \epsilon_8$, \eqref{eq: (k)} is also obtained.
\end{itemize}
\vspace{10pt}
\item Last, to show 
\begin{equation}\label{eq: (l)}
\text{(l)} \leq \epsilon_8 \Big(\lVert \bm{\nu}^{k+1} - \bm{\gamma}^k \rVert_2^2 + \lVert \mathbf{v}^{k+1} - \mathbf{u}^{k+1} \rVert_2^2\Big), 
\end{equation}
we apply \eqref{eq: Young's Inequality} on the rewriting of (l): 
\begin{equation}
\begin{aligned}
\text{(l)} \leq &2 \rho^{k+1} \Big(\frac{1}{2 \delta^2} \lVert \bm{\nu}^{k+1} - \bm{\gamma}^k \rVert_2^2 + \frac{\delta^2}{2} \vertiii{B}_2^2 \lVert \mathbf{v}^{k+1} - \mathbf{u}^{k+1} \rVert_2^2\Big) \\
\leq &2 \rho^{k+1} \Big(\frac{1}{2 \delta^2} \lVert \bm{\nu}^{k+1} - \bm{\gamma}^k \rVert_2^2 + \frac{\delta^2}{2} \lVert B \rVert_F^2 \lVert \mathbf{v}^{k+1} - \mathbf{u}^{k+1} \rVert_2^2\Big).
\end{aligned}
\end{equation}
\begin{itemize}[label=$\bullet$]
\item If $\lVert B \rVert_F \not= 0$, then letting $\delta^2 = \frac{1}{\lVert B \rVert_F}$ yields
\begin{equation}
\text{(l)} \leq \rho^{k+1} \lVert B \rVert_F \Big(\lVert \bm{\nu}^{k+1} - \bm{\gamma}^k \rVert_2^2 + \lVert \mathbf{v}^{k+1} - \mathbf{u}^{k+1} \rVert_2^2\Big).  
\end{equation}
Since $\rho^{k+1} \leq \rho_8 = \frac{\epsilon_8}{\lVert B \rVert_F}$, we obtain \eqref{eq: (l)}.
\vspace{4pt}
\item If $\lVert B \rVert_F = 0$, then letting $\delta^2 = 1$ yields
\begin{equation}
\text{(l)} \leq \rho^{k+1} \Big(\lVert \bm{\nu}^{k+1} - \bm{\gamma}^k \rVert_2^2 + \lVert \mathbf{v}^{k+1} - \mathbf{u}^{k+1} \rVert_2^2\Big).  
\end{equation}
Since $\rho^{k+1} \leq \rho_8 = \epsilon_8$, \eqref{eq: (l)} is also obtained.
\end{itemize}
\end{enumerate}
The summation of terms $\text{(a)}$ to $\text{(l)}$ can now be bounded as: 
\begin{equation}
\begin{aligned}
&\text{(a)} + \sum_{i=1}^{m_1} \text{(b)}_i + \text{(c)} + \text{(d)} + \text{(e)} + \text{(f)} + \text{(g)} + \text{(h)} + \text{(i)} + \text{(j)} + \text{(k)} + \text{(l)} \\
\leq &(\epsilon_1 + \epsilon_2 + \epsilon_4 + \epsilon_5 + \epsilon_7) \lVert \mathbf{y}^{k+1} - \mathbf{x}^{k+1} \rVert_2^2 \\
+ &(\epsilon_1 + \epsilon_2 + \epsilon_3 + \epsilon_4 + \epsilon_7) \lVert \mathbf{y}^{k+1} - \mathbf{x}^k \rVert_2^2 \\
+ &(\epsilon_6 + \epsilon_8) \lVert \mathbf{v}^{k+1} - \mathbf{u}^{k+1} \rVert_2^2 + (\epsilon_6 + \epsilon_8) \lVert \mathbf{v}^{k+1} - \mathbf{u}^k \rVert_2^2 \\
+ &(\epsilon_3 + \epsilon_4 + \epsilon_6) \lVert \bm{\mu}^{k+1} - \bm{\lambda}^{k+1} \rVert_2^2 + (\epsilon_4 + \epsilon_5 + \epsilon_6) \lVert \bm{\mu}^{k+1} - \bm{\lambda}^k \rVert_2^2 \\
+ &(\epsilon_7 + \epsilon_8) \lVert \bm{\nu}^{k+1} - \bm{\gamma}^{k+1} \rVert_2^2 + (\epsilon_7 + \epsilon_8) \lVert \bm{\nu}^{k+1} - \bm{\gamma}^k \rVert_2^2 \\
\leq &(\sum_{s=1}^8 \epsilon_s) \Big[\lVert \mathbf{y}^{k+1} - \mathbf{x}^{k+1} \rVert_2^2 + \lVert \mathbf{y}^{k+1} - \mathbf{x}^k \rVert_2^2 \\
&\hspace*{36pt} + \lVert \mathbf{v}^{k+1} - \mathbf{u}^{k+1} \rVert_2^2 + \lVert \mathbf{v}^{k+1} - \mathbf{u}^k \rVert_2^2 \\
&\hspace*{36pt} + \lVert \bm{\mu}^{k+1} - \bm{\lambda}^{k+1} \rVert_2^2 + \lVert \bm{\mu}^{k+1} - \bm{\lambda}^k \rVert_2^2 \\
&\hspace*{36pt} + \lVert \bm{\nu}^{k+1} - \bm{\gamma}^{k+1} \rVert_2^2 + \lVert \bm{\nu}^{k+1} - \bm{\gamma}^k \rVert_2^2\Big] \\
\leq &(1 - \epsilon_0) \Big[\lVert \mathbf{y}^{k+1} - \mathbf{x}^{k+1} \rVert_2^2 + \lVert \mathbf{y}^{k+1} - \mathbf{x}^k \rVert_2^2 \\
&\hspace*{36pt} + \lVert \mathbf{v}^{k+1} - \mathbf{u}^{k+1} \rVert_2^2 + \lVert \mathbf{v}^{k+1} - \mathbf{u}^k \rVert_2^2 \\
&\hspace*{36pt} + \lVert \bm{\mu}^{k+1} - \bm{\lambda}^{k+1} \rVert_2^2 + \lVert \bm{\mu}^{k+1} - \bm{\lambda}^k \rVert_2^2 \\
&\hspace*{36pt} + \lVert \bm{\nu}^{k+1} - \bm{\gamma}^{k+1} \rVert_2^2 + \lVert \bm{\nu}^{k+1} - \bm{\gamma}^k \rVert_2^2\Big].
\end{aligned}
\end{equation}
Substituting it back into \eqref{eq: adding two inequalities}, we have that for all $k \geq 0$, 
\begingroup
\begin{align}
&\lVert \mathbf{x}^{k+1} - \mathbf{x}^* \rVert_2^2 + \lVert \mathbf{u}^{k+1} - \mathbf{u}^* \rVert_2^2 + \lVert \bm{\lambda}^{k+1} - \bm{\lambda}^* \rVert_2^2 + \lVert \bm{\gamma}^{k+1} - \bm{\gamma}^* \rVert_2^2 \nonumber \\
\leq &\lVert \mathbf{x}^k - \mathbf{x}^* \rVert_2^2 + \lVert \mathbf{u}^k - \mathbf{u}^* \rVert_2^2 + \lVert \bm{\lambda}^k - \bm{\lambda}^* \rVert_2^2 + \lVert \bm{\gamma}^k - \bm{\gamma}^* \rVert_2^2 \nonumber \\
- &\epsilon_0 \Big[\lVert \mathbf{y}^{k+1} - \mathbf{x}^{k+1} \rVert_2^2 + \lVert \mathbf{y}^{k+1} - \mathbf{x}^k \rVert_2^2 + \lVert \mathbf{v}^{k+1} - \mathbf{u}^{k+1} \rVert_2^2 + \lVert \mathbf{v}^{k+1} - \mathbf{u}^k \rVert_2^2 \nonumber \\
+ &\lVert \bm{\mu}^{k+1} - \bm{\lambda}^{k+1} \rVert_2^2 + \lVert \bm{\mu}^{k+1} - \bm{\lambda}^k \rVert_2^2 + \lVert \bm{\nu}^{k+1} - \bm{\gamma}^{k+1} \rVert_2^2 + \lVert \bm{\nu}^{k+1} - \bm{\gamma}^k \rVert_2^2\Big], \label{eq: sequence inequality}
\end{align}    
\endgroup
which implies for all $k \geq 0$:
\begingroup
\begin{align}
0 \leq &\lVert \mathbf{x}^{k+1} - \mathbf{x}^* \rVert_2^2 + \lVert \mathbf{u}^{k+1} - \mathbf{u}^* \rVert_2^2 + \lVert \bm{\lambda}^{k+1} - \bm{\lambda}^* \rVert_2^2 + \lVert \bm{\gamma}^{k+1} - \bm{\gamma}^* \rVert_2^2 \nonumber \\
\leq &\lVert \mathbf{x}^k - \mathbf{x}^* \rVert_2^2 + \lVert \mathbf{u}^k - \mathbf{u}^* \rVert_2^2 + \lVert \bm{\lambda}^k - \bm{\lambda}^* \rVert_2^2 + \lVert \bm{\gamma}^k - \bm{\gamma}^* \rVert_2^2 \nonumber \\
\leq &\lVert \mathbf{x}^{k-1} - \mathbf{x}^* \rVert_2^2 + \lVert \mathbf{u}^{k-1} - \mathbf{u}^* \rVert_2^2 + \lVert \bm{\lambda}^{k-1} - \bm{\lambda}^* \rVert_2^2 + \lVert \bm{\gamma}^{k-1} - \bm{\gamma}^* \rVert_2^2 \nonumber \\
\leq &\cdots \leq \lVert \mathbf{x}^0 - \mathbf{x}^* \rVert_2^2 + \lVert \mathbf{u}^0 - \mathbf{u}^* \rVert_2^2 + \lVert \bm{\lambda}^0 - \bm{\lambda}^* \rVert_2^2 + \lVert \bm{\gamma}^0 - \bm{\gamma}^* \rVert_2^2.
\end{align}
\endgroup
It further implies that the sequence $\{\lVert \mathbf{x}^k - \mathbf{x}^* \rVert_2^2 + \lVert \mathbf{u}^k - \mathbf{u}^* \rVert_2^2 + \lVert \bm{\lambda}^k - \bm{\lambda}^* \rVert_2^2 + \lVert \bm{\gamma}^k - \bm{\gamma}^* \rVert_2^2\}$ is monotonically decreasing and bounded below by $0$; hence the sequence must be convergent to a limit, denoted by $\xi$:
\begin{equation}\label{eq: bounded below sequence convergence}
\lim_{k \to +\infty} \lVert \mathbf{x}^k - \mathbf{x}^* \rVert_2^2 + \lVert \mathbf{u}^k - \mathbf{u}^* \rVert_2^2 + \lVert \bm{\lambda}^k - \bm{\lambda}^* \rVert_2^2 + \lVert \bm{\gamma}^k - \bm{\gamma}^* \rVert_2^2 = \xi.
\end{equation}
Taking the limit on both sides of \eqref{eq: sequence inequality} yields:
\begin{equation}\label{eq: 8 sequence convergence}
\begin{aligned}
&\lim_{k \to +\infty} \lVert \mathbf{y}^{k+1} - \mathbf{x}^{k+1} \rVert_2^2 = 0, \quad &\lim_{k \to +\infty} \lVert \mathbf{y}^{k+1} - \mathbf{x}^k \rVert_2^2 = 0, \\
&\lim_{k \to +\infty} \lVert \mathbf{v}^{k+1} - \mathbf{u}^{k+1} \rVert_2^2 = 0, \quad &\lim_{k \to +\infty} \lVert \mathbf{v}^{k+1} - \mathbf{u}^k \rVert_2^2 = 0, \\
&\lim_{k \to +\infty} \lVert \bm{\mu}^{k+1} - \bm{\lambda}^{k+1} \rVert_2^2 = 0, \quad &\lim_{k \to +\infty} \lVert \bm{\mu}^{k+1} - \bm{\lambda}^k \rVert_2^2 = 0, \\
&\lim_{k \to +\infty} \lVert \bm{\nu}^{k+1} - \bm{\gamma}^{k+1} \rVert_2^2 = 0, \quad &\lim_{k \to +\infty} \lVert \bm{\nu}^{k+1} - \bm{\gamma}^k \rVert_2^2 = 0.
\end{aligned}
\end{equation}
Additionally, \eqref{eq: bounded below sequence convergence} also implies that $\{(\mathbf{x}^k, \mathbf{u}^k, \bm{\lambda}^k, \bm{\gamma}^k)\}$ is a bounded sequence, and there exists a sub-sequence $\{(\mathbf{x}^{k_j}, \mathbf{u}^{k_j}, \bm{\lambda}^{k_j}, \bm{\gamma}^{k_j})\}$ that converges to a limit point $(\mathbf{x}^{\infty}, \mathbf{u}^{\infty}, \bm{\lambda}^{\infty}, \bm{\gamma}^{\infty})$. We next show that the limit point is indeed a saddle point and is also the unique limit point of $\{(\mathbf{x}^k, \mathbf{u}^k, \bm{\lambda}^k, \bm{\gamma}^k)\}$. Given any $\mathbf{x} \in \mathbb{X}$ and $\mathbf{u} \in \mathbb{R}^{n_2}$, we have:
\begin{equation}
\begin{aligned}
&2 \rho^{k+1} \big[\mathcal{L}(\mathbf{x}^{k+1}, \mathbf{u}^{k+1}, \bm{\mu}^{k+1}, \bm{\nu}^{k+1}) - \mathcal{L}(\mathbf{x}, \mathbf{u}, \bm{\mu}^{k+1}, \bm{\nu}^{k+1})\big] \\
= &2 \rho^{k+1} \Big\{\big[\frac{1}{2} (\mathbf{x}^{k+1})^T P_0 \mathbf{x}^{k+1} - \frac{1}{2} \mathbf{x}^T P_0 \mathbf{x}\big] + \mathbf{q}_0^T (\mathbf{x}^{k+1} - \mathbf{x}) + \mathbf{c}_0^T (\mathbf{u}^{k+1} - \mathbf{u}) \\
+ &\sum_{i=1}^{m_1} \mu_i^{k+1} \big[\frac{1}{2} (\mathbf{x}^{k+1})^T P_i \mathbf{x}^{k+1} - \frac{1}{2} \mathbf{x}^T P_i \mathbf{x}\big] + \sum_{i=1}^{m_1} \mu_i^{k+1} \mathbf{q}_i^T (\mathbf{x}^{k+1} - \mathbf{x}) \\
+ &\sum_{i=1}^{m_1} \mu_i^{k+1} \mathbf{c}_i^T (\mathbf{u}^{k+1} - \mathbf{u})\\
+ &\bm{\nu}^{k+1} A(\mathbf{x}^{k+1} - \mathbf{x}) + \bm{\nu}^{k+1} B(\mathbf{u}^{k+1} - \mathbf{u})\Big\} \\
= &\underbrace{2 \rho^{k+1}\Big[\hspace*{-3pt}- \frac{1}{2}(\mathbf{x}^{k+1} - \mathbf{x})^T P_0 (\mathbf{x}^{k+1} - \mathbf{x}) - \hspace*{-3pt}\sum_{i=1}^{m_1} \mu_i^{k+1} \frac{1}{2}(\mathbf{x}^{k+1} - \mathbf{x})^T P_i (\mathbf{x}^{k+1} - \mathbf{x})\Big]}_{(\Delta)} \\
+ &2 \rho^{k+1} \Big[(P_0 \mathbf{x}^{k+1} + \mathbf{q}_0)^T (\mathbf{x}^{k+1} - \mathbf{x}) + \mathbf{c}_0^T (\mathbf{u}^{k+1} - \mathbf{u}) \\
+ &\sum_{i=1}^{m_1} \mu_i^{k+1} (P_i \mathbf{x}^{k+1} + \mathbf{q}_i)^T (\mathbf{x}^{k+1} - \mathbf{x}) + \sum_{i=1}^{m_1} \mu_i^{k+1} \mathbf{c}_i^T (\mathbf{u}^{k+1} - \mathbf{u}) \\
+ &\bm{\nu}^{k+1} A(\mathbf{x}^{k+1} - \mathbf{x}) + \bm{\nu}^{k+1} B(\mathbf{u}^{k+1} - \mathbf{u})\Big] \\
\leq &2 \rho^{k+1} \Big[(P_0 \mathbf{x}^{k+1} + \mathbf{q}_0)^T (\mathbf{x}^{k+1} - \mathbf{x}) + \mathbf{c}_0^T (\mathbf{u}^{k+1} - \mathbf{u}) \\
+ &\sum_{i=1}^{m_1} \mu_i^{k+1} (P_i \mathbf{x}^{k+1} + \mathbf{q}_i)^T (\mathbf{x}^{k+1} - \mathbf{x}) + \sum_{i=1}^{m_1} \mu_i^{k+1} \mathbf{c}_i^T (\mathbf{u}^{k+1} - \mathbf{u}) \\
+ &\bm{\nu}^{k+1} A(\mathbf{x}^{k+1} - \mathbf{x}) + \bm{\nu}^{k+1} B(\mathbf{u}^{k+1} - \mathbf{u})\Big].
\end{aligned}
\end{equation}
The positive semi-definiteness of each $P_i$ for all $i = 0, 1, \dots, m$ guarantees the non-positiveness of $(\Delta)$, which makes the last inequality hold. Applying Lemma~\ref{lem: Lemma1} on \eqref{eq: x_k+1} with $\widehat{\mathbf{z}} = (\mathbf{x}^{k+1}, \mathbf{u}^{k+1})$, $\bar{\mathbf{z}} = (\mathbf{x}^k, \mathbf{u}^k)$ and $\mathbf{z} = (\mathbf{x}, \mathbf{u})$ yields:
\begingroup
\begin{align}
&2 \rho^{k+1}\Big\{(P_0 \mathbf{y}^{k+1} + \mathbf{q}_0)^T \mathbf{x}^{k+1} + \mathbf{c}_0^T \mathbf{u}^{k+1} + r_0 \nonumber \\
&\hspace*{30pt} + \sum_{i=1}^{m_1} \mu_i^{k+1} \Big[(P_i \mathbf{y}^{k+1} + \mathbf{q}_i)^T \mathbf{x}^{k+1} + \mathbf{c}_i^T \mathbf{u}^{k+1} + r_i\Big] \nonumber \\
&\hspace*{30pt} + \bm{\nu}^{k+1} (A \mathbf{x}^{k+1} + B \mathbf{u}^{k+1} - \mathbf{b})\Big\} \nonumber \\
- &2 \rho^{k+1}\Big\{(P_0 \mathbf{y}^{k+1} + \mathbf{q}_0)^T \mathbf{x} + \mathbf{c}_0^T \mathbf{u} + r_0 \nonumber \\
&\hspace*{30pt} + \sum_{i=1}^m \mu_i^{k+1} \Big[(P_i \mathbf{y}^{k+1} + \mathbf{q}_i)^T \mathbf{x} + \mathbf{c}_i^T \mathbf{u} + r_i\Big] \nonumber \\
&\hspace*{30pt} + \bm{\nu}^{k+1} (A \mathbf{x} + B \mathbf{u} - \mathbf{b})\Big\} \nonumber \\
\leq &\ \lVert \mathbf{x}^k - \mathbf{x} \rVert_2^2 - \lVert \mathbf{x}^{k+1} - \mathbf{x} \rVert_2^2 - \lVert \mathbf{x}^{k+1} - \mathbf{x}^k \rVert_2^2 \nonumber \\
+ &\ \lVert \mathbf{u}^k - \mathbf{u} \rVert_2^2 - \lVert \mathbf{u}^{k+1} - \mathbf{u} \rVert_2^2 - \lVert \mathbf{u}^{k+1} - \mathbf{u}^k \rVert_2^2 \nonumber \\
\leq &\ \big(\lVert \mathbf{x}^k - \mathbf{x}^{k+1} \rVert_2^2 + \lVert \mathbf{x}^{k+1} - \mathbf{x} \rVert_2^2\big) - \lVert \mathbf{x}^{k+1} - \mathbf{x} \rVert_2^2 - \lVert \mathbf{x}^{k+1} - \mathbf{x}^k \rVert_2^2 \nonumber \\
+ &\ \big(\lVert \mathbf{u}^k - \mathbf{u}^{k+1} \rVert_2^2 + \lVert \mathbf{u}^{k+1} - \mathbf{u} \rVert_2^2\big) - \lVert \mathbf{u}^{k+1} - \mathbf{u} \rVert_2^2 - \lVert \mathbf{u}^{k+1} - \mathbf{u}^k \rVert_2^2 = 0.
\end{align}
\endgroup
Adding the above two inequalities yields
\begingroup
\begin{align}
&2 \rho^{k+1} \Big[\mathcal{L}(\mathbf{x}^{k+1}, \mathbf{u}^{k+1}, \bm{\mu}^{k+1}, \bm{\nu}^{k+1}) - \mathcal{L}(\mathbf{x}, \mathbf{u}, \bm{\mu}^{k+1}, \bm{\nu}^{k+1})\Big] \nonumber \\
+ &2 \rho^{k+1} \Big\{(\mathbf{y}^{k+1} - \mathbf{x}^{k+1})^T P_0 (\mathbf{x}^{k+1} - \mathbf{x}) \nonumber \\
&\hspace*{30pt} + \sum_{i=1}^m \mu_i^{k+1} \big[(\mathbf{y}^{k+1} - \mathbf{x}^{k+1})^T P_i (\mathbf{x}^{k+1} - \mathbf{x})\big]\Big\} \leq 0.
\end{align}
\endgroup
Taking the limits over an appropriate sub-sequence $\{k_j\}$ on both sides and using \eqref{eq: 8 sequence convergence}, we have:
\begin{equation}
\mathcal{L}(\mathbf{x}^{\infty}, \mathbf{u}^{\infty}, \bm{\lambda}^{\infty}, \bm{\gamma}^{\infty}) \leq \mathcal{L}(\mathbf{x}, \mathbf{u}, \bm{\lambda}^{\infty}, \bm{\gamma}^{\infty}), \quad \forall \mathbf{x} \in \mathbb{X}, \forall \mathbf{u} \in \mathbb{R}^{n_2}.
\end{equation}
Similarly, given any $\bm{\lambda} \in \mathbb{R}_+^{m_1}$ and $\bm{\gamma} \in \mathbb{R}^{m_2}$, applying Lemma~\ref{lem: Lemma1} on \eqref{eq: lambda_k+1} with $\widehat{\mathbf{z}} = (\bm{\lambda}^{k+1}, \bm{\gamma}^{k+1})$, $\bar{\mathbf{z}} = (\bm{\lambda}^k, \bm{\gamma}^k)$ and $\mathbf{z} = (\bm{\lambda}, \bm{\gamma})$ yields
\begingroup
\begin{align}
&2 \rho^{k+1}\big[\mathcal{L}(\mathbf{y}^{k+1}, \mathbf{v}^{k+1}, \bm{\lambda}, \bm{\gamma}) - \mathcal{L}(\mathbf{y}^{k+1}, \mathbf{v}^{k+1}, \bm{\lambda}^{k+1}, \bm{\gamma}^{k+1})\big] \nonumber \\
\leq &\ \lVert \bm{\lambda}^k - \bm{\lambda} \rVert_2^2 - \lVert \bm{\lambda}^{k+1} - \bm{\lambda} \rVert_2^2 - \lVert \bm{\lambda}^{k+1} - \bm{\lambda}^k \rVert_2^2 \nonumber \\
+ &\ \lVert \bm{\gamma}^k - \bm{\gamma} \rVert_2^2 - \lVert \bm{\gamma}^{k+1} - \bm{\gamma} \rVert_2^2 - \lVert \bm{\gamma}^{k+1} - \bm{\gamma}^k \rVert_2^2 \nonumber \\
\leq &\ \big(\lVert \bm{\lambda}^k - \bm{\lambda}^{k+1} \rVert_2^2 - \lVert \bm{\lambda}^{k+1} + \bm{\lambda} \rVert_2^2\big) - \lVert \bm{\lambda}^{k+1} - \bm{\lambda} \rVert_2^2 - \lVert \bm{\lambda}^{k+1} - \bm{\lambda}^k \rVert_2^2 \nonumber \\
+ &\ \big(\lVert \bm{\gamma}^k - \bm{\gamma}^{k+1} \rVert_2^2 - \lVert \bm{\gamma}^{k+1} + \bm{\gamma} \rVert_2^2\big) - \lVert \bm{\gamma}^{k+1} - \bm{\gamma} \rVert_2^2 - \lVert \bm{\gamma}^{k+1} - \bm{\gamma}^k \rVert_2^2 = 0.
\end{align}
\endgroup
Taking the limits over an appropriate sub-sequence $\{k_j\}$ on both sides and using \eqref{eq: 8 sequence convergence}, we have:
\begin{equation}
\mathcal{L}(\mathbf{x}^{\infty}, \mathbf{u}^{\infty}, \bm{\lambda}, \bm{\gamma}) \leq \mathcal{L}(\mathbf{x}^{\infty}, \mathbf{u}^{\infty}, \bm{\lambda}^{\infty}, \bm{\gamma}^{\infty}), \quad \forall \bm{\lambda} \in \mathbb{R}_+^{m_1}, \forall \bm{\gamma} \in \mathbb{R}^{m_2}.
\end{equation}
Therefore, we show that $(\mathbf{x}^{\infty}, \mathbf{u}^{\infty}, \bm{\lambda}^{\infty}, \bm{\gamma}^{\infty})$ is indeed a saddle point of the Lagrangian function $\mathcal{L}(\mathbf{x}, \mathbf{u}, \bm{\lambda}, \bm{\gamma})$. Then \eqref{eq: bounded below sequence convergence} implies that
\begin{equation}
\lim_{k \to +\infty} \lVert \mathbf{x}^k - \mathbf{x}^{\infty} \rVert_2^2 + \lVert \mathbf{u}^k - \mathbf{u}^{\infty} \rVert_2^2 + \lVert \bm{\lambda}^k - \bm{\lambda}^{\infty} \rVert_2^2 + \lVert \bm{\gamma}^k - \bm{\gamma}^{\infty} \rVert_2^2 = \xi.
\end{equation}
Since we have argued (after Eq. \eqref{eq: 8 sequence convergence}) that there exists a bounded sequence of $\{(\mathbf{x}^k, \mathbf{u}^k, \bm{\lambda}^k, \bm{\gamma}^k)\}$ that converges to $(\mathbf{x}^{\infty}, \mathbf{u}^{\infty}, \bm{\lambda}^{\infty}, \bm{\gamma}^{\infty})$; that is, there exists $\{k_j\}$ such that $\lim_{k_j \to +\infty} \lVert \mathbf{x}^{k_j} - \mathbf{x}^{\infty} \rVert_2^2 + \lVert \mathbf{u}^{k_j} - \mathbf{u}^{\infty} \rVert_2^2 + \lVert \bm{\lambda}^{k_j} - \bm{\lambda}^{\infty} \rVert_2^2 + \lVert \bm{\gamma}^{k_j} - \bm{\gamma}^{\infty} \rVert_2^2 = 0$, which then  implies that $\xi = 0$. Therefore, we show that $\{(\mathbf{x}^k, \mathbf{u}^k, \bm{\lambda}^k, \bm{\gamma}^k)\}$ converges globally to a saddle point $(\mathbf{x}^{\infty}, \mathbf{u}^{\infty}, \bm{\lambda}^{\infty}, \bm{\gamma}^{\infty})$. 
\par \hfill$\Box$
\end{appendices}
\end{document}